\documentclass[10pt]{article}
\usepackage[english]{babel}
\usepackage[toc,page]{appendix}
\usepackage{bbm}
\usepackage{fullpage}
\usepackage[top=1in, bottom=1.25in, left=0.7in, right=0.7in]{geometry}
\usepackage{etaremune}
\usepackage{enumitem}
\usepackage{graphicx}
\usepackage[labelformat=empty]{caption}
\usepackage{subcaption}

\usepackage{framed}

\usepackage{amssymb,amsmath,amsthm}

\usepackage{hyperref}

\DeclareSymbolFont{rsfs}{U}{rsfs}{m}{n}
\DeclareSymbolFontAlphabet{\mathscrsfs}{rsfs}
\usepackage{mathrsfs}

\usepackage{url}

\hypersetup{
  colorlinks   = true, 
  urlcolor     = blue, 
  linkcolor    = blue, 
  citecolor   = red 
}

\newtheorem{theorem}{Theorem}[section]
\newtheorem{lemma}[theorem]{Lemma}

\newtheorem{proposition}[theorem]{Proposition}
\newtheorem{corollary}[theorem]{Corollary}

\theoremstyle{definition}
\newtheorem{definition}{Definition}
\newtheorem{assumption}{Assumption}
\newtheorem{remark}[theorem]{Remark}

\newtheorem{fact}[theorem]{Fact}

\numberwithin{equation}{section}

\usepackage{times}

\newcommand{\bea}{\begin{eqnarray}}
\newcommand{\eea}{\end{eqnarray}}
\newcommand{\<}{\langle}
\renewcommand{\>}{\rangle}

\newcommand{\wt}{\widetilde}
\newcommand{\op}{\text{op}}
\newcommand{\wh}{\widehat}

\newcommand\eg{{\text{\eg~}}}
\def\iid{\text{i.i.d.~}}

\def\Vol{{\rm Vol}}

\def\Proj{{\sf P}}

\def\oC{{\overline C}}
\def\eps{{\varepsilon}}

\def\bh{\boldsymbol{h}}

\def\supp{{\rm supp}}
\def\spec{{\rm spec}}

\def\ind{{\mathbbm 1}}

\def\btau{{\boldsymbol{\tau}}}

\def\bW{{\boldsymbol{W}}}

\def\bM{{\boldsymbol{M}}}

\def\bV{{\boldsymbol{V}}}

\def\bg{{\boldsymbol{g}}}

\def\bx{{\boldsymbol{x}}}

\def\bzero{{\mathbf 0}}
\def\bone{{\mathbf 1}}

\def\cT{{\mathcal T}}
\def\cC{{\mathcal C}}
\def\cE{{\mathcal E}}

\def\cX{{\mathcal X}}

\def\blambda{{\boldsymbol \lambda}}

\def\sS{{\mathscr S}}

\def\op{{\rm op}}
\def\Band{{\rm Band}}

\def\good{\text{good}}
\def\bad{\text{bad}}

\def\vK{\vec K}

\def\vv{\vec v}
\def\veps{\vec \eps}
\def\vw{\vec w}
\def\vx{\vec x}
\def\vq{\vec q}
\def\vz{\vec z}

\def\vone{\vec 1}
\def\vzero{\vec 0}

\def\bsig{{\boldsymbol {\sigma}}}

\def\vm{{\vec m}}

\def\vu{{\vec u}}

\def\vlambda{{\vec \lambda}}
\def\vlambda{{\vec \lambda}}
\def\vtheta{{\vec \theta}}
\def\va{{\vec a}}
\def\vb{{\vec b}}
\def\ve{{\vec e}}

\def\brho{{\boldsymbol \rho}}

\def\by{{\boldsymbol y}}
\def\vy{{\vec y}}

\def\Cone{{\mathsf{Cone}}}

\def\bv{{\boldsymbol{v}}}
\def\bz{{\boldsymbol{z}}}
\def\bx{{\boldsymbol{x}}}

\def\bs{{\boldsymbol{s}}}

\def\bA{\boldsymbol{A}}
\def\bB{\boldsymbol{B}}

\def\bm{\boldsymbol{m}}

\def\va{\vec{a}}
\def\vb{\vec{b}}

\def\de{{\rm d}}

\def\bX{\boldsymbol{X}}

\def\bW{\boldsymbol{W}}

\def\<{\langle}
\def\>{\rangle}
\def\Tr{{\sf Tr}}

\def\diag{{\rm diag}}

\def\cH{{\cal H}}

\def\cN{{\cal N}}

\def\cL{{\cal L}}

\def\by{{\boldsymbol{y}}}

\def\P{\mathbb{P}}
\def\cE{{\mathcal E}}

\def\tA{{\widetilde A}}
\def\tB{{\widetilde B}}
\def\tD{{\widetilde D}}

\def\blambda{{\boldsymbol{\lambda}}}

\def\cD{{\cal D}}

\def\bu{{\boldsymbol{u}}}
\def\b0{{\boldsymbol{0}}}

\def\rd{{\mathrm {rad}}}

\def\spec{{\mathrm {spec}}}

\def\bG{{\boldsymbol G}}

\DeclareMathOperator*{\plim}{p-lim}

\def\ocD{{\overline {\mathcal D}}}
\def\cD{{{\mathcal D}}}

\def\Ent{{\rm Ent}}

\def\cA{{\mathcal A}}
\def\cI{{\mathcal I}}
\def\cS{{\mathcal S}}

\def\cB{{\mathcal B}}

\def\cK{{\mathcal K}}

\def\star{*}

\def\bzero{\boldsymbol{0}}

\renewcommand{\b}{\mathbf{b}}

\def\fr{\frac}
\def\lt{\left}
\def\rt{\right}

\def\la{\langle}
\def\ra{\rangle}

\def\eps{\varepsilon}
\def\ups{\upsilon}

\def\bbA{{\mathbb{A}}}

\def\bbC{{\mathbb{C}}}
\def\bbE{{\mathbb{E}}}
\def\bbH{{\mathbb{H}}}

\def\bbN{{\mathbb{N}}}
\def\bbP{{\mathbb{P}}}
\def\bbR{{\mathbb{R}}}
\def\bbS{{\mathbb{S}}}

\def\bbW{{\mathbb{W}}}

\def\obbH{{\overline \bbH}}

\def\cA{{\mathcal{A}}}
\def\cB{{\mathcal{B}}}

\def\cJ{{\mathcal{J}}}
\def\cK{{\mathcal{K}}}
\def\cN{{\mathcal{N}}}
\def\cP{{\mathcal{P}}}
\def\cQ{{\mathcal{Q}}}
\def\cR{{\mathcal{R}}}

\def\Crt{{\mathsf{Crt}}}
\def\tot{{\text{tot}}}

\def\sH{{\mathscr{H}}}

\def\one{{\mathbbm{1}}}
\def\bg{{\mathbf{g}}}
\def\bh{{\boldsymbol{h}}}

\def\bM{\mathbf{M}}

\def\ALG{{\mathsf{ALG}}}

\def\nonmax{{\rm nonmax}}

\DeclareMathOperator*{\E}{\bbE}

\def\sph{\mathrm{sp}}

\newcommand{\Gp}[1]{\mathbf{G}^{(#1)}}

\newcommand{\norm}[1]{{\lt\|#1\rt\|}}
\newcommand{\tnorm}[1]{{\|#1\|}}

\newcommand{\Span}{\mathrm{span}}

\def\vchi{\vec{\chi}}

\def\vR{\vec R}

\def\star{*}

\def\hM{\widehat M}
\def\oM{\overline M}
\def\oF{\overline F}

\def\oPsi{\overline \Psi}
\def\omu{\overline \mu}

\def\vDelta{{\vec\Delta}}

\def\scl{\mathsf{scale}}

\author{
    Brice Huang\thanks{Department of Electrical Engineering and Computer Science, Massachusetts Institute of Technology. Email: \texttt{bmhuang@mit.edu}.}
    \and
    Mark Sellke\thanks{Department of Statistics, Harvard University.
    Email: \texttt{msellke@fas.harvard.edu}.}
}

\title{
Strong Topological Trivialization of Multi-Species Spherical Spin Glasses
}

\date{}

\begin{document}

\maketitle

\begin{abstract}
    We study the landscapes of multi-species spherical spin glasses.
    Our results determine the phase boundary for annealed trivialization of the number of critical points, and establish its equivalence with a quenched \emph{strong topological trivialization} property.
    Namely in the ``trivial'' regime, the number of critical points is constant, all are well-conditioned, and all \emph{approximate} critical points are close to a true critical point.
    As a consequence, we deduce that Langevin dynamics at sufficiently low temperature has logarithmic mixing time.

    Our approach begins with the Kac--Rice formula. We characterize the annealed trivialization phase by explicitly solving a suitable multi-dimensional variational problem, obtained by simplifying certain asymptotic determinant formulas from \cite{arous2021exponential,mckenna2021complexity}.
    To obtain more precise quenched results, we develop general purpose techniques to
    avoid sub-exponential correction factors and show non-existence of \emph{approximate} critical points.
    Many of the results are new even in the $1$-species case.
\end{abstract}

\setcounter{tocdepth}{1}
{
\small
\tableofcontents
}

\section{Introduction}

This paper studies the landscapes of certain random, non-convex functions $H_N:\bbR^N\to\bbR$, namely the Hamiltonians of spherical spin glasses.
Mean-field spin glass models were introduced in \cite{sherrington1975solvable} to study disordered magnetic materials, and subsequently studied in many papers including \cite{ruelle1987mathematical,crisanti1992spherical,crisanti1993spherical}.
Of particular note, Parisi predicted the free energy through the phenomenon of replica symmetry breaking in \cite{parisi1979infinite}, which was later proved by \cite{talagrand2006parisi,talagrand2006spherical} following decades of progress.

The spin glasses we focus on will feature $r\ge 1$ species, with the domain of $H_N$ given by a product of $r$ high-dimensional spheres.
Such models include (in the Ising case) the $r=2$ bipartite SK model \cite{kincaid1975phase,korenblit1985spin,fyodorov1987antiferromagnetic,fyodorov1987phase} which has received recent attention due to connections with neural networks \cite{barra2010replica,agliari2012multitasking,hartnett2018replica}.
For $r>1$, a basic understanding of the low-temperature statics is still missing in general: due to a breakdown of the crucial interpolation method \cite{guerra2002thermodynamic}, it is not even known that the limiting ground state energy exists in general despite much recent work \cite{barra2015multi,panchenko2015free,mourrat2021nonconvex,mourrat2020free,baik2020free,bates2021crisanti,subag2021tap1,subag2021tap2,kivimae2021ground}.

We follow the landscape complexity approach pioneered by \cite{fyodorov2004complexity}, studying the set of critical points using techniques such as the Kac--Rice formula.
We focus in particular on \emph{topological trivialization}: the transition of the number of critical points from exponential to constant beyond a critical external field strength.
As further detailed below, precise understanding of this phenomenon faces several challenges.
First, while one expects a dimension-independent number of critical points under a strong external field, the reduced symmetry from multiple species means tools for computing expected counts of critical points are accurate only to leading exponential order.
Second, these leading order terms must be accessed implicitly through the solution to a vector Dyson equation.
Third, even perfect knowledge of the critical points does not suffice to understand \emph{approximate critical points} with \emph{small} gradient (e.g. low temperature Gibbs samples) which might be far from any genuine critical point.
Fourth, the Kac--Rice formula only gives annealed expectations, so the phase boundary it suggests might not correspond to any quenched property.
These challenges will lead us to develop new techniques which enhance the Kac--Rice formula and yield a more complete description of the landscape even in the one-species setting.

\paragraph*{Landscape Complexity}

Our starting point is the Kac--Rice formula introduced in \cite{rice1944mathematical,kac1948average} (see \cite[Chapter 11]{adler2007random} for a textbook treatment).
In general this formula allows one to compute moments for the number of critical points, local optima, and similar quantities for smooth Gaussian processes on manifolds.
It has been employed, usually to obtain annealed counts of critical points, in many settings including spiked tensor models \cite{arous2019landscape,ros2019complex,fan2021tap,celentano2021local,auffinger2022sharp},
non-gradient vector fields \cite{cugliandolo1997glassy,fyodorov2016nonlinear,fyodorov2016topology,garcia2017number,ben2021counting,kivimae2022concentration,subag2023concentration}, polymer models \cite{fyodorov2018exponential}, Euler characteristics \cite{taylor2003euler}, generalized linear models \cite{maillard2020landscape},
and the elastic manifold \cite{arous2021landscape}.
Typically the most complicated term in the Kac--Rice integrand is the expected determinant of a large random matrix.

For spherical spin glasses, the important works \cite{auffinger2013random,auffinger2013complexity} calculated the annealed exponential growth rates for the number of critical points of various indices and energy levels. Matching second moment estimates for pure models were established in \cite{subag2017complexity}, see also \cite{auffinger2020number,subag2021concentration}. These yielded in some cases an elementary proof of the Parisi formula at zero temperature, as well as further geometric results on the Gibbs measures \cite{subag2017extremal,subag2017geometry,arous2020geometry,arous2021shattering}.
Annealed asymptotics for the multi-species setting were obtained in \cite{mckenna2021complexity}, with a matching second moment computation in the pure case by \cite{kivimae2021ground}.

Our primary aim will be to identify the topologically trivial phase where the landscape contains a (dimension-free) constant number of critical points, and to understand it in detail.
This was done for the \emph{annealed} complexity of single-species spin glasses in \cite{fyodorov2014topology,fyodorov2015high,belius2022triviality}, which showed that in the trivial regime the only critical points are the unique global maximum and minimum (with high probability).
In these works and many others mentioned above, the high degree of symmetry is crucial: it ensures the random matrices appearing in Kac--Rice computations are from the Gaussian Orthogonal Ensemble, for which exact determinantal formulas are available.
Recently the work \cite{arous2021exponential} gave broadly applicable tools for random matrix determinants with less symmetry.
Their work enables quite general Kac--Rice computations, with the caveat that the results hold only to leading exponential order (i.e. with an extra $e^{o(N)}$ factor in dimension $N$).

Through the example of multi-species spin glasses, we aimed to study the following three meta-questions on random landscapes which do not seem to have been addressed in the literature. While the first is somewhat tailored to the multi-species setting and the recent work \cite{arous2021exponential}, we are unaware of rigorous results toward the latter two in any of the models above.
\begin{enumerate}[label=(\arabic*)]
    \item
    \label{it:non-invariant}
    For non-symmetric models with Hessians more complicated than GOE, does the topologically trivial regime still exhibit a dimension-free number of critical points? Or is the $e^{o(N)}$ upper bound on annealed complexity the end of the story?
    \item
    \label{it:quenched}
    With or without symmetry, does the phase boundary of annealed topological trivialization have genuine significance? Or can the regime of \emph{quenched} topological trivialization be strictly larger?
    \item
    \label{it:algorithms}
    Does topological trivialization imply rapid convergence for optimization algorithms such as Langevin dynamics?
    Or might regions with small but non-zero gradient lead to arbitrarily slow convergence?
\end{enumerate}
The last question in particular was highlighted in the recent book chapter
\cite{ros2022high}, which ends:
\[
    \parbox{16cm}{\emph{Finally, we find it appropriate to conclude this chapter by recalling that getting
    a refined information on the landscape topology and geometry can hopefully shed
    light and guide us into the comprehension of the dynamical evolution of the complex
    systems associated to it: establishing quantitatively this connection between landscape
    and dynamics is the underlying goal of the landscape program, and thus the most
    relevant perspective.}}
\]

\paragraph*{Our Results}

We make progress on all three of the above questions.
As the first step, we establish in Theorem~\ref{thm:annealed-crits} the annealed phase boundary for topological trivialization in the sense of leading exponential order.
Already these annealed estimates determine the ground state energy in the topologically trivial phase.
Answering Question~\ref{it:non-invariant}, we go further and show throughout the topologically trivial regime that the number of critical points in an $r$-species spin glass is exactly $2^r$, which is the minimum possible for any Morse function on a product of $r$ spheres.
Moreover the landscape trivializes in a quantitative sense: each of the $2^r$ critical points has dimension-free condition number, and all \emph{approximate} critical points (with small gradient) are close to one of them.
We call this confluence of properties \textbf{strong topological trivialization} (see Definition~\ref{def:STT}), and show that for any landscape satisfying it, the mixing time for low temperature Langevin dynamics is $O(\log N)$. This addresses Question~\ref{it:algorithms} above.
Conversely in the topologically non-trivial phase, our companion paper \cite{huang2023optimization} explicitly constructs exponentially many well-separated approximate critical points (see Proposition~\ref{prop:quenched-failure-of-STT} below). This implies \emph{quenched} failure of \emph{strong} topological trivialization whenever the annealed complexity is non-trivial, partially addressing Question~\ref{it:quenched}.

\paragraph*{Proof Techniques}

Our computations using the Kac--Rice formula rely on asymptotics for expectations of random determinants computed in \cite{arous2021exponential} via the vector Dyson equation, in particular Corollary 1.9.A therein.
We determine the expected number of critical points to leading exponential order, in particular identifying the trivial regime of annealed complexity $e^{o(N)}$.

First we discuss several new challenges arising in the Kac--Rice computations as compared to the single-species setting.
Whereas the random determinants arising for one species are of a GOE matrix plus a scalar multiple of the identity, with multiple species the GOE is replaced by a more general Gaussian block matrix.
Before the present work, the exponential-order growth rates of the relevant determinants were only known in the form of an integral $\Psi = \int \log |\gamma| \de \mu(\gamma)$ for a measure $\mu$ whose Stieltjes transform solves a vector Dyson equation (see \eqref{eq:Psi-def}).
While this integral can be explicitly evaluated in the single-species case because $\mu$ is a shift of the semicircle law, in general this representation is far from explicit.
Addressing this challenge, we find a closed-form formula for this integral (Lemma~\ref{lem:oF-oPsi-formulas}), expressed in terms of the solution to the vector Dyson equation.
This formula makes the Kac--Rice calculations reasonably explicit and may be of independent interest.

The annealed complexity is now given by the maximum of an $r$-dimensional complexity functional $F:\bbR^r\to\bbR$ whose main term is this $\Psi$.
Determining the maximum of $F$ is also complicated and requires more than finding its stationary points.
Even for one species, $F$ has both concave and convex regions and is $C^1$ but not $C^2$.
With $r>1$ species, the maximization of $F$ becomes a multi-dimensional optimization problem where the number of concave and nonconcave regions in $\bbR^r$ grows exponentially with $r$; see Subsection~\ref{subsec:proof-discussion} and Figure~\ref{fig:complexity-functional}.
In general, $F$ has about $3^r$ stationary points (see Lemma~\ref{lem:stationary-condition}), $2^r$ of which eventually yield critical points of $H_N$.
To show that the remaining stationary points do not maximize $F$, we construct at each of these points an explicit direction along which the Hessian $\nabla^2 F$ is positive.
%
The construction of this direction requires an understanding of the \emph{solution} to the vector Dyson equation, which ranges over a non-trivial subset of $\bbC^r$ that we characterize (Lemma~\ref{lem:u-manifold}).
%
%
%
To justify the necessary calculations, we also prove (in Theorem~\ref{thm:continuity}) new joint continuity properties of the vector Dyson equation; these extend results of \cite{ajanki2017singularities, alt2020dyson} in the case of finitely many blocks.

Having determined the trivial regime for annealed complexity, we next turn a more precise understanding of this regime, for example aiming to show the number of critical points is exactly $2^r$ with high probability.
It follows from the annealed estimates that all critical points are well-conditioned and must have one of $2^r$ ``types" corresponding to maxima of $F$ (see Definition~\ref{def:type}).
Here the type of a critical point essentially determines its overlap with the external field (which is an $r$-dimensional vector), as well as its energy and Hessian spectrum.
This motivates the following natural strategy. For each of the $2^r$ types we restrict attention to a lower-dimensional band having the correct overlap with the external field (thus containing all critical points of that type), and search for critical points of the restriction of $H_N$ to this band by repeating this process.
The conditional law of $H_N$ on such a band is again a spherical multi-species spin glass in the topologically trivial regime, and the relevant band shrinks in diameter each step. Thus we would hope to eventually show that all critical points of each fixed type are close together. Since the Kac--Rice estimates imply all critical points are well-conditioned, there can only be at most $1$ inside any small region. Hence this would imply a $2^r$ upper bound for the number of critical points.

Unfortunately this approach does not make sense on its face because critical points are \emph{brittle}.
In particular the set of critical points of $H_N$ restricted to a lower dimensional band might be unrelated to the set of critical points on an \emph{open neighborhood} of said band.
To overcome this difficulty, we establish in Theorem~\ref{thm:approx-crits-from-annealed} a way to pass from annealed upper bounds for exact critical points to high-probablity non-existence of \emph{approximate} critical points (with small gradient).
Because the notion of approximate critical point is more robust, the shrinking bands argument above can be salvaged, thus proving the $2^r$ upper bound.
The fact that all $2^r$ critical points actually exist then follows by the Morse inequalities from differential topology.
The aforementioned non-existence of approximate critical points far from any exact critical point also falls out of the shrinking bands argument.

We note that the ability to control approximate critical points via Kac--Rice estimates seems quite powerful and should have further applications to random landscapes.
In addition to the shrinking bands argument above, we also needed Theorem~\ref{thm:approx-crits-from-annealed} to show strict positivity of the annealed complexity in the complementary ``non-trivial'' regime, see Subsection~\ref{subsec:annealed-complexity-positive}.
Further, in Section~\ref{sec:approx-local-max-E-infty} we use these ideas to derive energy estimates for ``approximate local maxima'' in single-species spherical spin glasses without external field, via the annealed thresholds $E_{\infty}^{\pm}$ of \cite{auffinger2013complexity}.
We have already applied these estimates to obtain bounds on
the energy attained by low temperature Langevin dynamics \cite{sellke2023threshold} and the algorithmic threshold energy for Lipschitz optimization \cite{huang2023algorithmic,huang2023optimization}.

\paragraph*{Connections to Algorithms}

As mentioned previously, our landscape results yield non-asymptotic algorithmic consequences in the ``trivial'' regime.
We show in Theorem~\ref{thm:langevin-abstract} that low temperature Langevin dynamics rapidly enters a small neighborhood of the global maximum and remains there for an exponentially long time, even from disorder-dependent initialization.
The proof relies on recent work by one of us \cite{sellke2023threshold} to ensure the dynamics does not get stuck in saddle points.
Thanks to the local concavity of $H_N$ around its global maximum, we also deduce in this theorem that low temperature Langevin dynamics undergoes total variation mixing within $O(\log N)$ time.
These results follow in a black-box way from the strong topological trivialization property discussed above.

Being an optimization algorithm, Langevin dynamics can find only the global maximum or minimum. However an equally natural ``critical point following'' algorithm explained in Subsection~\ref{subsec:index} suffices to find all $2^r$ critical points. Here one first locates the critical point of the desired ``type'' under an amplified external field, and then follows its movement as the external field strength is gradually decreased. Well-conditioning of critical points ensures that this movement is stable and easy to follow for any model in the ``trivial'' regime.
This can be seen as a variant of ``state following'' \cite{barrat1997temperature,zdeborova2010generalization,sun2012following}; see also \cite[Proposition 9.1]{arous2020geometry} and \cite{song2019surfing} for similar ideas.

Finally, the phase boundary for topological trivialization coincides with a transition in the structure of algorithmically reachable states we recently identified in \cite{huang2023algorithmic,huang2023optimization}. These works study the optimization of $H_N$ using algorithms, viewed as functions of the disorder coefficients, with Lipschitz constant independent of $N$.
Roughly speaking, it is shown that the reachable points for the best such optimization algorithms have the structure of a continuously branching ultrametric tree, and both approximate message passing and a second-order ascent algorithm generalizing that of \cite{subag2018following} (and using a correlated ensemble of Hamiltonians) find these points.
The algorithmic tree is rooted at a random location correlated with the external field (which is simply the origin when the external field vanishes), and branches orthogonally outward until reaching the boundary of the state space.
When the external field is large enough, the algorithmic tree degenerates; the root moves all the way to the boundary of the state space and no branching occurs.
We show in this paper that said degeneracy coincides with topological trivialization.
In the ``nontrivial" regime, the algorithmic tree is non-degenerate and \cite{huang2023optimization} uses the above algorithms to construct $e^{cN}$ well-separated approximate critical points, yielding the quenched non-trivialization discussed previously.
Conversely in the ``trivial" regime, \cite{huang2023optimization} gives a signed generalization of the root-finding approximate message passing iteration which locates all $2^r$ critical points by implementing the previously mentioned recursive-bands argument as an algorithm.

\subsection{Model Description}
\label{subsec:setup}

Fix a finite set $\sS = \{1,\ldots,r\}$ and weights $\vlambda = (\lambda_1,\ldots,\lambda_r) \in\bbR_{>0}^{\sS}$ with $\sum_{s\in\sS} \lambda_s=1$.
For each positive integer $N$, fix a deterministic partition $\{1,\ldots,N\} = \bigsqcup_{s\in\sS}\, \cI_s$ with
$N_s/N=\lambda_{N,s}$ and $\lim_{N\to\infty} \lambda_{N,s} =\lambda_s$ for $N_s=|\cI_s|$.
For $s\in \sS$ and $\bx \in \bbR^N$, let $\bx_s \in \bbR^{\cI_s}$ denote the restriction of $\bx$ to coordinates $\cI_s$.
We consider the product-of-spheres state space
\begin{equation}
\label{eq:def-cS}
    \cS_N = \lt\{
        \bx \in \bbR^N :
        \norm{\bx_s}_2^2 = \lambda_s N
        \quad\forall~s\in \sS
    \rt\}.
\end{equation}
For each $k\ge 1$ fix a symmetric tensor
\[
  \Gamma^{(k)} = (\gamma_{s_1,\ldots,s_k})_{s_1,\ldots,s_k\in \sS} \in (\bbR_{\ge 0}^r)^{\otimes k}
\]
with $\sum_{k \ge 1} 2^k \norm{\Gamma^{(k)}}_\infty < \infty$, and let $\Gp{k} \in (\bbR^N)^{\otimes k}$ be a tensor with \iid standard Gaussian entries.
For $A\in (\bbR^\sS)^{\otimes k}$, $B\in (\bbR^N)^{\otimes k}$, define $A\diamond B \in (\bbR^N)^{\otimes k}$ to be the tensor with entries
\begin{equation}
    \label{eq:def-diamond}
    (A\diamond B)_{i_1,\ldots,i_k} = A_{s(i_1),\ldots,s(i_k)} B_{i_1,\ldots,i_k},
\end{equation}
where $s(i)$ denotes the $s\in \sS$ such that $i\in \cI_s$.
We consider the mean-field multi-species spin glass Hamiltonian
\begin{align}
    \notag
    H_N(\bsig)
    &=
    \sum_{k\ge 1}
    \fr{1}{N^{(k-1)/2}}
    \la \Gamma^{(k)} \diamond \bG^{(k)}, \bsig^{\otimes k} \ra \\
\label{eq:def-hamiltonian}
    &=
  \sum_{k \ge 1}
  \fr{1}{N^{(k-1)/2}}
  \sum_{i_1,\ldots,i_k=1}^N
  \gamma_{s(i_1),\ldots,s(i_k)} \bG^{(k)}_{i_1,\ldots,i_k} \sigma_{i_1}\cdots \sigma_{i_k}
\end{align}
with inputs $\bsig = (\sigma_1,\ldots,\sigma_N) \in \cS_N$.
For $\bsig,\brho\in \cS_N$, define the species $s$ overlap and overlap vector
\begin{equation}
    \label{eq:R}
    R_s(\bsig, \brho)
    =
    \fr{ \la \bsig_s, \brho_s \ra}{\lambda_s N},
    \qquad
    \vR(\bsig, \brho)
    =
    \big(R_1(\bsig, \brho), \ldots, R_r(\bsig, \brho)\big).
\end{equation}
Let $\odot$ denote coordinate-wise product.
For $\vx = (x_1,\ldots,x_r) \in \bbR^\sS$, let
\begin{align*}
    \xi(\vx)
    &= \sum_{k \ge 1} \la \Gamma^{(k)}\odot \Gamma^{(k)}, (\vlambda \odot \vx)^{\otimes k}\ra \\
    &= \sum_{k \ge 1}
  \sum_{s_1\ldots,s_k\in \sS}
  \gamma_{s_1,\ldots,s_k}^2
  (\lambda_{s_1} x_{s_1})
  \cdots
  (\lambda_{s_k} x_{s_k}).
\end{align*}
The random function $H_N$ can also be described as the Gaussian process on $\bbR^N$ with covariance
\[
  \bbE H_N(\bsig)H_N(\brho)
  =
  N\xi(\vR(\bsig, \brho)).
\]
It will be useful to define, for $s\in \sS$,
\begin{align}
\label{eq:def-xi-s}
    \xi^s(\vx)
    &=
    \lambda_s^{-1}
    \partial_{x_s}
    \xi(\vx),
    \\
    \xi'
    &= \nabla \xi(\vone) \in \bbR^r,
    \\
    \xi''
    &=
    \nabla^2 \xi(\vone) \in \bbR^{r\times r}.
\end{align}
We will often write $\diag(\xi')$ for the $r\times r$ matrix with $(s,s)$ entry $\xi'_s$, and similarly for other vectors. Finally, most of our results require the following generic non-degeneracy condition for $\xi$.

\begin{assumption}
    \label{as:nondegenerate}
    $\xi$ is \textbf{non-degenerate} if
    $\Gamma^{(1)},\Gamma^{(2)},\Gamma^{(3)} > 0$ holds entry-wise.
    For fixed $\vlambda$, a family of mixture functions $\xi$ is \textbf{uniformly non-degenerate} if the sums $\sum_{k\geq 1}2^k \|\Gamma^{(k)}\|_{\infty}$ are uniformly bounded above, and for some $\eps>0$, we have $\Gamma^{(1)},\Gamma^{(2)},\Gamma^{(3)} \geq \eps$ entry-wise for all $\xi$ in the family.
\end{assumption}

\subsection{Basic Notations and Conventions}

Here we detail some notations that will be useful to understand the statements in the next subsection.

\begin{definition}
    \label{dfn:Wasserstein}
    For probability measures $\mu,\nu$ on a metric space $(\cX,d)$, and $p\in [1,\infty]$, we denote by $\bbW_p(\mu,\nu)$ the Wasserstein distance
    \[
    \bbW_p(\mu,\nu) =  \lt(\inf_{\Pi \in \cC(\mu,\nu)}
    \bbE_{\Pi} \Big[ d(\bx, \by)^p\Big]\rt)^{1/p},
    \]
    where the infimum is over all couplings $(\bx,\by) \sim \Pi$ with marginals $\bx \sim \mu$ and $\by \sim \nu$. (For $p=\infty$, the distance is the essential supremum of $d(\bx,\by)$ under the coupling.) Unless otherwise specified, $(\cX,d)$ will always be $\bbR^n$ for some $n\geq 1$ with the standard Euclidean metric.
\end{definition}

\begin{definition}
\label{dfn:Hausdorff}
The Hausdorff distance between sets $S_1,S_2\subseteq \bbR$ is given by
\begin{equation}
\label{eq:hausdorff-distance}
d_{\cH}(S_1,S_2)\equiv \max\lt(\max_{s_1\in S_1}d(s_1,S_2),\max_{s_2\in S_2}d(S_1,s_2)\rt).
\end{equation}
Here $d(s_1,S_2)=\inf_{s_2\in S_2} d(s_1,s_2)$ is the usual point-to-set distance.
\end{definition}

Given a symmetric $n\times n$ matrix $M$, we denote by $\blambda_{\min}(M),\blambda_{\max}(M)$ its minimum and maximum eigenvalue, and by $\blambda_k(M)$ its $k$-th largest eigenvalue. Using $\cP(\bbR)$ to denote the space of probability measures on $\bbR$, denote by
\begin{equation}
\label{eq:spectral-defs}
    \spec(M)=\{\blambda_k(M)\,:\,k\in[n]\}\subseteq\bbR,
    \quad\quad\quad\quad\quad
    \wh\mu(M)=\frac{1}{n}\sum_{k=1}^n \delta_{\blambda_k(M)}\in\cP(\bbR),
\end{equation}
the empirical spectral support and measure of $M$. Let
\begin{equation}
\label{eq:Hessian-spectral-defs}
    \spec_{H_N}(\bx)=
    \spec\big(\nabla_{\sph}^2 H_N(\bx)\big),
    \quad\quad\quad\quad\quad
    \wh\mu_{H_N}(\bx)=\wh\mu\big(\nabla_{\sph}^2 H_N(\bx)\big)
\end{equation}
be the corresponding objects for the Riemannian Hessian defined just below.
We will always use non-bolded $\vlambda = (\lambda_s)_{s\in \sS}$ to denote the species weights as in Subsection~\ref{subsec:setup}.

Next we define the \emph{radial derivative} and \emph{Riemannian gradient and Hessian} $H_N$.
Throughout the paper we assume $\cI_1 = \{1,\ldots,m_1\}$, $\cI_2 = \{m_1+1,\ldots,m_2\}$, and so on.
Let $\cR = \{m_1,\ldots,m_r\}$ and $\cT = [N] \setminus \cR$.
For each $\bsig \in \cS_N$, we pick an orthonormal basis $\big\{e_1(\bsig),\ldots,e_N(\bsig)\big\}$ of $\bbR^N$ so that $\{e_i(\bsig) : i \in \cI_s\}$ constitutes an orthonormal basis of $\bbR^{\cI_s}$, and $\bsig_s = \sqrt{\lambda_s N} e_{m_s}(\bsig)$.
For $S\subseteq [N]$, let $\nabla_S H_N(\bsig) \in\bbR^S$ denote the restriction of $\nabla H_N(\bsig) \in\bbR^N$ to the coordinates in $S$ (in the orthonormal basis $\{e_1(\bsig),\ldots,e_N(\bsig)\}$), and define $\nabla^2_{S\times S} H_N \in \bbR^{S\times S}$ analogously.
The radial derivative is $\nabla_{\cR} H_N(\bsig)$; it will be convenient to define a rescaled radial derivative $\nabla_{\rd} H_N(\bsig)=N^{-1/2}\nabla_{\cR} H_N(\bsig)$ so that the below formulas become dimension-free.
Then, define the matrices
\begin{equation}
\label{eq:Lambda-and-A}
    \Lambda = \diag(\vlambda)\in\bbR^{r\times r},\quad\quad
    A = \diag(\xi') + \xi'' \in \bbR^{r\times r}\,.
\end{equation}
The following standard fact relates the Riemannian gradient and Hessian of $H_N$ to the Euclidean gradient and Hessian and can be taken as a definition.
\begin{fact}
    \label{fac:riemannian-to-euclidean}
    Let $\nabla_{\sph} H_N(\bsig)$, $\nabla_{\sph}^2 H_N(\bsig)$ denote the Riemannian gradient and Hessian of $H_N$ in $\cS_N$. Then,
    \begin{align*}
        \nabla_{\sph} H_N(\bsig) &= \nabla_\cT H_N(\bsig)\,, &
        \nabla^2_{\sph} H_N(\bsig) &= \nabla^2_{\cT \times \cT} H_N(\bsig)
        - \diag(\Lambda^{-1/2} \nabla_{\rd} H_N(\bsig) \diamond \bone_{\cT})\,.
    \end{align*}
    Explicitly, the curvature term $\diag(\Lambda^{-1/2} \nabla_{\rd} H_N(\bsig) \diamond \bone_{\cT})$ is a diagonal matrix $D \in \bbR^{\cT \times \cT}$ where for all $i \in \cT \cap \cI_s$,
    \[
        D_{i,i}
        = \fr{1}{\sqrt{\lambda_s}} (\nabla_\rd H_N(\bsig))_s
        = \fr{1}{\sqrt{\lambda_s N}} \partial_{m_s} H_N(\bsig).
    \]
\end{fact}

We now define approximate critical points and ground states. We remark that all $\eps$-approximate ground states are $\delta$-approximate critical points for some $\delta(\eps)$ tending to $0$ with $\eps$, assuming $H_N$ lies in the exponentially high probability set $K_N$ defined in Proposition~\ref{prop:gradients-bounded}.

\begin{definition}
\label{def:approx}
    A point $\bx\in\cS_N$ is an \textbf{$\eps$-approximate critical point} if $\|\nabla_{\sph}H_N(\bx)\|_2 \leq \eps\sqrt{N}$, and an \textbf{$\eps$-approximate ground state} if $H_N(\bx)+\eps N\geq \max_{\bsig\in\cS_N} H_N(\bsig)$.
    We will sometimes abbreviate these by \textbf{$\eps$-critical point} and \textbf{$\eps$-ground state}.
    Moreover, $\bx$ is \textbf{$C$-well conditioned} if $\nabla_{\sph}^2 H_N(\bx)$ has all eigenvalues in $\pm[C^{-1},C]$.
    Finally, $\bx$ is an \textbf{$(\eps,C)$-well conditioned critical point} if it is both an $\eps$-critical point and $C$-well conditioned.
\end{definition}

Finally, we will often say that an event holds with probability $1-e^{-cN}$. In these cases, unless specified otherwise, $c$ is a small constant which may depend on all other relevant $N$-independent constants.

\subsection{Main Results on Strong Topological Trivialization}

In this subsection we state our main results.
Theorem~\ref{thm:annealed-crits} below shows that super-solvable models have $e^{o(N)}$ critical points and identifies their possible asymptotic energies, correlation with $\bG^{(1)}$, and Hessian spectrum.
Theorem~\ref{thm:precise-landscape} considerably refines this statement for strictly super-solvable models, establishing \emph{strong} topological trivialization: the number of critical points equals $2^r$, all are well-conditioned, and all approximate critical points are close to one of them.

\begin{definition}
\label{def:xi'}
    Assume $\xi$ is non-degenerate. We say $\xi$ is \textbf{super-solvable} if $\diag(\xi') \succeq \xi''$, where $\succeq$ denotes the Loewner (positive semi-definite) partial order, \textbf{strictly super-solvable} if  $\diag(\xi') \succ \xi''$.
    Similarly we say $\xi$ is \textbf{strictly sub-solvable} if $\diag(\xi') \not\succeq \xi''$.
    We say $\xi$ is \textbf{solvable} if $\diag(\xi') - \xi'' \succeq 0$ is singular.
\end{definition}

\begin{remark}
    The condition $\diag(\xi') \succeq \xi''$ coincides with the condition for degeneracy of the algorithmic ultrametric trees identified in \cite[Theorem 3]{huang2023algorithmic}.
    In one species, this condition recovers the annealed trivialization condition determined by \cite[Theorem 1.1]{belius2022triviality} and coincides with the condition for zero-temperature replica symmetry \cite[Proposition 1]{chen2017parisi}.
\end{remark}

\begin{remark}
    The super-solvability condition above also appeared naturally in \cite{huang2023algorithmic} via the analysis of an ODE describing the algorithmic threshold for optimizing multi-species spherical spin glass Hamiltonians.
    In that paper, super-solvability is a property of points in $\bbR_{\ge 0}^r$ for fixed $\xi$, and this ODE has different behaviors on super-solvable and sub-solvable regions of $\bbR_{\ge 0}^r$.
    In the present work, $\xi$ is super-solvable if, in the language of \cite{huang2023algorithmic}, $\vone$ is super-solvable for $\xi$.
\end{remark}

\begin{remark}
\label{rem:external-field}
    We note that in both \cite{huang2023algorithmic,huang2023optimization}, the linear external field term in $H_N$ is a deterministic vector $\bh$ rather than the random $\bG^{(1)}$; both papers immediately apply to the model \eqref{eq:def-hamiltonian} by conditioning on $\bG^{(1)}$.
    On the other hand, the upper bounds we obtain on annealed complexity become slightly stronger when the external field is random.
\end{remark}

The $2^r$ critical points of a strictly super-solvable model will correspond naturally to sign patterns $\vDelta \in \{-1,1\}^r$, which determine whether the critical point is positively or negatively correlated with the external field in each species.
For each $\vDelta$, we define its associated energy, overlap, and radial derivative:
\begin{equation}
    \label{eq:ideal-stats}
    \begin{aligned}
    E(\vDelta) &= \sum_{s\in \sS} \Delta_s \sqrt{\lambda_s \xi'_s}\,, \\
    \vR(\vDelta) &= \lt(\fr{\Delta_s\gamma_s}{\sqrt{\xi'_s}}\rt)_{s\in \sS}\,, \\
    \vx(\vDelta) &= \lt(
        \Delta_s \sqrt{\xi'_s} +
        \sum_{s'\in \sS} \Delta_{s'}
        \sqrt{\fr{\lambda_{s'}}{\lambda_s}} \cdot
        \fr{\xi''_{s,s'}}{\sqrt{\xi'_{s'}}}
    \rt)_{s\in \sS}.
    \end{aligned}
\end{equation}
Below we refer to certain probability measures $\mu(\vx(\vDelta))\in\cP(\bbR)$, which are the limiting spectral measures for certain random block matrices, defined using the vector Dyson equation in \eqref{eq:mu-def}.
We also refer to their supports $S(\vDelta)=\supp\big(\mu(\vx(\vDelta))\big)\subseteq \bbR$ which are finite unions of intervals (see \eqref{eq:S-vDelta-def}).
We will use Definitions~\ref{dfn:Wasserstein} and \ref{dfn:Hausdorff} as well as $\nabla_\sph^2$ and $\nabla_\rd$ as defined in Fact~\ref{fac:riemannian-to-euclidean}.

\begin{definition}
    \label{def:good}
    Let $\eps > 0$ and $\vDelta \in \{-1,1\}^r$.
    A point $\bx \in \cS_N$ is \textbf{$(\eps,\vDelta)$-good} if:
    \begin{align}
        \label{eq:energy-good}
        \lt|\fr1N H_N(\bx)-E(\vDelta)\rt| &\le \eps, \\
        \label{eq:1spin-correlation-good}
        \norm{\vR(\bG^{(1)},\bx)-\vR(\vDelta)}_\infty &\le \eps  \\
        \label{eq:radial-derivative-good}
        \norm{\nabla_{\rd}H_N(\bx)-\vx(\vDelta)}_\infty & \le \eps \\
        \label{eq:hessian-good}
        \bbW_2\lt(\wh\mu_{H_N}(\bx),\mu(\vx(\vDelta))\rt)
        &\leq \eps
        \quad\text{and}\quad
        d_{\cH}\lt(\spec(\nabla_\sph^2 H_N(\bx)),S(\vDelta)\rt)
        \leq
        \eps\,.
    \end{align}
    A point is \textbf{$\eps$-good} if it is $(\eps,\vDelta)$-good for some $\vDelta \in \{-1,1\}^r$.
    Let $\cQ(\eps,\vDelta), \cQ(\eps) \subseteq \cS_N$ be the sets of $(\eps,\vDelta)$-good and $\eps$-good points, respectively.
\end{definition}


Define $\Crt_N^{\tot}=\Crt_N^{\tot}(H_N)$ to be the set of critical points of $H_N$, and $\Crt_N^{\good,\eps} = \Crt_N^{\tot} \cap \cQ(\eps)$ and $\Crt_N^{\bad,\eps} = \Crt_N^{\tot} \setminus \cQ(\eps)$.
The following theorem shows that solvability defines the phase boundary for annealed topological trivialization.
It further shows that in strictly super-solvable models, all critical points are $\eps$-good and thus belong to one of $2^r$ types corresponding to $\vDelta \in \{-1,1\}^r$, defined by Definition~\ref{def:good}.

\begin{theorem}
    \label{thm:annealed-crits}
    \begin{enumerate}[label=(\alph*), leftmargin=*]
        \item \label{itm:crt-tot} If $\xi$ is super-solvable, then
        \begin{equation}
            \label{eq:crt-tot}
            \lim_{N\to\infty}\fr{1}{N} \log \bbE|\Crt_N^{\tot}| = 0.
        \end{equation}
    \end{enumerate}\vspace*{-2\partopsep}
    \begin{enumerate}[resume, label=(\alph*)]
        \item \label{itm:crt-bad} If $\xi$ is strictly super-solvable, for all $\eps>0$, there exists $c=c(\xi,\eps) > 0$ such that
        \begin{equation}
            \label{eq:crt-bad}
            \limsup_{N\to\infty}\fr{1}{N} \log \bbE|\Crt_N^{\bad,\eps}| \le -c.
        \end{equation}
        \item \label{itm:crt-tot-sub-solvable} On the other hand, if $\xi$ is strictly sub-solvable, then
        \begin{equation}
            \label{eq:crt-tot-sub-solvable}
            \lim_{N\to\infty}\fr{1}{N} \log \bbE|\Crt_N^{\tot}| > 0.
        \end{equation}
    \end{enumerate}
\end{theorem}

Our next result states that $H_N$ has \emph{exactly} $2^r$ critical points with high probability when $\xi$ is strictly super-solvable: one for each type $\vDelta\in\{-1,1\}^r$.
Moreover all approximate critical points are near a true critical point, and (as a consequence) all approximate ground states are near the true ground state (recall Definition~\ref{def:approx}).
We find it helpful to abstract some of these results into the following definition, which could easily be extended to other manifolds besides $\cS_N$. (The first requirement, smoothness at the natural scale, holds with high probability by Proposition~\ref{prop:gradients-bounded}.)

\begin{definition}
\label{def:STT}
    We say the function $H_N:\cS_N\to\bbR$ is \textbf{$(C,\eps,\iota)$-strongly topologically trivial} if:
    \begin{enumerate}[label=(\roman*)]
        \item
        \label{it:STT-smooth}
        $\|\nabla^k H_N(\bsig)\|_{\op}\leq CN^{1-\frac{k}{2}}$ for $k\in \{0,1,2,3\}$.
        \item
        \label{it:STT-complexity}
        $|\Crt_N^{\tot}|\leq C$.
        \item
        \label{it:STT-well-conditioned}
        All critical points of $H_N$ are $C$-well-conditioned.
        \item
        \label{it:STT-unique-approx-max}
        All critical points $\bx$ of $H_N$ besides the unique global maximum satisfy $\blambda_{N/C}(\nabla_{\sph}^2 H_N(\bx))\geq 1/C$.
        \item
        \label{it:STT-approx-crits}
        All $\eps$-approximate critical points of $H_N$ are within distance $\iota\sqrt{N}$ of a critical point.
    \end{enumerate}
    We say the sequence of functions $(H_N)_{N\geq 1}$ is \textbf{$C$-strongly topologically trivial} if for any $\iota>0$, for $\eps>0$ sufficiently small, all but finitely many are $(C,\eps,\iota)$-strongly topologically trivial.
    We say the sequence is \textbf{strongly topologically trivial} if the previous condition holds for some finite $C$.
\end{definition}

We note that from conditions~\ref{it:STT-smooth} and \ref{it:STT-well-conditioned}, it follows that if $\bx$ is a critical point for $H_N$ and $\|\wt\bx-\bx\|\geq C^{-4}\sqrt{N}$, then $\wt\bx$ is not a $C^{-10}$-approximate critical point.
Hence if condition~\ref{it:STT-approx-crits} holds for $\iota$ which is small depending on $C$, then one can actually take $\iota=O(\eps)$.
The next main result shows that if $\xi$ is strictly super-solvable then the sequence $(H_N)_{N\geq 1}$ is almost surely topologically trivial. In fact, we precisely describe the $2^r$ critical points of $H_N$.

\begin{theorem}
    \label{thm:precise-landscape}
    If $\xi$ is strictly super-solvable, then the following holds with probability $1-e^{-cN}$.
    For sufficiently small $\eps > 0$, $H_N$ has exactly one critical point $\bx_{\vDelta}$ satisfying \eqref{eq:energy-good} through \eqref{eq:hessian-good} for each $\vDelta\in \{-1,1\}^r$.
    Moreover for all $\iota>0$ there exists $\eps>0$ such that:
    \begin{enumerate}[label=(\alph*)]
        \item
        \label{it:bsig-in-union}
        All $\eps$-critical points of $H_N$ are $C(\vlambda,\xi)$-well conditioned and lie in the disjoint union
        \[
            \bigcup_{\vDelta\in \{-1,1\}^r}
            B_{\iota\sqrt{N}}(\bx_{\vDelta}).
        \]
        \item
        \label{it:index}
        The number of positive eigenvalues of $\nabla_{\sph}^2 H_N(\bx_{\vDelta})$ is exactly $\sum_{s\in\sS} N_s\cdot 1_{\Delta_s=-1}$.
        \item
        \label{it:ground-states}
        All $\eps$-ground states lie in $B_{\iota\sqrt{N}}(\bx_{\vone})$.
    \end{enumerate}
\end{theorem}

On the other hand, strong topological trivialization becomes false for strictly sub-solvable $\xi$ (with probability $1-e^{-cN}$).
This exhibits a natural \emph{quenched} phase transition coinciding with the annealed transition of $\bbE|\Crt_N^{\tot}|$ in Theorem~\ref{thm:annealed-crits}.
Indeed our companion paper explicitly constructs exponentially many $\sqrt{N}/C$-separated approximate critical points whenever $\xi$ is strictly sub-solvable, which contradicts parts \ref{it:STT-complexity}, \ref{it:STT-approx-crits} of Definition~\ref{def:STT} when $\iota\leq 1/(2C)$. The precise result is quoted below.
(Note however that we give no quenched lower bounds on the number of \emph{exact} critical points when $\xi$ is strictly sub-solvable, which would be interesting to obtain.)

\begin{proposition}[{\cite[Proposition 3.6]{huang2023optimization}}]
\label{prop:quenched-failure-of-STT}
    For any strictly sub-solvable $\xi$ there exists $C(\vlambda,\xi)>0$ such that for any $\eps>0$ there exists $\delta>0$ such that with probability $1-e^{-cN}$ the following holds.
    There exist $M=e^{\delta N}$ distinct $\eps$-approximate critical points $\bx_1,\dots,\bx_M\in\cS_N$ such that $\|\bx_i-\bx_j\|_2 \geq \sqrt{N}/C$ for all $1\leq i<j\leq M$.
\end{proposition}

\subsection{Consequences for Langevin Dynamics}
\label{subsec:langevin}

Here we obtain dynamical consequences from our landscape results, showing that for any strongly topologically trivial landscape, low temperature Langevin dynamics mixes in $O(\log N)$ time.
The main ingredient is a recent result by one of us \cite{sellke2023threshold} showing that, roughly speaking, low temperature Langevin dynamics can get stuck only in approximate local maxima.

\begin{definition}
\label{def:langevin}
    Given a Hamiltonian $H_N:\cS_N\to\bbR$, initial condition $\bX(0)\in\cS_N$, and $\beta\geq 0$, the $\beta$-Langevin dynamics $\bX(t)$ driven by standard $\bbR^N$-valued Brownian motion $\bB(t)$ is the process solving the stochastic differential equation
    \begin{equation}
    \de \bX(t)
    =
    \lt(\beta\nabla_{\sph} H_N(\bX(t))
    -
    \sum_{s\in \sS} \frac{N_s-1}{2\lambda_s N} \bX_s(t) \rt)
    \de t
    +
    P_{\bX(t)}^{\perp}\sqrt{2}~\de \bB(t).
    \end{equation}
    Here $P_{\bX(t)}^{\perp}$ is the rank $N-r$ projection matrix onto the orthogonal complement of $\mathrm{span}\big(\bX_1(t),\dots,\bX_r(t)\big)$.
\end{definition}

\begin{theorem}
\label{thm:langevin-abstract}
    Fix $\vlambda$ and let $(H_N)_{N\geq 1}$ be a $C$-strongly topologically trivial sequence of functions $H_N:\cS_N\to\bbR$ with unique local maximum $\bx_*=\bx_{*,N}$. (E.g. $H_N$ as above, with $\bx_*=\bx(\vone).)$ Then:
    \begin{enumerate}[label=(\alph*)]
    \item
    \label{it:langevin-abstract-1}
    For any $\eps>0$, if $\beta\geq\beta_0(\vlambda,\eps,C)$ and $T\geq T_0(\vlambda,\eps,C)$ and $N$ are sufficiently large, $\beta$-Langevin dynamics started from any $\bX(0)\in\cS_N$ satisfies with probability $1-e^{-cN}$:
    \begin{align}
    \label{eq:langevin-energy}
    \inf_{t\in [T,T+e^{cN}]}
    H_N(\bX(t))
    &\geq
    H_N(\bx_*)-\eps N,
    \\
    \label{eq:langevin-near-crit}
    \inf_{t\in [T,T+e^{cN}]} \norm{\bX(t)-\bx_{*}}_2
    &\leq
    \eps \sqrt{N}.
    \end{align}
    \item
    \label{it:langevin-abstract-2}
    For $\beta\geq \beta_0(\vlambda,C)$, the $\beta$-Langevin dynamics has total variation mixing time $O(\log N)$.
    \end{enumerate}
\end{theorem}

\begin{proof}
    \textbf{Part~\ref{it:langevin-abstract-1}:}
    We use \cite[Theorem 1.2]{sellke2023threshold}, the proof of which easily extends to finite products of spheres as considered here (see Remark 2.8 therein, and note that Proposition~\ref{prop:gradients-bounded} below ensures the needed $C$-boundedness).
    The implication is that to prove \eqref{eq:langevin-energy} it suffices to show that for $\eta\leq \eta_0(\xi,\vlambda,\eps)\leq \eps/2$ small enough (playing the role of $\eps$ therein), $E_*^{(\eta)}$ as defined in \cite[Definition 2]{sellke2023threshold} satisfies
    \[
    E_*^{(\eta)} \geq \frac{H_N(\bx_*)}{N}-\frac{\eps}{2}.
    \]
    By said definition, this holds if no $\bsig\in\cS_N$ satisfies all of the following properties:
    \begin{enumerate}[label=(\arabic*)]
        \item \label{itm:approx-crit} $\|\nabla_{\sph}H_N(\bsig)\|_2 \leq \eta\sqrt{N}$.
        \item \label{itm:hessian-mostly-negative} The Hessian $\nabla^2_{\sph} H_N(\bsig)$ satisfies $\blambda_{\lfloor \eta N\rfloor}\big(\nabla^2_{\sph} H_N(\bsig)\big)\leq \eta$.
        \item \label{itm:energy-opt} $H_N(\bsig)\leq H_N(\bx_*)-\frac{\eps N}{2}$.
    \end{enumerate}

    By the definition of strong topological trivialization, for small $\eta$ condition \ref{itm:approx-crit} implies that $\bsig \in B_{\iota\sqrt{N}}(\bx)$ for some critical point $\bx$, where $\iota\to 0$ as $\eta\to 0$.
    Condition \ref{itm:hessian-mostly-negative} implies $\bx=\bx_*$: otherwise  Definition~\ref{def:STT}\ref{it:STT-unique-approx-max} implies $\blambda_{N/C}(\nabla^2_\sph H_N(\bx)) \ge 1/C$, so Definition~\ref{def:STT}\ref{it:STT-smooth} with $k=3$ implies (for small enough $\iota$) that $\blambda_{N/C}(\nabla^2_\sph H_N(\bx)) \ge 1/(2C)$, which is a contradiction for $\eta_0 < 1/(2C)$.
    Thus $\bsig \in B_{\iota\sqrt{N}}(\bx_*)$, which (for small $\eta$) contradicts condition \ref{itm:energy-opt}.

    We conclude that \eqref{eq:langevin-energy} holds, and \eqref{eq:langevin-near-crit} follows (with a different choice of $\eps$) thanks to Theorem~\ref{thm:precise-landscape}.

    \textbf{Part~\ref{it:langevin-abstract-2}:}
    Let us choose $\eps$ small enough that $\blambda_{\max}\big(\nabla_{\sph}^2 H_N(\bx)\big)\leq 0$ is negative semi-definite on $B_{\eps\sqrt{N}}(\bx_*)$; this is possible thanks to parts~\ref{it:STT-smooth} (with $k=3$) and \ref{it:STT-well-conditioned} of Definition~\ref{def:STT}. Also choose $\beta$ large enough that Part~\ref{it:langevin-abstract-1} of this theorem applies with this choice of $\eps$.

    Let $\cC\subseteq \cS_N$ be the product of diameter $\eps\sqrt{N}/r$ spherical caps inside each $(\cS_{N,s})_{1\leq s\leq r}$ centered at $\bx_*$. Each of the $r$ factors is a convex Riemannian manifold with boundary (see e.g. \cite{kronwith1979convex}), hence so is $\cC$. We consider the reflected Langevin dynamics with inward normal reflection in $\cC$, as constructed in e.g. \cite[Section 2.1]{cheng2017reflecting}. Let $P_t$ be the transition kernel for ordinary Langevin dyamics, and $\wt P_t$ that of the reflected dynamics.

    Recall that $\cS_N$ has uniformly positive Ricci curvature (as $N$ varies) and $\nabla_{\sph}^2 H_N(\bx)$ is negative semidefinite on all of $\cC$. It follows that the Gibbs measure $\de\nu_{\beta}(\bx)=Z_{N,\beta}^{-1} e^{\beta H_N(\bx)}\de\bx$ also has uniformly positive Ricci curvature within $\cC$ (see e.g. \cite[Proposition 22]{gheissari2019spectral}).
    Let $\wt\nu_{\beta}$ be the Gibbs measure $\nu_{\beta}$ of $H_N$ conditioned to lie in $\cC$.
    By \cite[Theorem 3.3.2]{wang2014analysis} and convexity of the manifold $\cC$, it follows that the reflected Langevin dynamics inside $\cC$ contracts exponentially in (Riemannian) Wasserstein distance: for any probability measures $\rho_0,\rho_0'$ on $\cC$,
    \begin{equation}
    \label{eq:reflected-W2-contraction}
    \bbW_2(\wt P_t \rho_0,\wt P_t\rho_0')\leq e^{-ct} \bbW_2(\rho_0,\rho_0')\leq e^{-ct}\sqrt{N},
    \quad\quad
    c=c(\vlambda,\xi)>0
    \end{equation}
    Finally, we combine \eqref{eq:reflected-W2-contraction} with \cite[Lemma 4.2]{bobkov2001hypercontractivity} for a small time $\delta$ (denoted $T$ therein) to find (using Pinsker's inequality in the first step)
    \begin{equation}
    \label{eq:reflected-entropy-small}
    \|\wt P_{t+\delta}\rho_0-\wt\nu_{\beta}\|_{TV}^2
    \leq
    \Ent\big(\wt P_{t+\delta} \rho_0||\wt \nu_{\beta}\big)
    \leq
    CNe^{-2ct}.
    \end{equation}
    (\cite[Lemma 4.2]{bobkov2001hypercontractivity} is stated for Euclidean space, but all proof ingredients remain available by \cite[Theorem 3.3.2]{wang2014analysis}.)

    On the other hand, it follows from Part~\ref{it:langevin-abstract-1} that from any initial $\bX(0)\in\cS_N$, with probability $1-e^{-cN}$, we have $\bX(0)\in\cC$ for $T\leq t\leq e^{cN}$.
    Taking the convention that $\wt P_t \delta_{\by}=\delta_{\by}$ for $\by\in\cS_N\backslash \cC$, we write for any $\bx\in\cS_N$:
    \begin{align*}
    \|P_{2t}\delta_{\bx}-\nu_{\beta}\|_{TV}
    &\leq
    \|P_{2t}\delta_{\bx}-\wt P_t (P_t \delta_{\bx})\|_{TV}
    +
    \|\wt P_t (P_t \delta_{\bx})-\wt\nu_{\beta}\|_{TV}
    +
    \|\wt\nu_{\beta}-\nu_{\beta}\|_{TV}
    \\
    &\leq
    e^{-cN}
    +
    CNe^{-ct}
    +
    e^{-cN}.
    \end{align*}
    Here the bound on $\|P_{2t}\delta_{\bx}-\wt P_t (P_t \delta_{\bx})\|_{TV}$ follows from the definition of $\cC$, which ensures that with probability $e^{-cN}$, the ordinary and reflected Langevin dynamics (with shared Brownian motion) agree for $t$ units of time when started from $\bX(t)$.
    The bound on $\|\wt\nu_{\beta}-\nu_{\beta}\|_{TV}$ follows from Part~\ref{it:langevin-abstract-1}.

    In particular, for $t/(\log N)$ at least a large constant, we deduce that for $N$ large enough,
    \[
    \sup_{\bx\in\cS_N}\|P_{2t}\delta_{\bx}-\nu_{\beta}\|_{TV}\leq 1/4
    \]
    which concludes the proof.
\end{proof}

\begin{remark}
\label{rem:poincare}
    The above proof established that any $C$-strongly topologically trivial sequence $H_N$ satisfies the Bakry-Emery conditions on a small neighborhood of the global optimum for all $\beta\geq 0$.
    This implies exponential concentration of overlaps (via concentration of Lipschitz functions). Namely for any $\eps>0$, with $\bsig,\wt\bsig\stackrel{i.i.d}{\sim}\nu_{\beta}$ for $\beta$ large, there exists $\vq_*\in [0,1]^r$ depending on $H_N$ such that
    \[
    \bbP[\|\vR(\bsig,\wt\bsig)-\vq_*\|_{\infty}\leq \eps]\geq 1-e^{-cN}.
    \]
    For spin glass Hamiltonians distributed according to some strictly super-solvable $\xi$, the value $\vq_*$ can be chosen deterministically depending only on $(\vlambda,\xi,\beta)$. This is because the restricted free energies
    \[
    \frac{1}{N}
    \log
    \iint
    \limits_{\substack{\bsig,\wt\bsig\in \cS_N,\\
    \|\vR(\bsig,\wt\bsig)-\vq\|_\infty \leq \eps}}
    \exp(\beta(H_N(\bsig)+H_N(\wt\bsig))
    ~\de\bsig\de\wt\bsig
    .
    \]
    concentrate exponentially for any $\vq$ and $\eps>0$.

    To take $\vq$ independent of $N$ one may fix a large $T>0$, and deterministic $\bX(0)\in\cS_N$ for each $N$, and combine the following two observations. First $\bbE\bbW_2(P_T \delta_{\bX(0)},\nu_{\beta})\leq Ce^{-cT}\sqrt{N}+e^{-cN}$ since \eqref{eq:reflected-W2-contraction} holds with exponentially good probability.
    Second, one may show via two-replica multi-species analogs of the Cugliandolo--Kurchan equations \cite{crisanti1993spherical,cugliandolo1994out,ben2006cugliandolo,dembo2007limiting} that $\vR(\bX(T),\wt\bX(T))$ concentrates exponentially around an $N$-independent value $\vq_*(T)$ when $\bX(T),\wt\bX(T)$ are driven by independent Brownian motions for the same $H_N$.
    Indeed from \cite[Lemma 3.1 and Subsection 4.2]{sellke2023threshold} it suffices to prove these Cugliandolo--Kurchan equations for soft spherical Langevin dynamics, which follows mechanically from the approach of \cite{celentano2021high}.
    From these observations, taking $T\to\infty$ after $N\to\infty$, this implies that overlaps concentrate around $\vq_*$.
\end{remark}

\begin{remark}
    We argued above that the following holds for strictly super-solvable $\xi$ with probability $1-e^{-cN}$: the expected hitting time of a radius $\delta\sqrt{N}$ neighborhood $B_{\delta\sqrt{N}}(\bx(\vone))$ of the global optimum of $H_N$ is at most $C(\vlambda,\xi,\delta)$ for $\beta$ sufficiently large, uniformly in $\bX(0)$.
    Combining these two facts, the Lyapunov function technique of \cite{bakry2008simple} implies that the Gibbs measure $\nu_{\beta}$ has Poincar{\'e} constant at most $C(\vlambda,\xi)$ for $\beta$ sufficiently large. The relevant Lyapunov function $L:\cS_N\to\bbR_{\geq 1}$ is essentially an exponential moment of the hitting time of $B_{\delta\sqrt{N}}(\bx(\vone))$; see \cite[Proposition 9.13 and Appendix B.4]{li2020riemannian} for a detailed derivation of this implication.
\end{remark}

\section{Further Preliminaries}
\label{sec:prelim}

Below we provide further notations and background.
Subsection~\ref{subsec:prelim-geo} will be assumed throughout the entire paper.
Subsections~\ref{subsec:prelim-linalg} and \ref{subsec:prelim-rmt} will be used primarily in Section~\ref{sec:variational}.

\subsection{Geometry of $\cS_N$}
\label{subsec:prelim-geo}

\begin{definition}
    \label{dfn:species-aligned}
    A linear subspace $U \subseteq \bbR^N$ is \textbf{species-aligned} if it is the direct sum of subspaces $U_s \subseteq\bbR^{\cI_s}$, for $s\in \sS$.
\end{definition}
For $\bz\in\bbR^N$ or a species-aligned subspace $U\subseteq\bbR^N$, we define
\begin{align}
\label{eq:multi-perp}
    \bz^{\perp}
    &=
    \lt\{
        \bx\in \bbR^N ~:~ \vR(\bz,\bx)=\vzero
    \rt\}\,, &
    U^{\perp}
    &=
    \lt\{
        \bx\in \bbR^N ~:~ \vR(\bu,\bx)=\vzero~\forall \bu\in U
    \rt\}\,.
\end{align}

Recalling the definitions in Fact~\ref{fac:riemannian-to-euclidean}, we now explicitly describe the law of the local behavior of $H_N$ around a given $\bsig\in\cS_N$.

\begin{lemma}
    \label{lem:derivative-laws}
    Fix $\bsig \in \cS_N$.
    The random variables $\nabla_\cT H_N(\bsig)$, $\nabla^2_{\cT \times \cT} H_N(\bsig)$, and $(H_N(\bsig), \vR(\bG^{(1)},\bsig), \nabla_\rd H_N(\bsig))$ are mutually independent Gaussians with the following distributions.
    \begin{enumerate}[label=(\alph*)]
        \item
        \label{it:tangential-derivative-law}
        Tangential derivative: for each $i\in\cT$, $\partial_i H_N(\bsig) \sim \cN(0, \xi^{s(i)})$ and these are independent across $i$.
        \item Tangential Hessian: $\bW = \nabla^2_{\cT \times \cT} H_N(\bsig)$ is a symmetric random matrix with independent centered Gaussian entries on and above the diagonal, where
        \begin{equation}
            \label{eq:tangential-hessian-law}
            \E[W_{i,j}^2] = \fr{(1+\delta_{i,j})\xi''_{s(i),s(j)}}{N\lambda_{s(i)}\lambda_{s(j)}}\,.
        \end{equation}
        \item Energy, 1-spin overlap, and radial derivative: $(H_N(\bsig), \vR(\bG^{(1)},\bsig), \nabla_\rd H_N(\bsig))$ is a centered Gaussian vector with covariance satisfying
        \begin{align}
            \label{eq:energy-law}
            \E \lt[H_N(\bsig)^2\rt] &= N\xi(\vone)\,, \\
            \label{eq:1spin-law}
            \E \lt[\vR(\bG^{(1)},\bsig)\vR(\bG^{(1)},\bsig)^\top\rt] &= N^{-1} \Lambda^{-1}\,, \\
            \label{eq:radial-derivative-law}
            \E \lt[
                \nabla_{\rd} H_N(\bsig) \nabla_{\rd} H_N(\bsig)^\top
            \rt] &= N^{-1} \Lambda^{-1/2} A \Lambda^{-1/2}\,, \\
            \E \lt[H_N(\bsig) \nabla_{\rd}H_N(\bsig)\rt] &= \Lambda^{-1/2} \xi'\,, \\
            \label{eq:1spin-radial-derivative-correlation}
            \E \lt[
                \vR(\bG^{(1)},\bsig) \nabla_{\rd} H_N(\bsig)^\top
            \rt] &= N^{-1} \diag(\Gamma^{(1)}) \Lambda^{-1} \,.
        \end{align}
        As a consequence,
        \begin{align}
            \label{eq:conditional-energy-mean}
            \E\lt[H_N(\bsig) | \nabla_{\rd}H_N(\bsig)\rt]
            &= N(\xi')^\top A^{-1} \Lambda^{1/2} \nabla_{\rd}H_N(\bsig)\,, \\
            \label{eq:conditional-1spin-mean}
            \E\lt[\vR(\bG^{(1)},\bsig) | \nabla_{\rd}H_N(\bsig)\rt]
            &= \diag(\Gamma^{(1)}) \Lambda^{-1/2} A^{-1} \Lambda^{1/2} \nabla_{\rd}H_N(\bsig)\,, \\
            \label{eq:conditional-energy-variance}
            \mathrm{Var} \lt[H_N(\bsig) | \nabla_{\rd}H_N(\bsig)\rt]
            &= N \lt(\xi(\vone) - (\xi')^\top A^{-1} \xi'\rt)\,.
        \end{align}
    \end{enumerate}
\end{lemma}

\begin{proof}
    Due to the symmetry of the sphere it suffices to verify these statements for $\bsig$ equal to the ``$r$-tuple north pole," i.e. $\sigma_{m_s} = \sqrt{\lambda_s N}$ for all $m_s \in \cR$, and $\sigma_i=0$ for all $i\in \cT$.
    Then $\nabla_\cT H_N(\bsig)$, $\nabla^2_{\cT \times \cT} H_N(\bsig)$, and $(H_N(\bsig), \vR(\bG^{(1)},\bsig), \nabla_{\rd} H_N(\bsig))$ can be evaluated as explicit linear combinations of disorder coefficients, which readily implies the stated covariance structure.
    In particular, they are mutually independent because the sets of disorder coefficients contributing to them are disjoint.
    As an example calculation (see also \cite[Section 3.3]{mckenna2021complexity}),
    \[
        \partial_{m_s} H_N(\bsig)
        = \lambda_s^{-1/2}
        \sum_{p\geq 1}
        \sum_{s_1,\ldots,s_p}
        n_s(s_1,\ldots,s_p)
        \gamma_{s_1,\ldots,s_p}
        \bG^{(p)}_{m_{s_1},\ldots,m_{s_p}}
        \sqrt{\lambda_{s_1}\lambda_{s_2}\cdots \lambda_{s_p}}\,,
    \]
    where $n_s(s_1,\ldots,s_p)$ is the number of times $s$ appears in $s_1,\ldots,s_p$.
    This readily implies that
    \begin{align*}
        \E\lt[\partial_{m_s} H_N(\bsig)^2\rt]
        &= \lambda_s^{-1}
        \sum_{p\geq 1}
        \sum_{s_1,\ldots,s_p}
        n_s(s_1,\ldots,s_p)^2
        \gamma_{s_1,\ldots,s_p}^2
        = \lambda_s^{-1} (\xi'_s + \xi''_{s,s})\,, \\
        \E\lt[\partial_{m_s} H_N(\bsig)\partial_{m_{s'}} H_N(\bsig)\rt]
        &= \lambda_s^{-1/2}\lambda_{s'}^{-1/2}
        \sum_{p\geq 1}
        \sum_{s_1,\ldots,s_p}
        n_s(s_1,\ldots,s_p)n_{s'}(s_1,\ldots,s_p)
        \gamma_{s_1,\ldots,s_p}^2 \\
        &= \lambda_s^{-1/2}\lambda_{s'}^{-1/2} \xi''_{s,s'}\,,
    \end{align*}
    which implies \eqref{eq:radial-derivative-law}.
    The rest of \eqref{eq:energy-law} through \eqref{eq:1spin-radial-derivative-correlation} are verified similarly.
    The formula \eqref{eq:conditional-energy-mean} is verified from the standard fact
    \begin{align*}
        &\E\lt[H_N(\bsig) \,|\, \nabla_\rd H_N(\bsig)\rt] \\
        &=
        \E\lt[H_N(\bsig) \nabla_\rd H_N(\bsig)\rt]^\top
        \E\lt[\nabla_\rd H_N(\bsig) \nabla_\rd H_N(\bsig)^\top\rt]^{-1}
        \nabla_\rd H_N(\bsig)\,,
    \end{align*}
    and \eqref{eq:conditional-1spin-mean} follows similarly.
    Finally \eqref{eq:conditional-energy-variance} follows from
    \[
        \mathrm{Var} \lt[H_N(\bsig) | \nabla_{\rd}H_N(\bsig)\rt]
        = \E\lt[H_N(\bsig)^2\rt] - \E \lt[
            \E\lt[H_N(\bsig) \,|\, \nabla_\rd H_N(\bsig)\rt]^2
        \rt]\,.
    \qedhere
    \]
\end{proof}

\begin{fact}
    \label{fac:volume}
    The volume of $\cS_N$ w.r.t. the $(N-r)$-dimensional Hausdorff measure $\cH^{N-r}$ satisfies
    \[
        \fr1N \log \cH^{N-r}(\cS_N) = \fr{1+\log (2\pi)}{2} + o_N(1)\,.
    \]
\end{fact}

\begin{proof}
    By Stirling's approximation, the volume of $\sqrt{N} \bbS^{N-1}$ is
    \[
        \fr{2\pi^{N/2}N^{(N-1)/2}}{\Gamma(N/2)}
        = e^{o(N)} \fr{(\pi N)^{N/2}}{(N/2e)^{N/2}}
        = e^{o(N)} (2\pi e)^{N/2}\,.
    \]
    Thus the volume of $\cS_N$ is
    \[
        \Vol(\cS_N)
        = e^{o(N)} \prod_{s\in \sS} (2\pi e)^{\lambda_s N/2}
        = e^{o(N)} (2\pi e)^{N/2}\,. \qedhere
    \]
\end{proof}

Let $\sH_N$ denote the space of possible Hamiltonians $H_N$, which we identify as (infinite-dimensional) vectors consisting of their disorder coefficients $(\bG^{(p)})_{p\ge 1}$ concatenated in an arbitrary but fixed order.
Also let $S_N = \{\bx \in \bbR^N : \tnorm{\bx}_2^2 = N\}$ and for any tensor $\bA \in (\bbR^N)^{\otimes k}$, define the operator norm
\[
    \tnorm{\bA}_\op =
    \max_{\|\bsig^1\|_2,\ldots,\|\bsig^k\|_2\leq 1}
    |\la \bA, \bsig^1 \otimes \cdots \otimes \bsig^k \ra|\,.
\]
\begin{proposition}[{\cite[Proposition 1.13]{huang2023algorithmic}}]
    \label{prop:gradients-bounded}
    For any $\xi$ there exists $c>0$, a sequence $(K_N)_{N\geq 1}$ of symmetric convex sets $K_N\subseteq \sH_N$, and constant $C=C(\xi)$, such that the following holds.
    \begin{enumerate}[label=(\alph*)]
        \item
        \label{it:KN-high-prob}
        $\P[H_N\in K_N]\geq 1-e^{-cN}$;
        \item For all $H_N\in K_N$, $k\leq 3$, and $\bx, \by \in \cS_N$,
        \begin{align}
            \label{eq:gradient-bounded}
            \norm{\nabla^k H_N(\bx)}_{\op}
            &\le
            CN^{1-\frac{k}{2}}, \\
            \norm{\nabla^k H_N(\bx) - \nabla^k H_N(\by)}_{\op}
            &\le
            CN^{\frac{1-k}{2}} \norm{\bx-\by}_2.
        \end{align}
    \end{enumerate}
\end{proposition}

\begin{proposition}
\label{prop:spectral-perturbation}
    For symmetric matrices $M,M' \in \bbR^{r\times r}$ we have (recall \eqref{eq:hausdorff-distance}):
    \[
        d_{\cH}\big(\spec(M),\spec(M')\big)
        \leq \bbW_{\infty}\big(\wh\mu(M),\wh\mu(M')\big)
        \leq
        \|M-M'\|_{\op}.
    \]
    In particular for $H_N\in K_N$ and all $\bx,\by\in\cS_N$:
    \[
        d_{\cH}\lt(
        \spec(\nabla^2_{\sph} H_N(\bx)),\spec(\nabla^2_{\sph} H_N(\by))
        \rt)
        \leq \fr{C}{\sqrt{N}} \|\bx-\by\|_2.
    \]
\end{proposition}

\begin{proof}
    The first part is immediate from the Weyl inequalities.
    For the second part,
    \begin{align*}
        &\norm{\nabla^2_{\sph} H_N(\bx) - \nabla^2_{\sph} H_N(\by)}_\op \\
        &\le
        \norm{\nabla^2_{\cT \times \cT} H_N(\bx) - \nabla^2_{\cT \times \cT} H_N(\by)}_\op \\
        &\qquad + \norm{\diag(\Lambda^{-1/2} (\nabla_{\rd} H_N(\bx) - \nabla_{\rd} H_N(\by)) \diamond \bone_\cT)}_\op \\
        &\le
        \norm{\nabla^2 H_N(\bx) - \nabla^2 H_N(\by)}_\op
        + \fr{1}{\sqrt{N \min \vlambda}} \norm{\diag(\nabla_{\cR} H_N(\bx) - \nabla_{\cR} H_N(\by))}_\op\,.
    \end{align*}
    The final term is bounded by
    \begin{align*}
        \norm{\diag(\nabla_{\cR} H_N(\bx) - \nabla_{\cR} H_N(\by))}_\op
        &\le \|\nabla H_N(\bx) - \nabla H_N(\by)\|_2 \\
        &= \|\nabla H_N(\bx) - \nabla H_N(\by)\|_\op.
    \end{align*}
    The result now follows from the first part and Proposition~\ref{prop:gradients-bounded}.
\end{proof}

\subsection{Elementary Linear Algebra}
\label{subsec:prelim-linalg}

\begin{definition}
    A symmetric matrix $M \in \bbR^{r\times r}$ is \textbf{diagonally signed} if $M_{i,i} \ge 0$ and $M_{i,j}<0$ for all distinct $i, j\in [r]$.
\end{definition}

\begin{lemma}
    \label{lem:diagonally-signed-hs23}
    If $M \in \bbR^{r\times r}$ is diagonally signed, then the minimal eigenvalue $\blambda_{\min}(M)$ has multiplicity $1$, and the corresponding eigenvector $\vw$ has strictly positive entries.
    Moreover,
    \[
        \blambda_{\min}(M)
        = \sup_{\vv \succ \vzero}
        \min_{s\in\sS}
        \fr{(M\vv)_s}{v_s}\,.
    \]
\end{lemma}
\begin{proof}
    \cite[Proposition 4.3]{huang2023algorithmic} shows the final equality, and the proof therein shows that any minimal eigenvector $\vw$ of $M$ must have strictly positive entries.
    Since $M$ is symmetric its eigenvectors are orthogonal, so $\vw$ is unique.
\end{proof}

\begin{lemma}
    \label{lem:psd-sign-flip}
    If $M\in \bbR^{r\times r}$ is diagonally signed, $M \succeq 0$, and $M' \in \bbR^{r\times r}$ is defined by $M'_{i,j} = |M_{i,j}|$, then $M' \succeq 0$.
\end{lemma}
\begin{proof}
    By Lemma~\ref{lem:diagonally-signed-hs23}, the minimal eigenvector $\vw$ of $M$ has strictly positive entries.
    Let $\blambda_{\min}(M) = t \ge 0$.
    The equation $M\vw = t\vw$ implies that for any $s\in \sS$,
    \[
        (M_{s,s}-t) w_s + \sum_{s'\neq s} M_{s,s'} w_{s'} = 0
        \quad\implies\quad
        M_{s,s} = t + \sum_{s'\neq s} |M_{s,s'}|\fr{w_{s'}}{w_s}\,.
    \]
    Thus for any $\vx \in \bbR^r$,
    \[
        \la \vx, M'\vx \ra
        = t \norm{\vx}_2^2 +
        \sum_{s\neq s'} |M_{s,s'}|
        \lt(
            \sqrt{\fr{w_{s'}}{w_{s}}}x_{s} +
            \sqrt{\fr{w_{s}}{w_{s'}}}x_{s'}
        \rt)^2 \ge 0\,.
    \qedhere
    \]
\end{proof}

\begin{corollary}
    \label{cor:A-pd}
    If $\xi$ is super-solvable, then (recall \eqref{eq:Lambda-and-A}) $A \succ 0$.
\end{corollary}
\begin{proof}
    Applying Lemma~\ref{lem:psd-sign-flip} to $M = \diag(\xi') - \xi'' \succeq 0$ shows $M' = A - 2\diag(\xi'') \succeq 0$, where $\diag(\xi'') \in \bbR^{r\times r}$ denotes the diagonal matrix with the same diagonal entries as $\xi''$.
    By Assumption~\ref{as:nondegenerate}, $A \succ 0$.
\end{proof}


\subsection{Random Matrix Theory}
\label{subsec:prelim-rmt}

Our calculations will involve standard notions from random matrix theory. For a probability measure $\mu\in\cP(\bbR)$, and with $\bbH$ the open complex upper half-space, its Stiejtles transform $m:\bbH\to\bbC$ is the holomorphic function
\begin{equation}
    \label{eq:stieltjes}
    m(z) = \int \fr{\mu(\de \gamma)}{\gamma-z},
    \qquad z\in \bbH.
\end{equation}
If $\mu$ is compactly supported with piecewise smooth density $\rho(x)$, it is well known (see e.g. \cite[Chapter 2.4]{Guionnet}) that $m$ extends continuously to $\bbR$ at all points of smoothness with $m(x) = \Im(\pi \rho(x))$.
Here and throughout, we use $\Re(\cdot)$ and $\Im(\cdot)$ respectively to denote the real and imaginary parts of a complex scalar or vector.
Throughout this paper it will also be understood that $m(x) = \lim_{z\in \bbH, z\to x} m(z)$.

Next, fixing $\vx \in \bbR^r$, let $m_s(z) = m_s(z;\vx) \in \bbH$ solve the vector Dyson equation
\begin{equation}
    \label{eq:dyson-equation}
    1 + \lt(
        z + \fr{x_s}{\sqrt{\lambda_s}} +
        \sum_{s' \in \sS} \fr{\xi''_{s,s'}}{\lambda_s} m_{s'}(z)
    \rt) m_s(z) = 0\,, \qquad z \in \bbH\,.
\end{equation}
For each $s$, let $\mu_s$ be such that $m_s$ is the Stieltjes transform of $\mu_s$, existence and uniqueness of which is guaranteed by Proposition~\ref{prop:MDE-basic} below. Then define
\begin{align}
\label{eq:mu-def}
    \mu &= \sum_{s\in \sS} \lambda_s \mu_s,
    \\
\label{eq:m-def}
    m(z) &= m(z;\vx) = \sum_{s\in \sS} \lambda_s m_s(z).
\end{align}
We will sometimes write $\mu = \mu_{\xi,\vlambda}(\vx)$ to emphasize the dependence on $\xi,\vlambda,\vx$ (or include some arguments but not others). The next proposition details useful properties of $m_s$ and $\mu_s$.
Note that $\mu_{\xi,\vlambda}(\vx)$ depends only on $(\xi'',\vlambda,\vx)$.
In particular this is a finite-dimensional vector (while $\xi$ is in principle infinite-dimensional).
We also let $\vlambda_N^{\circ} = (\lambda_{N,s}^{\circ})_{s\in \sS} \in \bbR^r$, where
\begin{equation}
\label{eq:lambda-N-s}
\lambda_{N,s}^{\circ}
= \frac{N_s-1}{N-r} = \frac{|\cI_s|-1}{N-r}.
\end{equation}
These slightly modified values of $\vlambda_N$ will be useful because they are the exact relative sizes of the species blocks in $M_N$ (see also \cite[Eq. (2.1)]{mckenna2021complexity}).
Of course $\vlambda_N^{\circ}\to\vlambda$ as $N\to\infty$ since we assume $\vlambda_N\to\vlambda$.

\begin{proposition}
    \label{prop:MDE-basic}
    For each $\vx \in \bbR^r$, there exists a unique solution $(m_1,\ldots,m_r)$ to \eqref{eq:dyson-equation} consisting of holomorphic functions $m_s:\bbH\to\bbH$, each given by the Stieltjes transform of some $\mu_s\in\cP(\bbR)$.
    Moreover for any compact set $\cK\subseteq (0,1)^r\times  (0,\infty)^{r\times r}\times \bbR^r$ there exists $C=C(\cK)$ such that the following hold whenever $(\vlambda_N^{\circ},\xi'',\vx)\in\cK$ (and $\sum_s \lambda_{N,s}^{\circ}=1$).
    \begin{enumerate}[label=(\alph*)]
        \item
        \label{it:Dyson-compact-supports}
        The support sets $\supp(\mu_s)\subseteq\bbR$ are contained in $[-C,C]$ and do not depend on $s$.
        \item
        \label{it:Dyson-continuous-density}
        Each $\mu_s$ is absolutely continuous, with density $\rho_s$ having $1/3$-H{\"o}lder norm at most $C$, piece-wise smooth on at most $C$ intervals with disjoint interiors, and otherwise zero.
        \item
        \label{it:Dyson-continuous-transform}
        Each $m_s(\,\cdot\,;\cdot)$ extends to a jointly continuous function $\obbH\times\bbR^r\mapsto \obbH$ solving \eqref{eq:dyson-equation}, with $1/3$-H{\"older} norm at most $C$.
        \item
        \label{it:Dyson-not-zero}
        $|m_s(z;\vx)|\geq 1/C$ for all $z\in\obbH$ and $\vx\in\bbR^r$.
        \item
        \label{it:Dyson-solves-RMT}
        For any $\eps>0$ and any fixed $\bx\in\cS_N$, conditionally on $\nabla_{\rd} H_N(\bx)=\vx$, we have the bulk-typicality (recall \eqref{eq:Hessian-spectral-defs}):
        \begin{equation}
        \label{eq:bulk-typicality}
        \begin{aligned}
        \bbW_2\big(\wh\mu_{H_N}(\bx),\mu_{\xi,\vlambda_N^{\circ}}(\vx)\big)
        &\leq \eps,
        \\
        d_{\cH}\big(\supp(\wh\mu_{H_N}(\bx)),\supp(\mu_{\xi,\vlambda_N^{\circ}}(\vx))\big)
        &\leq \eps.
        \end{aligned}
        \end{equation}
        with probability $1-e^{-cN}$ for $c=c(\eps,\cK)>0$ and $N$ large enough.
        (Recall that $\mu_{\vlambda_N^{\circ}}(\vx)$ is defined by \eqref{eq:dyson-equation}, \eqref{eq:mu-def} with $\vlambda_N^{\circ}$ in place of $\vlambda$.)
    \end{enumerate}
\end{proposition}

\begin{proof}
    The first three statements follow by \cite[Proposition 2.1, Theorem 2.6, Corollary 2.7]{ajanki2017singularities} (see the last sentence of \cite[Theorem 2.4]{ajanki2019quadratic} for relevant local uniformity statements), except for the continuity in $\vx$ in part~\ref{it:Dyson-continuous-transform}.
    This is proved in Appendix~\ref{app:dyson} as Theorem~\ref{thm:continuity} (which also allows $\vv\in\obbH^r$).
    Part~\ref{it:Dyson-not-zero} follows since $m_s\approx 0$ is impossible in \eqref{eq:dyson-equation}.

    We now explain part~\ref{it:Dyson-solves-RMT}, which requires a bit more work.
    Throughout, we argue conditionally on $\nabla_{\rd} H_N(\bx)=\vec x$.
    To start, the random matrix $\nabla_{\sph}^2 H_N(\bx)$ obeys the general conditions of \cite[Theorem 4.7(i)]{alt2019location}; in particular its conditional mean (recall Fact~\ref{fac:riemannian-to-euclidean}) $\bbE[\nabla_{\sph}^2 H_N(\bx) | \nabla_{\rd} H_N(\bx)=\vec x] = -\diag(\Lambda^{-1/2} \vx \diamond \bone_\cT)$
  is bounded in operator norm by a constant (denoted $\kappa_4$ in \cite{alt2019location}), uniformly for all $\vx$ in any given compact set (with the precise value $\kappa_4$ depending on the compact set).
    This result implies\footnote{In translating \cite[Theorem 4.7(i)]{alt2019location}, we use the exact equivalence between size $N-r$ Dyson equations with constant entries on the partitions $(\cI_s-1)\times (\cI_{s'}-1)$, and size $r$ Dyson equations with weights $\lambda_{N,s}^{\circ}$.
    See \cite[Section 11.5]{ajanki2019quadratic} for more details.}
    that with probability at least $1-O(1/N)$ (with implicit constant uniform over compact sets of $\vx$), the set $\supp(\wh\mu_{H_N}(\bx))$ is contained within an $\eps/2$-neighborhood of $\supp(\mu_{\xi,\vlambda_N^{\circ}}(\vx))$.
    We first improve this probability to be exponentially close to $1$.
    By the Hoffman--Wielandt lemma (see e.g. \cite[Lemma 2.1.19]{Guionnet}), the $k$-th eigenvalue of any symmetric matrix is an $1$-Lipschitz function of its entries.
    In our setting, Lemma~\ref{lem:derivative-laws} implies that conditionally on $\vec x$, the entries of $\nabla_{\sph}^2 H_N(\bx)$ are independent Gaussians up to symmetry, each with variance $O(1/N)$ (indeed Lemma~\ref{lem:derivative-laws} shows that this variance is exactly determined by $\xi$ and does not depend on $\vec x$ in any way).
    By concentration of Lipschitz functions of Gaussians, we find that $\lambda_k=\lambda_k(\nabla_{\sph}^2 H_N(\bx))$ satisfies for any $\vec x$:
    \[
    \bbP[|\lambda_k-\bbE[\lambda_k]|\geq \eps/4]
    \leq e^{-c(\eps)N}.
    \]
    In particular, if $\lambda_k,\lambda_k'$ are IID copies, then $\bbP[|\lambda_k-\lambda_k'|\geq \eps/2]
    \leq 2e^{-c(\eps)N}$ by the triangle inequality.
    With $E_k(\eps)$ the event that $d(\lambda_k,\supp(\mu_{\xi,\vlambda_N^{\circ}}(\vx)))\geq \eps$, we thus see that
    \[
    \bbP[E_k(\eps)]\cdot (1-O(1/N))
    \leq
    2e^{-c(\eps)N}.
    \]
    This is because if $E_k(\eps)$ holds for $\lambda_k$ but $\lambda_k'$ obeys the conclusion above from \cite[Theorem 4.7(i)]{alt2019location}, then $|\lambda_k-\lambda_k'|\geq \eps/2$ must hold.
    We thus find that $\bbP[E_k(\eps)]\leq e^{-c'(\eps)N}$ for $N$ large.

    Next, we employ \cite[Corollary 1.10]{ajanki2017universality}, which shows that the bounded Lipschitz distance $d_{BL}$ between $\wh\mu_{H_N}(\bx)$ and $\mu_{\xi,\vlambda_N^{\circ}}(\vx)$ tends to $0$, with probability $1-O(1/N)$ and uniform implicit constant over compact sets of $\vx$.
    (Here we again use the equivalence between size $N-r$ Dyson equations with block sizes $(\cI_s-1)\times (\cI_{s'}-1)$ and size $r$ Dyson equations with weights $\lambda_{N,s}^{\circ}$.)
    We have seen that both probability measures are supported in a fixed compact set of $\bbR$ (with probability $1-e^{-cN}$ in the former case; this compact set can be taken uniform over $\vx$ in a compact set).
    This immediately upgrades convergence in probability within $d_{BL}$ to $\bbW_2$.
    Finally, another application of Hoffman--Wielandt shows that the spectral distribution of a symmetric matrix $M\in\bbR^{d\times d}$ is a jointly $1$-Lipschitz function of the entries, as a map from $\bbR^{d\times d}\to \bbW_2(\bbR)$.
    In particular, it follows that
    \[
    D\equiv \bbW_2(\wh\mu_{H_N}(\bx),\mu_{\xi,\vlambda_N^{\circ}}(\vx))
    \]
    is a $1$-Lipschitz function of the entries of $\nabla_{\sph}^2 H_N(\bx)$, and thus concentrates exponentially.
    We have seen that $D$ converges in distribution to $0$, hence its median $m(D)$ satisfies $|m(D)|\leq \eps/2$ for large $N$.
    Therefore $\bbP[|D|\leq \eps]\geq 1-e^{-c(\eps)N}$ for large $N$, yielding the desired $\bbW_2$ convergence claim.

    Finally, we deduce convergence in $d_{\cH}$.
    By adjusting $\eps$, it remains to argue that with probability $1-e^{-c(\eps)N}$, each $y\in \supp(\mu_{\xi,\vlambda_N^{\circ}}(\vx))$ satisfies $d(y,\supp(\wh\mu_{H_N}(\bx)))\leq 2\eps$, as the opposite direction was shown earlier.
    We claim that $\mu_{\xi,\vlambda_N^{\circ}}(\vx)$ is ``locally dense'' in that for any $\eps>0$ there is $\delta>0$ (independent of $\vx$ within any given compact set) such that for all $y\in \supp(\mu_{\xi,\vlambda_N^{\circ}}(\vx))$, we have
    \[
    \mu_{\xi,\vlambda_N^{\circ}}(\vx)([y-\eps,y+\eps])\geq \delta.
    \]
    Indeed this assertion follows by \cite[Theorem 2.6]{ajanki2019quadratic}, which gives a local description of how the density for $\mu_{\xi,\vlambda_N^{\circ}}(\vx)$ behaves near its singularities. (In particular, the local scaling factor $h_x$ therein is stated to be of constant order $h_x\sim 1$, with implicit constants depending only on norms of model parameters.)
    This completes the proof: if $y\in \supp(\mu_{\xi,\vlambda_N^{\circ}}(\vx))$ satisfied $d(y,\supp(\wh\mu_{H_N}(\bx)))\geq 2\eps$, we would directly obtain $\bbW_2\big(\wh\mu_{H_N}(\bx),\mu_{\xi,\vlambda_N^{\circ}}(\vx)\big)\geq \eps\delta>0$, but this $\bbW_2$ distance has been shown to tend to $0$ with exponentially good probability.
\end{proof}

We also have continuity of vector Dyson equation solutions in the various parameters, which is needed to apply the results of \cite{arous2021exponential}.
Sophisticated stability results for the Dyson equation were established for universality for random matrices in \cite{ajanki2017singularities,ajanki2017universality,ajanki2019quadratic,ajanki2019stability}.

\begin{proposition}
\label{prop:VDE-continuity}
    The map
    \[
    (\xi'',\vlambda,\vx)
    \mapsto
    (\vm(z),\mu(z))
    \]
    is uniformly continuous on compact subsets of its domain.
    (That is, for a general symmetric matrix $\xi''\in (0,\infty)^{r\times r}$, vector $\vlambda\in (0,1)^{r}$ with $\sum_s \lambda_s=1$, and vector $\vx\in\bbR^r$. We equip $\vm$ with the compact-open topology and $\mu$ with the $\bbW_1$ distance.)
\end{proposition}

\begin{proof}
    Suppose $(\xi''_n, \vlambda_n,\vx_n)$ converge to $(\xi'',\vlambda,\vx)$ as $n\to\infty$.
    Let $\vm$ be a subsequential limit of the corresponding Dyson equation solutions $\vm_n$.
    Then $\vm$ solves the limiting Dyson equation for $(\xi'',\vlambda,\vx)$ by continuity of the coefficients.
    Since the coefficients $(\xi''_n, \vlambda_n,\vx_n)$ are uniformly bounded above and below, the supports of the corresponding spectral measures $\mu_{n,s}$ are uniformly bounded by \cite{ajanki2019quadratic}.
    Hence the imaginary parts $\Im(m_{n,s}(z))$ of their Stieltjes transforms are bounded below by $\Omega(\Im(z))$, uniformly on compact sets of $(\xi'',\vlambda,\vx,z)$.
    In particular, the limit $\vm$ is still a function from the strict upper half-plane $\bbH$ to itself.
    By uniqueness in Proposition~\ref{prop:MDE-basic}, we find that $\vm$ is \emph{the} solution to the limiting Dyson equation for $(\xi'',\vlambda,\vx)$.
    Since $\vm$ was an arbitrary subsequential limit, and the $\vm_n$ are clearly tight, we find that $\lim_n \vm_n=\vm$, say uniformly on compact subsets of $\bbH$.
    Continuity of $\vm$ follows; this is equivalent to continuity of $\mu_s$ and thus yields continuity of $\mu$.
\end{proof}

In light of Proposition~\ref{prop:MDE-basic}\ref{it:Dyson-compact-supports} and recalling \eqref{eq:ideal-stats}, for any $\vDelta \in \{-1,1\}^r$ we define
\begin{equation}
\label{eq:S-vDelta-def}
S(\vDelta) = \supp(\mu(\vx(\vDelta))).
\end{equation}
Next we define
\begin{equation}
\label{eq:Psi-def}
    \Psi(\vx) = \int \log |\gamma|~[\mu(\vx)](\de \gamma)\,.
\end{equation}
This will capture the exponential growth rate of
\[
    \bbE\lt[
    |\det\lt(\nabla_{\sph}^2 H_N(\bsig)\rt)
    |
    ~\big|~\nabla_{\rd} H_N(\bsig)=\vx\rt]
\]
which is the main term appearing in the Kac--Rice formula.
We show its continuity in Proposition~\ref{prop:Psi-continuous} below, using the following lemmas which will also be useful later.
\begin{lemma}
\label{lem:Dyson-weakly-continuous}
    For any $(\xi,\vx)$ and $(\wt\xi,\wt\vx)$, and some $C = C(\vlambda)>0$,
    \begin{equation}
        \label{eq:mu-tmu-approx}
        \bbW_{\infty}(\mu_{\xi}(\vx),\mu_{\wt\xi}(\wt\vx))
        \leq
        C\big(
            \|\vx-\wt\vx\|_{\infty}
            + \|\xi''-\wt\xi''\|_{\infty}^{1/2}
        \big).
    \end{equation}
    Moreover for $C=C(\vlambda,\xi)>0$ independent of $\vx$,
    \begin{equation}
    \label{eq:mu-vx-approx}
    \bbW_{\infty}
    \big(\mu(\vx),
    \sum_{s\in\sS}
    \lambda_s
    \delta_{-x_s/\sqrt{\lambda_s}}
    \big)
    \leq C.
    \end{equation}
\end{lemma}

\begin{proof}
    Let $\bG = (g_{i,j})_{i,j\in \cT} \in \bbR^{\cT \times \cT}$ be a GOE matrix with $\E [g_{i,j}^2] = (1 + \delta_{i,j})/N$.
    Let $\bW, \wt \bW \in \bbR^{\cT \times \cT}$ be defined by
    \begin{align*}
        W_{i,j} &= \sqrt{\fr{\xi''_{s(i),s(j)}}{\lambda_{s(i)}\lambda_{s(j)}}} g_{i,j}, &
        \wt W_{i,j} &= \sqrt{\fr{\wt \xi''_{s(i),s(j)}}{\lambda_{s(i)}\lambda_{s(j)}}} g_{i,j},
    \end{align*}
    and $\bM, \wt \bM \in \bbR^{\cT \times \cT}$ by $\bM = \bW - \diag(\Lambda^{-1/2} \vx \diamond \bone_\cT)$, $\wt \bM = \wt \bW - \diag(\Lambda^{-1/2} \wt \vx \diamond \bone_\cT)$.
    Then, by Proposition~\ref{prop:spectral-perturbation},
    \[
        \bbW_\infty(\wh\mu(\bM), \wh\mu(\wt \bM))
        \le \|\bM - \wt \bM\|_\op
        \le \|\bW - \wt \bW\|_\op + \fr{\|\vx - \wt \vx\|_\infty}{\sqrt{\min \vlambda}}.
    \]
    It is classical that $\|\bG\|_\op \le 3$ with probability $1-e^{-cN}$.
    By Slepian's lemma $\|\bW - \wt \bW\|_\op$ is stochastically dominated by
    \[
        \frac{\|\xi'' - \wt \xi''\|_\infty^{1/2}}{\min \vlambda} \|\bG\|_\op,
    \]
    so with probability $1-e^{-cN}$, for suitable $C$,
    \begin{align*}
        \bbW_\infty(\wh\mu(\bM), \wh\mu(\wt \bM))
        \le \frac{3\|\xi'' - \wt \xi''\|_\infty^{1/2}}{\min \vlambda}
        + \fr{\|\vx - \wt \vx\|_\infty}{\sqrt{\min \vlambda}}
        &\le
        C\lt(
            \|\vx-\wt\vx\|_{\infty} + \|\xi''-\wt\xi''\|_{\infty}^{1/2}
        \rt)
        \equiv \oC
        \\
        \implies
        \wh\mu(\bM)([t+\oC,\infty))
        &\geq
        \wh\mu(\wt \bM)([t+2\oC,\infty)), \quad \forall t\in \bbR.
    \end{align*}
    By Propositions~\ref{prop:MDE-basic}\ref{it:Dyson-solves-RMT} and \ref{prop:VDE-continuity}, for any $\eps > 0$, with probability $1-e^{-cN}$
    \[
        \bbW_2(\wh\mu(\bM), \mu_{\xi}(\vx)),
        \bbW_2(\wh\mu(\wt \bM), \mu_{\wt \xi}(\wt \vx))
        \le \eps\oC.
    \]
    In particular for any $\eps$ (depending on $\xi,\vx,\tilde \xi,\tilde\vx)$ and $t\in\bbR$ we have
    \[
    \mu_{\xi}(\vx)([t,\infty))
    \geq
    \wh\mu(\bM)([t+\oC,\infty))-\eps,
    \quad \quad
    \wh\mu(\wt \bM)([t+2\oC,\infty))
    \geq
    \mu_{\wt \xi}(\wt \vx)([t+3\oC,\infty))-\eps.
    \]
    Combining the above displays gives
    \[
    \mu_{\xi}(\vx)([t,\infty))
    \geq
    \mu_{\wt \xi}(\wt \vx)([t+3\oC,\infty))-2\eps,\quad\forall \eps>0.
    \]
    By similar reasoning the same inequality holds with $(\xi,\vx)$ and $(\widetilde \xi, \widetilde \vx)$ interchanged.
    This completes the proof (with $\oC$ replaced by $3\oC$) since $\eps$ is arbitrary.
    The second part \eqref{eq:mu-vx-approx} follows by similar reasoning since in the corresponding matrix model, the centered Gaussian contribution has spectral norm at most $C(\xi)$ with probability $1-e^{-cN}$.
\end{proof}

The following definition of distributions with bounded density and support will be convenient to ensure continuity of integrals against singular log potentials; it also reappears in Section~\ref{sec:approximate}.

\begin{definition}
\label{def:C-regular-mu}
    The probability distribution $\mu\in\cP(\bbR)$ is \textbf{$C$-regular} if $\supp(\mu)\subseteq [-C,C]$ and $\mu$ has density at most $C$ with respect to Lebesgue measure.
\end{definition}

\begin{lemma}
\label{lem:log-integral-convergence}
    For any $C,\eps>0$ there exists $\delta>0$ such that if $\mu,\wt\mu$ are $C$-regular and $\bbW_1(\mu,\wt\mu)\leq \delta$ then
    \[
    \lt|
    \int
    \log |\lambda|
    \de\mu(\lambda)
    -
    \int
    \log |\lambda|
    \de\wt\mu(\lambda)
    \rt|
    \leq
    \eps
    .
    \]
\end{lemma}

\begin{proof}
    Define the truncation $\log_K(x)=\min(K,\max(-K,\log x))$. It is easy to see that $\log_K(x)$ is $L_K$-Lipschitz for some constant $L_K$, so for $\delta\leq \frac{\eps}{2L_K}$ we have
    \[
    \lt|
    \int
    \log_K |\lambda|
    \de\mu(\lambda)
    -
    \int
    \log_K |\lambda|
    \de\wt\mu(\lambda)
    \rt|
    \leq
    L_K\cdot \bbW_1(\mu,\wt\mu)
    \leq
    \eps/2
    .
    \]
    For $K\geq \log(C)$ and $|x|\le C$, we have
    \[
    f_K(x)
    \equiv
    \log(|x|)-\log_K(|x|)
    =
    (K+\log(|x|))\cdot 1_{|x|\leq e^{-K}}
    .
    \]
    $C$-regularity implies
    \begin{align*}
    \lt|
    \int
    f_K(\lambda)
    \de\mu(\lambda)
    -
    \int
    f_K(\lambda)
    \de\wt\mu(\lambda)
    \rt|
    &\leq
    2C
    \lt|
    \int_{-e^{-K}}^{e^{-K}}
    K+\log|x|
    ~\de x
    \rt|
    \\
    &=
    -4C(x\log x-x+Kx)|_{x=0}^{e^{-K}}
    \\
    &=
    4Ce^{-K}
    .
    \end{align*}
    It remains to choose $K$ so $4Ce^{-K}\leq \eps/2$ and then take $\delta\leq \eps/2L_K$ as above.
\end{proof}

\begin{proposition}
    \label{prop:Psi-continuous}
    $\Psi(\vx)$ is continuous in $(\xi'',\vlambda,\vx)$, uniformly on compact sets of $(\xi',\xi'',\vlambda,\vx)$ with $\xi$ non-degenerate.
\end{proposition}

\begin{proof}
    This is immediate from Propositions~\ref{prop:MDE-basic}\ref{it:Dyson-continuous-density} and \ref{prop:VDE-continuity}, and Lemmas~\ref{lem:Dyson-weakly-continuous}, \ref{lem:log-integral-convergence}.
\end{proof}

\paragraph*{Organization}

The remainder of the paper is structured as follows.
In Section~\ref{sec:annealed-kac-rice} we determine the annealed complexity of critical points (Theorem~\ref{prop:annealed-kac-rice}).
In Section~\ref{sec:variational} we solve the resulting variational problem, identifying the $2^r$ potential \emph{types} of critical points for super-solvable $\xi$ and showing that no others occur (Proposition~\ref{prop:variational-maximum}).
In Section~\ref{sec:approximate} we connect Kac--Rice estimates to non-existence of approximate critical points (Theorem~\ref{thm:approx-crits-from-annealed}).
In Section~\ref{sec:2^r} we complete the proof of strong topological trivialization (Theorem~\ref{thm:precise-landscape}) through the shrinking bands recursion explained in the introduction.
In Section~\ref{sec:approx-local-max-E-infty} we present further implications of Theorem~\ref{thm:approx-crits-from-annealed} to approximate local maxima and marginal states in the single-species case (e.g. Corollary~\ref{cor:zero-eigenvalue-ground-states}).
Finally in Appendix~\ref{app:dyson} we study solutions to the vector Dyson equation, obtaining joint $1/3$-H{\"o}lder continuity (Theorem~\ref{thm:continuity}), a detailed characterization of the boundary behavior (e.g. Theorem~\ref{thm:feasible-region} and Proposition~\ref{prop:edge-vs-cusp}), and an explicit formula for the main determinant term appearing in the Kac--Rice computation (Theorem~\ref{thm:logdet}).

\section{Expected Critical Point Counts}
\label{sec:annealed-kac-rice}

In this section we determine the annealed critical point statistics of $H_N$ to leading exponential order.

\subsection{Formula for the Complexity Functional}

It will be crucial that only an exponentially small fraction of critical points in the annealed sense are atypical in the sense below, which closely resembles the definition of $\Crt_N^{\good,\eps}$.

\begin{definition}
\label{def:typical}
We say $\bx\in\cS_N$ is respectively $\eps$-energy-typical, $\eps$-overlap-typical, and $\eps$-bulk-typical if with $\vx=\nabla_{\rd}H_N(\bx)$ it satisfies the three conditions (recall \eqref{eq:Lambda-and-A}):
\begin{align*}
    \lt|
        \fr1N H_N(\bx) -
        (\xi')^\top A^{-1}\Lambda^{1/2} \vx
    \rt|
    &\leq \eps,
    \\
    \norm{
        \vR(\bG^{(1)},\bx) -
        \Lambda^{-1/2} \diag(\Gamma^{(1)}) A^{-1} \Lambda^{1/2} \vx
    }_\infty
    &\le \eps\,,
    \\
    \bbW_2\lt(\wh\mu_{H_N}(\bx),\mu(\vx)\rt)
    &\leq \eps
    \quad\text{and}\quad
    d_{\cH}\lt(\spec_{H_N}(\bx),\supp(\mu(\vx))\rt)\leq \eps
    .
\end{align*}
If these conditions are not satisfied, $\bx$ is respectively $\eps$-energy-atypical, $\eps$-overlap-atypical, and $\eps$-bulk-atypical.
We say $\bx$ is $\eps$-typical if all three typicality conditions hold, and $\eps$-atypical otherwise.
\end{definition}

Given a set $\cD\subseteq\bbR^r \times \bbR$, let $\Crt_N(\cD)$ denote the set of critical points $\bsig$ for $H_N$ with $(\nabla_{\rd}H_N(\bsig), H_N(\bsig)/N) \in \cD$.
Also, for $\cD \subseteq \bbR^r$, let $\Crt_N^{(\eps)}(\cD)$ the set of critical points with $\nabla_{\rd}H_N(\bsig) \in \cD$ which are $\eps$-atypical.
Recalling \eqref{eq:Psi-def}, define the complexity functionals $F:\bbR^r\to\bbR$ and $F:\bbR^r \times \bbR \to \bbR$ by
\begin{align}
    \label{eq:def-F}
    F(\vx) &=
    \frac{1}{2}
    \Big(
    1
    -
    \sum_{s\in \sS} \lambda_s \log \xi^s(\vone)
    - \tnorm{A^{-1/2} \Lambda^{1/2} \vx}_2^2
    \Big)
    + \Psi(\vx)\,, \\
    \label{eq:def-F-extended}
    F(\vx,E) &= F(\vx) - \fr{(E - (\xi')^\top A^{-1}\Lambda^{1/2} \vx)^2}{2(\xi(\vone) - (\xi')^\top A^{-1} \xi')}\,.
\end{align}
The following main result of this section characterizes the annealed critical point complexity of $H_N$ in terms of these functionals.

\begin{proposition}
\label{prop:annealed-kac-rice}
    Fix $\vlambda$ and non-degenerate $\xi$.
    Let $\cD\subseteq\bbR^r \times \bbR$ be the closure of its non-empty interior. Then,
    \begin{equation}
    \label{eq:annealed-kac-rice}
    \lim_{N\to\infty}
    \frac{1}{N}\log \bbE |\Crt_N(\cD)|
    =
    \sup_{(\vx,E)\in\cD}
    F(\vx,E).
    \end{equation}
    Moreover, for $\cD \subseteq \bbR^r$ equal to the closure of its non-empty interior and $\eps>0$, there exists $c=c(\xi,\eps) > 0$ such that
    \begin{equation}
    \label{eq:annealed-kac-rice-atypical}
    \limsup_{N\to\infty}
    \frac{1}{N}\log \bbE |\Crt_N^{(\eps)}(\cD)|
    \le
    \sup_{\vx\in\cD}
    F(\vx) - c.
    \end{equation}
\end{proposition}

\subsection{Proof of Proposition~\ref{prop:annealed-kac-rice}}

Below, we focus on proving \eqref{eq:annealed-kac-rice} and then explain the necessary changes to reach \eqref{eq:annealed-kac-rice-atypical}.
For fixed $\vx\in\bbR^r$, let $M_N=M_N(\vx)\in\bbR^{\cT \times \cT}$ be a Gaussian matrix with distribution
\[
    M_N \sim \cL(\nabla_{\sph}^2 H_N(\bsig)\,|\,\nabla_{\rd} H_N(\bsig)=\vx).
\]
(Here $\cL(\cdot|\cdot)$ denotes a conditional law; since $\nabla_{\rd}H_N(\bsig)$ is a linear function of $H_N$, there is no difficulty in defining regular conditional laws.)
This can be written explicitly as follows.
Let $\bW$ be the random matrix with law given by \eqref{eq:tangential-hessian-law}.
By Fact~\ref{fac:riemannian-to-euclidean},
\begin{equation}
    \label{eq:def-M_N}
    M_N \stackrel{d}{=} \bW - \diag(\Lambda^{-1/2} \vx \diamond \bone_{\cT})\,.
\end{equation}
We let $\wh\mu_{M_N} = \wh\mu(M_N)$ denote the (random) spectral measure of this matrix.
We further define the finite-$N$ vector Dyson equation
\begin{equation}
\label{eq:finite-N-Dyson}
    1 + \lt(
        z + \fr{x_s}{\sqrt{\lambda_s}} +
        \sum_{s' \in \sS} \fr{\lambda_{N,s}^{\circ} \xi''_{s,s'}}{\lambda_s^2} m_{N,s'}(z)
    \rt) m_{N,s}(z) = 0\,, \qquad z \in \bbH\,.
\end{equation}
with unique solution $\vm_N(z) = (m_{N,s}(z))_{s\in \sS}$.
(This is the matrix Dyson equation from e.g. \cite[Section 1.10]{arous2021exponential} with $M(z) = \diag(\vm_N(z) \diamond \bone_\cT)$.)
We let $\mu_{M_N,s}$ be the measure with Stieltjes transform $m_{N,s}$ and
\[
    \mu_{M_N} = \sum_{s\in \sS} \lambda_{N,s}^{\circ} \mu_{M_N,s}.
\]
This $\vm_N(z)$ and $\mu_{M_N}$ exist and are unique by \cite[Proposition 5.1~(i),(ii)]{erdos2019random}.
In the next proposition, we give the required conditions to apply \cite[Theorem 4.1]{arous2021exponential}.


\begin{proposition}
    \label{prop:kac-rice-conditions}
    Given $\vlambda,\delta$, a bounded family of $\vx\in\bbR^r$ and a uniformly non-degenerate family of $\xi$, the following hold uniformly over the families for some $C,c>0$:
    \begin{enumerate}[label=(\alph*)]
        \item
        \label{it:kac-rice-condition-0}
        $\bbP[\supp(\wh\mu_{M_N})\subseteq [-C,C]]\geq 1-e^{-cN}$.
        \item
        \label{it:kac-rice-condition-1}
        $\bbW_1\big(\bbE[\wh\mu_{M_N}],\mu_{M_N}\big)\leq N^{-c}$.
        \item
        \label{it:kac-rice-condition-2}
        $\bbP\lt[\bbW_1\big(\wh\mu_{M_N},\bbE\wh\mu_{M_N}\big)\geq \delta\rt]\leq e^{-cN}$.
        \item
        \label{it:kac-rice-condition-3}
        $\lim_{N\to\infty} \bbP[\wh{\mu}_{M_N}([-N^{-5},N^{-5}])>0]=0$.
        \item
        \label{it:kac-rice-condition-4}
        There exists an entrywise continuous-in-$\vx$ coupling of the matrices $M_N(\vx)$.
        \item
        \label{it:kac-rice-condition-5}
        For all $\vx\in\bbR^r$,
        \[
        \bbE\lt[\lt|\det(M_N(\vx))\rt|\rt]\leq C^N(\|\vx\|_{\infty}+1)^N.
        \]
    \end{enumerate}
\end{proposition}

\begin{proof}
    Point~\ref{it:kac-rice-condition-0} follows by Proposition~\ref{prop:MDE-basic}\ref{it:Dyson-compact-supports}\ref{it:Dyson-solves-RMT}.
    Point~\ref{it:kac-rice-condition-1} is a standard result on stability of the vector Dyson equation, see \cite[Proof of Corollary 1.9.B]{arous2021exponential}, except that \ref{it:kac-rice-condition-1} is usually shown for a variant of $\mu_{M_N}$ solving a Dyson equation with additional $O(1/N)$ terms on the diagonal entries $\xi''_{i,i}$.
    This discrepancy causes negligible error $N^{-c}$ as shown in \cite[Proposition 3.1]{arous2021exponential} and \cite[Lemma 3.1]{arous2021landscape}, so the claim does follow (see also \cite[Section 3.1]{mckenna2021complexity} for further discussion of this purely technical issue).
    By \cite[Lemma 1.2(b)]{guionnet2000concentration}, for any $1$-Lipschitz test function $f$ the map
    \[
        (W_{i,j})_{1\le i<j\le k} \mapsto \int f(\lambda)\wh\mu_{M_N}(\de \lambda)
    \]
    is $O(1)$-Lipschitz.
    In particular, writing $W_{i,j} = \sqrt{(1+\delta_{i,j})\xi''_{s(i),s(j)} / \lambda_{s(i)} \lambda_{s(j)}} g_{i,j}$ for i.i.d. gaussians $g_{i,j}$, and applying gaussian concentration of measure,
    \[
        \bbP\lt[
            \lt|\int f(\lambda)\wh\mu_{M_N}(\de \lambda) - \int f(\lambda)\E \wh\mu_{M_N}(\de \lambda)\rt| \ge \delta'
        \rt] \le e^{-c(\delta')N}.
    \]
    for any $\delta' > 0$.
    Point \ref{it:kac-rice-condition-2} then follows by union bounding over $O(1)$ test functions $f$.
    Point \ref{it:kac-rice-condition-3} follows by averaging over a small global shift of $M_N$ by the identity matrix.
    Indeed since $\xi$ is non-degenerate, we can express the law of $M_N$ as the sum of two independent matrices, one of which is $N^{-2}gI_{\cT}$ for a scalar Gaussian $g\sim\cN(0,1)$.
    Then point \ref{it:kac-rice-condition-3} holds even after conditioning on the other summand, since each of the $N$ eigenvalues has conditional probability $O(N^{-2})$ to lie in $[-N^{-5},N^{-5}]$.
    Point \ref{it:kac-rice-condition-4} is clear.
    Finally point \ref{it:kac-rice-condition-5} easily follows from the deterministic inequality
    \[
    |\det(M_N)|\leq \lt(\frac{\|M_N\|_F^2}{N-r}\rt)^{(N-r)/2}
    \]
    which is a consequence of the arithmetic mean-geometric mean inequality.
\end{proof}

\begin{proof}[Proof of Proposition~\ref{prop:annealed-kac-rice}]
    Define
    \begin{align*}
        \Psi_N(\vx)
        &=
        \fr{1}{N} \log \E |\det(M_N(\vx))|,
        \\
        F_N(\vx)
        &=
        \fr12
        \Big(
        1-\sum_{s\in \sS} \lambda_s \xi^s(\vone)
        - \tnorm{A^{-1/2} \Lambda^{1/2} \vx}_2^2
        \Big)
        + \Psi_N(\vx)\,.
    \end{align*}
    Let $\varphi_{X}$ be the density of random variable $X$ w.r.t. Lebesgue measure.
    By the Kac--Rice formula (see e.g. \cite[Chapter 11]{adler2007random}),
    \begin{align*}
        \E|\Crt_N(\cD)|
        &=
        \int_{\cS_N}
        \int_{\bbR^r}
        \bigg(
        \bbE\Big[
            |\det \nabla^2_{\sph} H_N(\bsig)|
            \ind\{(\vx, H_N(\bsig)/N) \in \cD\} \\
            &\qquad\qquad ~\Big|~ \nabla_{\sph} H_N(\bsig) = \bzero,\nabla_{\rd}H_N(\bsig) = \vx
        \Big] \\
        &\qquad \times \varphi_{\nabla_{\sph} H_N(\bsig)}(\bzero)
        \varphi_{\nabla_{\rd}H_N(\bsig)}(\vx)
        \bigg)
        ~\de \vx
        ~\de \cH^{N-r}(\bsig)\,,
    \end{align*}
    where $\cH^{N-r}$ denotes the $(N-r)$-dimensional Hausdorff measure on $\cS_N$.
    By spherical invariance, the integrand does not depend on $\bsig$, so the integral over $\cH^{N-r}$ simply contributes a volume factor given by Fact~\ref{fac:volume}.
    By Fact~\ref{fac:riemannian-to-euclidean} and Lemma~\ref{lem:derivative-laws},
    \begin{align}
        \notag
        &\bbE\Big[
            |\det \nabla^2_{\sph} H_N(\bsig)|
            \ind\{(\vx, H_N(\bsig)/N) \in \cD\}
            ~\Big|~ \nabla_{\sph} H_N(\bsig) = \bzero,\nabla_{\rd}H_N(\bsig) = \vx
        \Big] \\
        \label{eq:conditional-independence}
        &= \bbE\lt[|\det M_N(\vx)|\rt]
        \bbP\lt[
            (\vx, H_N(\bsig)/N) \in \cD
            ~\Big|~ \nabla_{\rd}H_N(\bsig) = \vx
        \rt] \\
        \notag
        &= e^{o(N)} \bbE\lt[|\det M_N(\vx)|\rt] \int_{\bbR}
        \ind\{(\vx,E) \in \cD\} \\
        \notag
        &\qquad \qquad \times
        \exp\lt(
            -\fr{N (E - (\xi')^\top A^{-1} \Lambda^{1/2} \vx)^2}{2(\xi(\vone) - (\xi')^\top A^{-1} \xi')}
        \rt) ~\de E.
    \end{align}
    We further have
    \begin{align}
        \nonumber
        \varphi_{\nabla_{\sph} H_N(\bsig)}(\bzero)
        &= \prod_{s\in \sS}
        (2\pi \xi^s(\vone))^{-(|\cI_s|-1)/2} \\
        \label{eq:density-exp-rate}
        \implies \fr{1}{N} \log \varphi_{\nabla_{\sph} H_N(\bsig)}(\bzero)
        &= -\frac{\log(2\pi) +\sum_{s\in \sS} \lambda_s \log \xi^s(\vone)}{2} + o_N(1)\,.
    \end{align}
    and
    \[
        \varphi_{\nabla_{\rd}H_N(\bsig)}(\vx)
        =
        (2\pi)^{-r/2} \sqrt{\fr{\det \Lambda}{\det A}}
        \exp\lt(-\fr N2 \tnorm{A^{-1/2} \Lambda^{1/2}\vx}_2^2 \rt).
    \]
    Thus, up to additive $o_N(1)$ error, using Fact~\ref{fac:volume} in the second step gives
    \begin{align}
        \notag
        \frac{1}{N}\log \E|\Crt_N(\cD)|
        &\approx
        \fr{1}{N} \log \cH^{N-r}(\cS_N)
        + \fr{1}{N} \log \varphi_{\nabla_{\sph} H_N(\bsig)}(\bzero)
        + \frac{1}{N}
        \log
        \int_{\cD}
        \bbE|\det M_N(\vx)| \\
        \notag
        &\quad \times
        \exp\lt(
            -\fr N2 \tnorm{A^{-1/2} \Lambda^{1/2} \vx}_2^2
            -\fr{N (E - (\xi')^\top A^{-1} \Lambda^{1/2} \vx)^2}{2(\xi(\vone) - (\xi')^\top A^{-1} \xi')}
        \rt)
        \de (\vx,E)
        \\
        \label{eq:need-laplace-method}
        &\approx
        \fr12 \Big(1-\sum_{s\in \sS}
        \lambda_s \log \xi^s(\vone)\Big)
        + \frac{1}{N}
        \log
        \int_{\cD}
        \exp(N\Psi_N(\vx))
        \\
        \label{eq:quadratic-term}
        &\quad
        \times
        \exp\lt(
            -\fr N2 \tnorm{A^{-1/2} \Lambda^{1/2} \vx}_2^2
            -\fr{N (E - (\xi')^\top A^{-1} \Lambda^{1/2} \vx)^2}{2(\xi(\vone) - (\xi')^\top A^{-1} \xi')}
        \rt)
        \de (\vx,E)
        .
    \end{align}
    We next analyze the behavior of $\Psi_N$ via \cite[Corollary 1.9.A]{arous2021exponential}, where the necessary conditions hold by Proposition~\ref{prop:kac-rice-conditions}.
    This result expresses the asymptotic value of $\Psi_N$ based on the solution to the finite-$N$ vector Dyson equation \eqref{eq:finite-N-Dyson}.
    In fact \eqref{eq:finite-N-Dyson} is exactly equivalent to the limiting Dyson equation \eqref{eq:dyson-equation} with $\vlambda$ replaced by $\vlambda_N^{\circ}$ from \eqref{eq:lambda-N-s} (see e.g. the discussion in \cite[Section 11.5]{ajanki2019quadratic}).
    Therefore \cite[Corollary 1.9.A]{arous2021exponential} shows that uniformly on compact sets, up to $o_N(1)$ error:
    \[
    \Psi_N(\vx)
    \approx
    \int \log|\gamma|[\mu_{\vlambda_N^{\circ}}(\vx)](\de\gamma).
    \]
    Recalling Propositions~\ref{prop:VDE-continuity} and \ref{prop:Psi-continuous} shows that, again uniformly on compact sets:
    \[
    \int \log|\gamma|[\mu_{\vlambda_N^{\circ}}(\vx)](\de\gamma)
    \approx
    \int \log|\gamma|[\mu(\vx)](\de\gamma)
    =
    \Psi(\vx)
    .
    \]
    We next deduce \eqref{eq:annealed-kac-rice} via Laplace's method similarly to \cite[Theorem 4.1]{arous2021exponential}. (The latter result is not directly applicable as it is stated with additional technical requirements, but the proof can be routinely adapted.)
    For compact $\cD$, local uniformity of the approximations just above allows replacement of $\Psi_N(\vx)$ by its limit $\Psi(\vx)$ in \eqref{eq:need-laplace-method} up to $o_N(1)$ error.
    Since $\Psi(\vx)$ is continuous by Proposition~\ref{prop:Psi-continuous}, Laplace's method then immediately gives \eqref{eq:annealed-kac-rice} for compact $\cD$.
    It remains to show that \eqref{eq:annealed-kac-rice} respects exhaustion by compact sets, both for finite $N$ and after passing to the limit.

    The needed statement at finite $N$ is that
    \begin{equation}
    \label{eq:finite-N-exp-tight}
    \lim_{R\to\infty}
    \limsup_{N\to\infty}
    \frac{1}{N}
    \log\bbE[|\Crt_N(\bbR^{r+1}\backslash\cD_R)|]=-\infty
    \end{equation}
    where $\cD_R=(-R,R)^{r+1}$.
    Similarly to \cite[Lemma 4.3]{arous2021exponential}, this follows from the non-asymptotic bound $\bbE[|\det(M_N(\vx))|]\leq (C\max(\|\vx\|_{\infty},1))^N$, since when either $\|\vx\|_{\infty}\geq R$ or $|E|\geq R$ the quadratic term \eqref{eq:quadratic-term} contributes an overwhelming $e^{-\Omega(NR^2)}$ factor.
    Said bound is shown exactly as in \cite[Lemma 3.7]{mckenna2021complexity}, by writing $M_N(\vx)=W_N+A_N(\vx)$ for deterministic $A_N$ and centered Gaussian $W_N$.
    Namely one can separate $M_N(\vx)$ with the deterministic estimate $|\det(M_N(\vx)|^N\leq 2^N(\|W_N\|_{\op}^N+\|A_N(\vx)\|_{\op}^N)$, and use the simple bound $\bbP[\|W_N\|_{\op}\geq t]\leq e^{-cN(t-C)_+}$ (which follows because $\|W_N\|_{\op}$ is typically $O(1)$ and is $O(1)$-Lipschitz in its independent Gaussian entries) to control the random part.

    We also need to show that $F(\vx,E)$ tends to $-\infty$ as $\max(\|\vx\|_{\infty},|E|)\to\infty$.
    This follows because $\Psi(\vx)\leq C(\max(\|\vx\|_{\infty},1))$ due to \eqref{eq:mu-vx-approx}, which is dominated by the quadratic terms of $F(\vx,E)$.
    Together with \eqref{eq:finite-N-exp-tight}, this allows us to deduce \eqref{eq:annealed-kac-rice} for general $\cD$ from the compact $\cD$ case via exhaustion.
    Namely one restricts to the compact set $\cD\cap [-R,R]^{r+1}$ and sends $R\to\infty$ after $N\to\infty$ (exactly as in e.g. \cite[Proof of Theorem 4.1]{arous2021exponential}).
    This completes the first part of the proof.

    Moving onto \eqref{eq:annealed-kac-rice-atypical}, we separately address the cases of energy, overlap, and bulk-atypicality.
    Energy-atypicality follows directly from \eqref{eq:annealed-kac-rice}, as the term involving $E$ in \eqref{eq:def-F-extended} is nonzero.
    For overlap-atypicality, let $E_{\eps}^{\text{o}}(\bsig)$ denote the event that $\bsig$ is $\eps$-overlap-atypical, and $\Crt_N^{(\eps,\text{o})}(\cD)$ be the set of critical points with $\nabla_{\rd} H_N(\bsig) \in \cD$ which are $\eps$-overlap-atypical.
    (Recall that for \eqref{eq:annealed-kac-rice-atypical}, $\cD$ is a subset of $\bbR^r$ rather than $\bbR^r \times \bbR$.)
    By the Kac--Rice formula,
    \begin{align*}
        \E|\Crt_N^{(\eps,\text{o})}(\cD)|
        &=
        \int_{\cS_N}
        \int_{\cD}
        \bigg(
        \bbE\Big[
            |\det \nabla^2_{\sph} H_N(\bsig)|
            \ind\{E_{\eps}^{\text{o}}(\bsig)\}
            ~\Big|~ \nabla_{\sph} H_N(\bsig) = \bzero,\nabla_{\rd}H_N(\bsig) = \vx
        \Big] \\
        &\qquad \qquad \qquad
        \varphi_{\nabla_{\sph} H_N(\bsig)}(\bzero)
        \varphi_{\nabla_{\rd}H_N(\bsig)}(\vx)
        \bigg)
        ~\de \vx
        ~\de \cH^{N-r}(\bsig).
    \end{align*}
    By calculations similar to above, up to additive $o_N(1)$ error
    \begin{align*}
        &\frac{1}{N}\log \E|\Crt_N^{(\eps,\text{o})}(\cD)| \\
        &\approx
        \fr12 \lt(1-\sum_{s\in \sS} \lambda_s \log \xi^s(\vone)\rt)
        + \frac{1}{N}
        \log
        \int_{\cD}
        \exp\lt(
            N\Psi(\vx)
            -\fr N2 \tnorm{A^{-1/2} \Lambda^{1/2} \vx}_2^2
        \rt) \\
        &\qquad \qquad \qquad \qquad \qquad \qquad \qquad \qquad  \times
        \bbP\lt[
            E_{\eps}^{\text{o}}(\bsig)
            ~\Big|~
            \nabla_{\rd}H_N(\bsig) = \vx
        \rt]
        \de \vx.
    \end{align*}
    By \eqref{eq:1spin-law}, each entry of $\vR(\bG^{(1)},\bsig)$ has variance bounded above by $1/(N \min \vlambda)$.
    Since $\vR(\bG^{(1)},\bsig)$ and $\nabla_{\rd}H_N(\bsig)$ are jointly gaussian, this remains true after conditioning on $\nabla_{\rd}H_N(\bsig)$.
    In light of \eqref{eq:conditional-1spin-mean}, this implies
    \[
        \bbP\lt[
            E_{\eps}^{\text{o}}(\bsig)
            ~\Big|~
            \nabla_{\rd}H_N(\bsig) = \vx
        \rt] \le e^{-cN}.
    \]
    This implies \eqref{eq:annealed-kac-rice-atypical} for overlap-typicality.

    The main case that needs to be addressed is bulk-atypicality.
    Let $E_{\eps}^{\text{b}}(\bsig)$ be the event that $\bsig$ is $\eps$-bulk-atypical and $\Crt_N^{(\eps,\text{b})}(\cD)$ be the set of critical points with $\nabla_{\rd} H_N(\bsig) \in \cD$ which are $\eps$-bulk-atypical.
    Similarly to above,
    \begin{align*}
        &\E|\Crt_N^{(\eps,\text{b})}(\cD)| \\
        &=
        \int_{\cS_N}
        \int_{\cD}
        \bigg(
        \bbE\Big[
            |\det \nabla^2_{\sph} H_N(\bsig)|
            \ind\{E_{\eps}^{\text{b}}(\bsig)\}
            ~\Big|~ \nabla_{\sph} H_N(\bsig) = \bzero,\nabla_{\rd}H_N(\bsig) = \vx
        \Big] \\
        &\qquad \qquad \times
        \varphi_{\nabla_{\sph} H_N(\bsig)}(\bzero)
        \varphi_{\nabla_{\rd}H_N(\bsig)}(\vx)
        \bigg)
        ~\de \vx
        ~\de \cH^{N-r}(\bsig),
    \end{align*}
    and so up to $o_N(1)$ additive error
    \begin{align*}
        \frac{1}{N}\log \E|\Crt_N^{(\eps,\text{b})}(\cD)|
        &\approx
        \fr12 \lt(1-\sum_{s\in \sS} \lambda_s \log \xi^s(\vone)\rt)
        + \frac{1}{N}
        \log
        \int_{\cD}
        \exp\lt(
            -\fr N2 \tnorm{A^{-1/2} \Lambda^{1/2} \vx}_2^2
        \rt) \\
        &\qquad \times
        \E\lt[
            |\det \nabla^2_{\sph} H_N(\bsig)|
            \bone\{E_{\eps}^{\text{b}}(\bsig)\}
            ~\Big|~
            \nabla_{\rd}H_N(\bsig) = \vx
        \rt]
        \de \vx.
    \end{align*}
    Unlike above, $\bone\{E_{\eps}^{\text{b}}(\bsig)\}$ and the Hessian determinant are not independent.
    Instead, by Cauchy--Schwarz, the last expectation is bounded by
    \[
        \E\lt[
            |\det \nabla^2_{\sph} H_N(\bsig)|^2
            ~\Big|~
            \nabla_{\rd}H_N(\bsig) = \vx
        \rt]^{1/2}
        \bbP\lt[
            E_{\eps}^{\text{b}}(\bsig)
            ~\Big|~
            \nabla_{\rd}H_N(\bsig) = \vx
        \rt]^{1/2}.
    \]
    By \cite[Theorem A.2]{arous2021exponential},
    \begin{align*}
        &\E\lt[
            |\det \nabla^2_{\sph} H_N(\bsig)|^2
            ~\Big|~
            \nabla_{\rd}H_N(\bsig) = \vx
        \rt]^{1/2} \\
        &= e^{o(N)}
        \E\lt[
            |\det \nabla^2_{\sph} H_N(\bsig)|
            ~\Big|~
            \nabla_{\rd}H_N(\bsig) = \vx
        \rt]
        = e^{o(N)} \exp(N\Psi_N(\vx)),
    \end{align*}
    and by Proposition~\ref{prop:MDE-basic}\ref{it:Dyson-solves-RMT} and \ref{prop:VDE-continuity},
    \[
        \bbP\lt[
            E_{\eps}^{\text{b}}(\bsig)
            ~\Big|~
            \nabla_{\rd}H_N(\bsig) = \vx
        \rt]^{1/2}
        \le e^{-cN/2}.
    \]
    Combining and arguing as above completes the proof (with $c/2$ in place of $c$).
\end{proof}

\section{Solving the Variational Problem}
\label{sec:variational}

Due to Proposition~\ref{prop:annealed-kac-rice}, in order to establish Theorem~\ref{thm:annealed-crits} it remains to maximize $F(\vx)$ over $\bbR^r$.
In this section we prove parts \ref{itm:crt-tot} and \ref{itm:crt-bad} of this theorem, regarding super-solvable $\xi$.
The strictly sub-solvable case \ref{itm:crt-tot-sub-solvable} will be proved in Subsection~\ref{subsec:annealed-complexity-positive}.
\begin{proposition}
    \label{prop:variational-maximum}
    Assume $\xi$ is strictly super-solvable.
    Then $F(\vx) \le 0$ for all $\vx \in \bbR^r$, with equality at precisely the $2^r$ points $\vx(\vDelta)$ for $\vDelta \in \{-1,1\}^r$.
\end{proposition}
We also state the following regularity property of $F$, which implies that its maximum is attained at a stationary point.
\begin{lemma}
    \label{lem:F-regularity}
    The function $F$ is continuously differentiable in $\vx$, and
    \[
        \lim_{R\to\infty}
        \sup_{\norm{\vx}_\infty \ge R} F(\vx) = -\infty\,.
    \]
    Moreover the latter limit is uniform on bounded, uniformly non-degenerate $\xi$.
\end{lemma}
\begin{proof}
    Continuous differentiability follows from Lemma~\ref{lem:oF-oPsi-formulas} below and Lemma~\ref{lem:Dyson-weakly-continuous}.
    For $R = \norm{\vx}_\infty$, we have $\Psi(\vx) \lesssim \log R$ while $\la \Lambda^{1/2} \vx, A^{-1} \Lambda^{1/2} \vx \ra \gtrsim R^2$,  which establishes the decay at infinity.
\end{proof}

\begin{proof}[Proof of Theorem~\ref{thm:annealed-crits} parts~\ref{itm:crt-tot}, \ref{itm:crt-bad}]
    The result is immediate from
    Propositions~\ref{prop:annealed-kac-rice} and \ref{prop:variational-maximum} for strictly super-solvable $\xi$.
    The proof of \eqref{eq:crt-tot} for solvable $\xi$ follows since $\xi\mapsto\sup_{\vx\in\bbR}F_{\xi}(\vx)$ is continuous at any non-degenerate $\xi$. Indeed $F$ is locally uniformly continuous in non-degenerate $\xi$ on compact $\vx$-sets by Proposition~\ref{prop:Psi-continuous}.
\end{proof}

We also prove the following fact which will be useful in later sections.

\begin{lemma}
    \label{lem:spec-no-0}
    If $\xi$ is strictly super-solvable, there exists $\eps > 0$ such that for any $\vDelta \in \{-1,1\}^r$, $[-\eps,\eps] \cap S(\vDelta) = \emptyset$.
\end{lemma}

\subsection{Stationarity Condition}
\label{subsec:stationarity}

We next identify all stationary points of $F$.
It will be convenient to perform the below derivative calculations in the variable $\vv = \Lambda^{1/2} \vx$ (recall \eqref{eq:Lambda-and-A}).
To this end, we define:
\begin{align*}
    \oF(\vv) &= F(\vx),
    \\
    \oPsi(\vv) &= \Psi(\vx),
    \\
    \omu(\vv) &= \mu(\vx),
    \\
    u_s(z;\vv) &= m_s(z;\vx),
    \\
    \vu(z;\vv) &= \vm(z;\vx).
\end{align*}
Here we recall $m_s(z;\vx)$ is defined above \eqref{eq:dyson-equation} and define $\vm(z;\vx) = (m_s(z;\vx))_{s\in \sS}$.
Thus,
\begin{equation}
    \label{eq:def-oF}
    \oF(\vv)
    =
    \fr12
    \Big(
    1-\sum_{s\in \sS} \lambda_s \log \xi^s(\vone) - \tnorm{A^{-1/2} \vv}_2^2
    \Big)
    + \oPsi(\vv)\,.
\end{equation}
The Dyson equation \eqref{eq:dyson-equation} is equivalent (after some rearrangement) to
\begin{equation}
    \label{eq:dyson-equation-1}
    \lambda_s z + v_s = -\fr{\lambda_s}{u_s(z;\vv)} - \sum_{s'\in \sS} \xi''_{s,s'} u_{s'}(z;\vv)\,.
\end{equation}
The next lemma gives exact formulas for $\oPsi,\oF$ and their gradients.
These seem to be new and extend known results in the single-species case (see e.g. \eqref{eq:Theta-formula} and the discussion below). We believe they are of independent interest, and might lead to more explicit thresholds in e.g. \cite[Theorem 2.5]{mckenna2021complexity}.
(This would still require optimization over the complicated set of vectors $\vu$ corresponding to some $\vv\in\bbR^r$; see Lemma~\ref{lem:u-manifold} below.)
The majority of the proof is carried out in Appendix~\ref{app:dyson}.

Note that below and throughout, we always use $\la \va,\vb\ra = \sum_{s=1}^r a_s b_s$ to denote a bilinear form rather than a complex inner product, even when $\va,\vb$ are complex vectors.
Also recall that $\Re(\cdot)$ denotes the real part of a complex number or vector.
\begin{lemma}
    \label{lem:oF-oPsi-formulas}
    The functions $\oPsi,\oF$ are $C^1$ and satisfy, with $\vu = \vu(0;\vv)$,
    \begin{align}
    \label{eq:oPsi-formula}
        \oPsi(\vv)
        &=
        \fr12 \Re(\la \vu,\xi'' \vu\ra)
        -
        \sum_{s\in\sS}\lambda_s \log|u_s|
        ,
        \\
    \label{eq:oF-formula}
        \oF(\vv)
        &=
        \frac{1}{2}
        \lt(1-\sum_{s\in \sS} \lambda_s \log \xi^s(\vone)
        - \la \vv, A^{-1} \vv\ra
        +
        \Re(\la \vu,\xi'' \vu\ra)
        \rt)
        -
        \sum_{s\in\sS}\lambda_s \log|u_s|,
        \\
    \label{eq:oPsi-derivative}
        \nabla \oPsi(\vv) &= -\Re(\vu),
        \\
    \label{eq:oF-derivative}
        \nabla \oF(\vv) &= - A^{-1} \vv - \Re(\vu).
    \end{align}
\end{lemma}
\begin{proof}
    The formulas \eqref{eq:oPsi-formula} and \eqref{eq:oPsi-derivative} follow from Theorem~\ref{thm:logdet} and Lemma~\ref{lem:oPsi-derivative}.
    Then \eqref{eq:oF-formula} and \eqref{eq:oF-derivative} follow as straightforward consequences.
\end{proof}

For the rest of this section, we let $\vu = \vu(0;\vv) \in \obbH^r$.
Note that \eqref{eq:dyson-equation-1}, specialized to $z = 0$, gives
\begin{equation}
    \label{eq:dyson-equation-2}
    v_s = -\fr{\lambda_s}{u_s} - \sum_{s'\in \sS} \xi''_{s,s'}u_{s'}\,.
\end{equation}
We next describe the condition for $\vv$ to be a stationary point of $\oF$.
\begin{lemma}
    \label{lem:stationary-condition}
    If $\nabla \oF(\vv) = \vzero$, then for all $s\in \sS$, either $\Re(u_s)=0$ or $|u_s| = \sqrt{\lambda_s/\xi'_s}=1/\sqrt{\xi^s(\vone)}$ (recall \eqref{eq:def-xi-s}).
\end{lemma}
\begin{proof}
    We have
    \[
        \fr{\lambda_s}{u_s} + \sum_{s'\in \sS} \xi''_{s,s'}u_{s'}
        = -v_s
        = \Re(A\vu)_s
        = \Re\lt(\xi'_s u_s + \sum_{s'\in \sS} \xi''_{s,s'}u_{s'}\rt)\,,
    \]
    where the first equality is \eqref{eq:dyson-equation-2}, the second is \eqref{eq:oF-derivative}, and the third is the definition~\eqref{eq:Lambda-and-A} of $A$.
    Taking real parts of both sides implies $\Re(\lambda_s/u_s) = \Re(\xi'_s u_s)$, which implies the conclusion.
\end{proof}
We will see that the maximizers of $F$ described in Proposition~\ref{prop:variational-maximum} correspond to $u_s = \pm 1/\sqrt{\xi^s(\vone)}$.
The primary remaining difficulty is to show all \emph{other} remaining stationary points are \emph{not} local maxima.

\subsection{Non-Maximality of Stationary Points with Pure-Imaginary $u_s$}

The following main result of this subsection rules out the first case identified in Lemma~\ref{lem:stationary-condition} for strictly super-solvable $\xi$.
From it, we will easily conclude (in Corollary~\ref{cor:only-maximizers} below) that all maximizers of $F$ are as in Proposition~\ref{prop:variational-maximum}.
\begin{proposition}
    \label{prop:stationary-point-not-local-max}
    Suppose $\xi$ is strictly super-solvable.
    If $\nabla \oF(\vv) = \vzero$ and $\Re(u_s)=0$ for some $s\in \sS$, then $\vv$ is not a local maximum of $\oF$.
\end{proposition}
We will need as input from Appendix~\ref{app:dyson} the following two lemmas.
For $\vu \in \obbH^r$ define the matrices
\begin{align}
    \label{eq:def-oM}
    M(\vu)   &= \diag\lt(\fr{\lambda_s}{ u_s^2 }\rt)_{s\in \sS} - \xi'', &
    \oM(\vu) &= \diag\lt(\fr{\lambda_s}{|u_s|^2}\rt)_{s\in \sS} - \xi''.
\end{align}
\begin{lemma}
    \label{lem:vu-derivative-application}
    At all $\vv \in \bbR^r$ such that $M(\vu)$ is invertible, the function $\vv\mapsto \vu(0;\vv)$ is differentiable and $\nabla_{\vv} \vu(0;\vv) = M(\vu)^{-1}$.
\end{lemma}
\begin{proof}
    Follows from Lemma~\ref{lem:vu-derivative}.
\end{proof}
The equation \eqref{eq:dyson-equation-2} relates $\vv$ to its associated $\vu$, which is well-defined by Proposition~\ref{prop:MDE-basic}.
Because $\vv \in \bbR^r$ while $\vu \in \obbH^r$, one roughly expects that those $\vu$ corresponding to some $\vv \in \bbR^r$ lie within an $r$-dimensional real submanifold of $\obbH^r$.
The next lemma describes this set of $\vu$.

\begin{lemma}
    \label{lem:u-manifold}
    Let $\vu^* \in \obbH^r$.
    There exists $\vv \in \bbR^r$ such that $\vu^* = \vu(0;\vv)$ if and only if one of the following conditions holds.
    \begin{enumerate}[label=(\roman*)]
        \item \label{itm:semicircle-flat} $\vu^* \in \bbR^r$ and $M(\vu^*) \succeq 0$.
        \item \label{itm:semicircle-dome} $\vu^* \in \bbH^r$, $\oM(\vu^*) \succeq 0$, and $\oM(\vu^*) \Im(\vu^*) = 0$.
    \end{enumerate}
    Moreover, in case \ref{itm:semicircle-dome}, $M(\vu^*)$ is invertible.
\end{lemma}
\begin{proof}
    Follows from Theorem~\ref{thm:feasible-region}\ref{itm:feasible-boundary} and Corollary~\ref{cor:feasible-region-M-inv}.
\end{proof}

\begin{lemma}
    \label{lem:oM-singular-u-bd}
    Suppose $\xi$ is strictly super-solvable.
    If $\oM(\vu)$ is singular, then $|u_s| > 1/\sqrt{\xi^s(\vone)}$ for some $s\in \sS$.
\end{lemma}
\begin{proof}
    Suppose otherwise; then $\diag(\lambda_s/|u_s|^2)_{s\in \sS} \succeq \diag(\xi')$, so
    \[
        \oM(\vu) \succeq \diag(\xi') - \xi'' \succ 0\,.
    \]
    However, $\oM(\vu)$ is singular, contradiction.
\end{proof}
Further define
\[
    \hM(\vu) = \diag\lt(\fr{\lambda_s}{|u_s|^2}\rt)_{s\in \sS} + \xi''.
\]
\begin{lemma}
    \label{lem:hM-pd}
    If $\oM(\vu) \succeq 0$, then $\hM(\vu) \succ 0$.
\end{lemma}
\begin{proof}
    Let $D_{\xi''}$ be the diagonal matrix with $(s,s)$ entry $\xi''_{s,s}$.
    We will apply Lemma~\ref{lem:psd-sign-flip} with matrices $\hM(\vu) - 2D_{\xi''}$ and $\oM(\vu)$.
    Note that the diagonal entries of both matrices coincide, as
    \[
        (\hM(\vu) - 2D_{\xi''})_{s,s}
        = \oM(\vu)_{s,s}
        = \fr{\lambda_s}{|u_s|^2} - \xi''_{s,s}
    \]
    and the off-diagonal entries are related by
    \[
        (\hM(\vu) - 2D_{\xi''})_{s,s'}
        = \xi''_{s,s'}
        = |\oM(\vu)_{s,s'}|.
    \]
    Lemma~\ref{lem:psd-sign-flip} thus implies $\hM(\vu) - 2D_{\xi''} \succeq 0$, which by Assumption~\ref{as:nondegenerate} implies the result.
\end{proof}

\begin{proof}[Proof of Proposition~\ref{prop:stationary-point-not-local-max}]
    The hypothesis $\Re(u_s) = 0$ for some $s$ implies that $\vu\notin\bbR^r$, since $u_s\neq 0$ by Proposition~\ref{prop:MDE-basic}.
    Therefore Lemma~\ref{lem:u-manifold} case~\ref{itm:semicircle-dome} applies, so $\oM(\vu) \succeq 0$ is singular and $M(\vu)$ is invertible.
    Differentiation of \eqref{eq:oF-derivative} using Lemma~\ref{lem:vu-derivative-application} then gives
    \begin{equation}
        \label{eq:oF''}
        \nabla^2 \oF(\vv) = -A^{-1} - \Re(M(\vu)^{-1})\,.
    \end{equation}

    Let $I\subseteq [r]$ be the set of indices $s$ with $|u_s| \neq 1/\sqrt{\xi^s(\vone)}$, which is nonempty by Lemma~\ref{lem:oM-singular-u-bd}.
    By Lemma~\ref{lem:stationary-condition}, we have $\Re(u_s)=0$ for all $s\in I$.
    Moreover, Lemma~\ref{lem:hM-pd} implies $\hM(\vu) \succ 0$.
    We will construct a vector $\vw \in \bbR^r$ such that $\vw^\top (\nabla^2 \oF(\vv)) \vw > 0$, which implies $\vv$ is not a local maximum.
    We work in the subspace $\bbR^I \subseteq \bbR^r$, consider $\va \in \bbR^I$ to be chosen later, and set
    \[
        \vw = -M(\vu) \va = \hM(\vu) \va.
    \]
    Here the second equality uses that $\va \in \bbR^I$ and that $u_s$ is pure imaginary for $s\in I$.
    Importantly, all entries of $\vw$ are real.

    Abbreviate $M = M(\vu)$, $\oM = \oM(\vu)$, $\hM = \hM(\vu)$ and let $D = A - \hM$.
    Note that $D$ is diagonal and its entry $D_{s,s} = \xi'_s - \fr{\lambda_s}{|u_s|^2}$ is nonzero if and only if $s\in I$.
    Let $D^\dagger$ denote the Moore--Penrose inverse of $M$ and $P_I = \sum_{s\in I} \ve_s \ve_s^\top$ be the projection onto $I$.
    Then
    \begin{align}
        \notag
        \vw^\top (\nabla^2 \oF(\vv)) \vw
        &= -\vw^\top A^{-1} \vw - \Re(\vw^\top M^{-1} \vw) \\
        \notag
        &= -\va^\top \hM A^{-1} \hM \va - \Re(\va^\top M \va) \\
        \notag
        &= -\va^\top \hM A^{-1} \hM \va + \va^\top \hM \va \\
        \notag
        &= -\va^\top (A-D) A^{-1} (A-D) \va + \va^\top (A-D) \va \\
        \notag
        &= \va^\top D(D^\dagger-A^{-1})D\va \\
        \notag
        &= \va^\top D A^{-1} \lt(
                (\hM + D)D^\dagger(\hM + D) - (\hM + D)
        \rt) A^{-1} D \va \\
        \label{eq:w-F-w-calc}
        &= \va^\top D A^{-1} \lt(
            \hM D^\dagger \hM
            + \hM P_I + P_I \hM - \hM
        \rt) A^{-1} D \va\,.
    \end{align}
    Recall that $A,\hM$ agree on rows indexed by $[r]\setminus I$.
    So, if $\vy \in \bbR^I$ and $\vz = A^{-1} \vy$, then for any $s\in\sS\setminus I$,
    \[
        (\hM \vz)_s = (A \vz)_s = y_s = 0.
    \]
    Thus $\bbR^I$ is an invariant subspace for $\hM A^{-1}$, i.e. $\hM A^{-1}\bbR^I\subseteq\bbR^I$.
    Since $A$ and $\hM$ are both full rank (by Corollary~\ref{cor:A-pd} and Lemma~\ref{lem:hM-pd}), in fact $\hM A^{-1}\bbR^I=\bbR^I$ is a bijection on $\bbR^I$, and the same holds for $\hM A^{-1} D$.

    Since we showed earlier in this proof that $\oM$ is singular, Lemma~\ref{lem:oM-singular-u-bd} implies that there exists $s\in I$ such that $D_{s,s} > 0$.
    Using the bijectivity just established, we choose $\va$ such that $\hM A^{-1} D \va = \ve_s$. Then
    \[
        \va^\top D A^{-1} \hM D^\dagger \hM A^{-1} D \va
        = \ve_s^\top D^\dagger \ve_s = D_{s,s}^\dagger > 0\,.
    \]
    Since $s\in I$ we further have
    \begin{align*}
        \va^\top D A^{-1} \lt(
            \hM P_I + P_I \hM - \hM
        \rt) A^{-1} D \va
        &= \ve_s^\top P_I \hM^{-1} \ve_s
        + \ve_s^\top \hM^{-1} P_I \ve_s
        - \ve_s^\top \hM^{-1} \ve_s \\
        &= \ve_s^\top \hM^{-1} \ve_s > 0\,.
    \end{align*}
    Summing and recalling \eqref{eq:w-F-w-calc}, we conclude that $\vw^\top (\nabla^2 \oF(\vv)) \vw > 0$ as desired.
\end{proof}

\begin{corollary}
    \label{cor:only-maximizers}
    If $\vv \in \bbR^r$ maximizes $\oF$, then $\vv = \vv(\vDelta) \equiv \Lambda^{1/2} \vx(\vDelta)$ for some $\vDelta \in \{-1,1\}^r$.
    Conversely, each $\vv(\vDelta)$ is a stationary point of $\oF$ with $\oF(\vv(\vDelta))=0$.
\end{corollary}
\begin{proof}
    Lemma~\ref{lem:F-regularity} implies that if $\vv$ maximizes $\oF$ then it is a stationary point.
    Lemma~\ref{lem:stationary-condition} and Proposition~\ref{prop:stationary-point-not-local-max} imply that $|u_s| = 1/\sqrt{\xi^s(\vone)}$ for all $s\in \sS$.
    Thus $\oM(\vu) = \diag(\xi') - \xi'' \succ 0$ is not singular.
    By Lemma~\ref{lem:u-manifold}, we have $\vu \in \bbR^r$, so $u_s = \pm 1/\sqrt{\xi^s(\vone)}$.
    The $2^r$ possible choices of $\vu$ are indexed by $\vDelta \in \{-1,1\}^r$ and given by
    \begin{equation}
        \label{eq:def-u-Delta}
        u(\vDelta)_s = -\Delta_s / \sqrt{\xi^s(\vone)}\,.
    \end{equation}
    Substituting $\vu(\vDelta)$ into \eqref{eq:dyson-equation-2} shows that $\vv = \vv(\vDelta)$.
    For the converse, note that $M(\vu(\vDelta)) = \diag(\xi') - \xi'' \succ 0$, so Lemma~\ref{lem:u-manifold} case \ref{itm:semicircle-flat} implies that $\vu(\vDelta) = \vu(0;\vv(\vDelta))$.
    We can verify from the formulas for $\vu(\vDelta)$ and $\vv(\vDelta)$ that
    \begin{equation}
        \label{eq:vv-to-vu}
        \vv(\vDelta) = -A\vu(\vDelta)\,.
    \end{equation}
    Thus, by Lemma~\ref{lem:stationary-condition},
    \[
        \nabla \oF(\vv(\vDelta))
        = -A^{-1} \vv(\vDelta)-\Re(\vu(\vDelta)) = 0\,,
    \]
    so $\vv(\vDelta)$ is a stationary point.

    Finally, we verify that $\oF(\vv(\vDelta))=0$ for all $\vDelta$ by directly using \eqref{eq:oF-formula}. Since $\vu(\vDelta)\in\bbR^r$, the quadratic terms combine to give:
    \begin{align*}
    \frac{1}{2}\lt(\la \vu(\vDelta),\xi''\vu(\vDelta)\ra
    -
    \la \vv(\vDelta),A^{-1}\vv(\vDelta)\ra\rt)
    &\stackrel{\eqref{eq:vv-to-vu}}{=}
    \la \vu(\vDelta),(\xi''-A)\vu(\vDelta)\ra/2
    \\
    &\stackrel{\eqref{eq:Lambda-and-A}}{=}
    -\la \vu(\vDelta),\diag(\xi')\vu(\vDelta)\ra/2
    \\
    &\stackrel{\eqref{eq:def-u-Delta}}{=}
    -1/2.
    \end{align*}
    This cancels the first term in \eqref{eq:oF-formula}.
    Meanwhile recalling \eqref{eq:def-u-Delta}, the logarithmic terms give
    \[
    -\sum_{s\in\sS}
    \lambda_s
    \log\big( |u_s|\sqrt{\xi^s(\vone)}\big)
    =
    0.
    \]
    Combining completes the proof.
\end{proof}

\begin{proof}[Proof of Proposition~\ref{prop:variational-maximum}]
    Lemma~\ref{lem:F-regularity} implies that $\oF$ possesses at least one global maximizer. The preceding results imply that the only possibilities are the $2^r$ points $\vv(\vDelta)$, and we have just computed $\oF(\vv(\vDelta))=0$ for all $\vDelta$. This completes the proof.
\end{proof}

\begin{proof}[Proof of Lemma~\ref{lem:spec-no-0}]
    By Proposition~\ref{prop:MDE-basic} and e.g. \cite[Chapter 2.4]{Guionnet}, $\mu(\vx(\vDelta))$ has piecewise smooth density given by
    \[
        \rho(\gamma)
        = \fr{1}{\pi} \Im(m(\gamma;\vx(\vDelta)))
        = \fr{1}{\pi} \Im(u(\gamma;\vv(\vDelta)))\,,
    \]
    where $u(\gamma;\vv(\vDelta)) = \sum_s \lambda_s u_s(\gamma;\vv(\vDelta))$.
    So, it suffices to show $\vu(\gamma;\vv(\vDelta))$ is real for all $|\gamma|\le \eps$.
    It is clear from \eqref{eq:dyson-equation-1} that
    \[
        \vu(\gamma;\vv(\vDelta)) = \vu(0;\vv(\vDelta)+\gamma\vlambda)\,.
    \]
    Recall $M(\vu(\vDelta)) = \diag(\xi') - \xi'' \succ 0$.
    Thus, $M(\vu) \succ 0$ for $\vu$ in an open neighborhood $\cN \subseteq \bbR^r$ of $\vu(\vDelta)$. Hence solving \eqref{eq:dyson-equation-2} near $\vu(\vDelta)$ via inverse function theorem bijectively maps $\cN$ to an open neighborhood $\cN' \subseteq \bbR^r$ of $\vv(\vDelta)$.
    By Lemma~\ref{lem:u-manifold} case \ref{itm:semicircle-flat}, if $\vu \in \cN$ maps to $\vv \in \cN'$ under \eqref{eq:dyson-equation-2}, then $\vu = \vu(0;\vv)$.
    In particular, for suitably small $\eps > 0$, we have $\vv(\vDelta) + \gamma\vlambda \in \cN'$ for all $|\gamma| \le \eps$.
    Thus $\vu(0;\vv(\vDelta)+\gamma\vlambda)$ is real.
\end{proof}

\begin{figure}[h!]
    \centering
    \begin{framed}
    \begin{subfigure}[b]{.5\textwidth}
        \includegraphics[width=.9\linewidth]{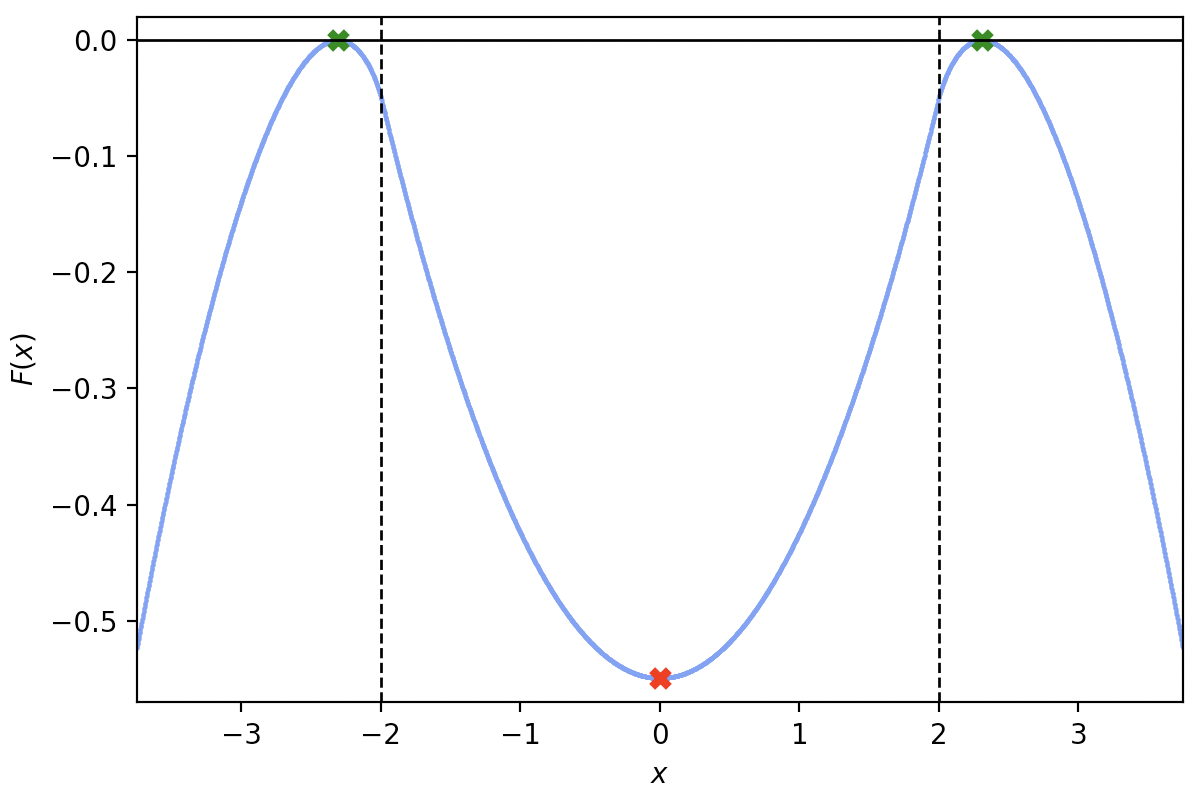}
        \caption{One species, $\xi' = 3$, $\xi''=1$}
        \label{subfig:one-species}
    \end{subfigure} \\
    \begin{subfigure}[b]{.47\textwidth}
        \includegraphics[width=.9\linewidth]{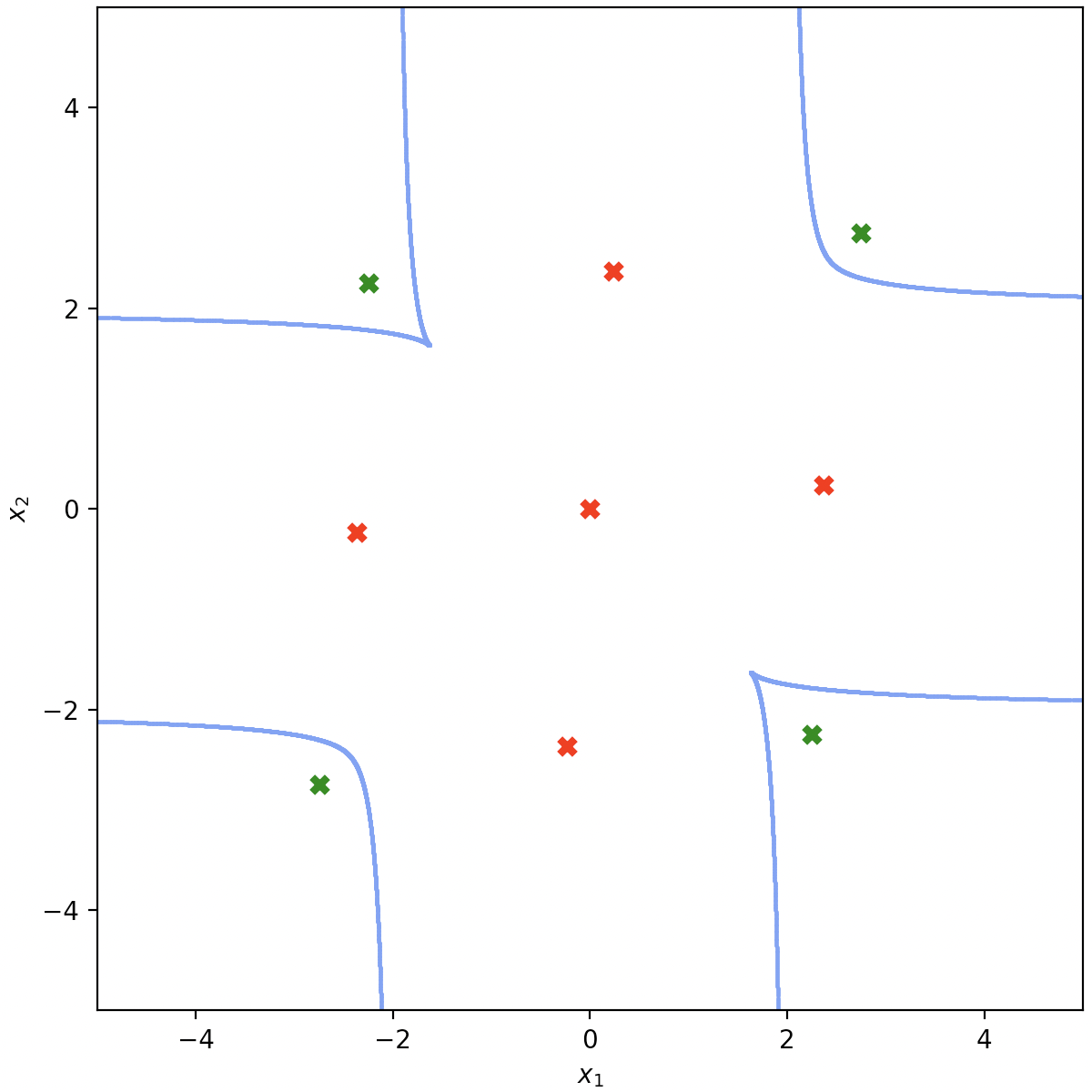}
        \caption{Two species, $\vlambda = (0.5,0.5)$, $\xi' = (4,4)$, $\xi'' = \lt(\begin{smallmatrix}1 & 0.5 \\ 0.5 & 1\end{smallmatrix}\rt)$}
        \label{subfig:two-species}
    \end{subfigure}
    \qquad
    \begin{subfigure}[b]{.47\textwidth}
        \includegraphics[width=.9\linewidth]{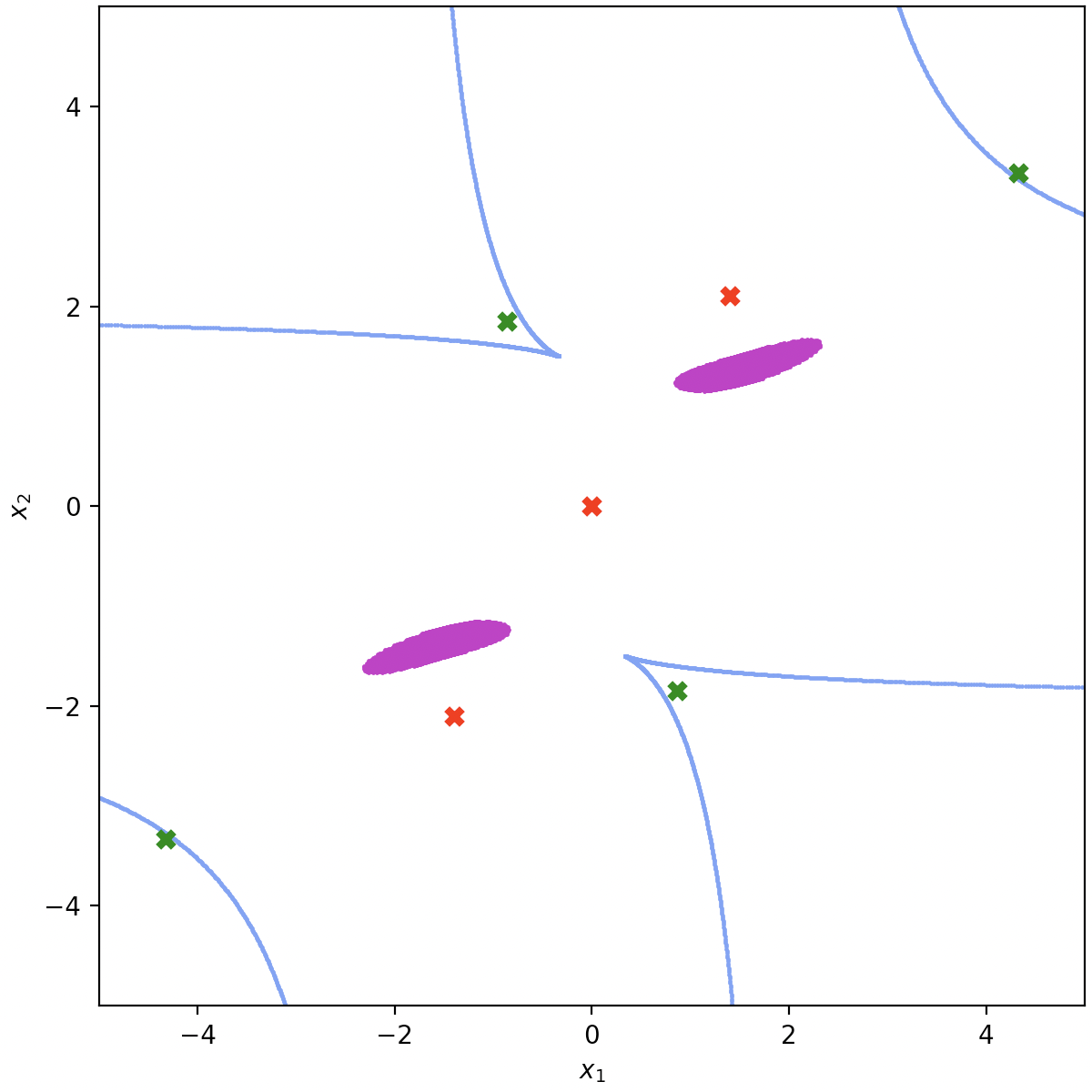}
        \caption{Two species, $\vlambda = (0.3,0.7)$, $\xi'=(4.5,4.5)$, $\xi'' = \lt(\begin{smallmatrix}1 & 2.4 \\ 2.4 & 1\end{smallmatrix}\rt)$}
        \label{subfig:two-species-weird}
    \end{subfigure}
    \caption{Figure~\ref{subfig:one-species}: the complexity functional $F$ of a $1$-species model is shown. $F$ is tangent to the $x$-axis at two global maxima marked by green X's. The red X is a local minimum. The two dashed vertical lines mark the transition from local convexity to concavity, and $F''$ is discontinuous at these points. \\~\\
    Figures~\ref{subfig:two-species} and \ref{subfig:two-species-weird}: points of interest are shown in the domain $\bbR^2$ of the complexity functionals $F$ for two different $2$-species models. The green X's are global maxima where $F$ equals $0$, while the red X's are stationary points that are not local maxima. \\
    The blue boundary is analogous to the dashed vertical lines in Figure~\ref{subfig:one-species}, and is where $\vm(0;\vx)$ transitions from real and nonreal.
    In the four regions outside this boundary, $\vm(0;\vx)$ is real, and in the region inside it $\vm(0;\vx)$ is non-real.
    By Lemma~\ref{lem:u-manifold} and continuity of $\vx \mapsto \vm(0;\vx)$ (see Theorem~\ref{thm:continuity}), this boundary is also the set of $\vx$ for which $\vm(0;\vx)$ is real and $M(\vm(0;\vx))$ is singular. \\
    In Figure~\ref{subfig:two-species}, $F$ is locally non-concave inside this boundary, but in Figure~\ref{subfig:two-species-weird} $F$ is also locally concave in the shaded purple regions. Note also that in Figure~\ref{subfig:two-species-weird}, there are only three red X's instead of five; the $3^r$ stationary points identified by Lemma~\ref{lem:stationary-condition} do not necessarily all exist.
    }
    \label{fig:complexity-functional}
    \end{framed}
\end{figure}

\subsection{Discussion of Proof Technique}
\label{subsec:proof-discussion}

Lemma~\ref{lem:stationary-condition} identifies approximately $3^r$ stationary points of $F$: for each species $s$, we may choose whether $\Re(u_s)=0$, $|u_s| = 1/\sqrt{\xi^s(\vone)}$ has positive real part, or $|u_s| = 1/\sqrt{\xi^s(\vone)}$ has negative real part (though these do not always all exist, see Figure~\ref{subfig:two-species-weird}).
As we saw in the above proof, the $2^r$ stationary points where $u_s = \pm 1/\sqrt{\xi^s(\vone)}$ are global maximizers and the rest are saddle points of index at least $1$.
The main task in the proof of Proposition~\ref{prop:variational-maximum} was to rule out the extraneous critical points; in this subsection we motivate our method for doing so.

In the single-species case $r=1$, in the topologically trivial regime $F$ is convex on the interval $[-2\sqrt{\xi''}, 2\sqrt{\xi''}]$ and concave on its complement, and the maximum $0$ is attained in the latter set; see Figure~\ref{subfig:one-species}.
The boundary points $\pm 2\sqrt{\xi''}$ correspond to the radial derivative values where $0$ enters or exits the limiting bulk spectrum of $\nabla^2_\sph H(\bx)$, and can be detected by the Stieltjes transform $m(0;x)$ becoming non-real.
(In fact, $F''$ is discontinuous at these points.)
This characterization of the convexity and concavity of $F$ allows us to easily identify which of the critical points given by Lemma~\ref{lem:stationary-condition} are local maximizers.

However, a similar ``region by region" convexity analysis will not work with multiple species.
Similarly to the one-species case, $\nabla^2 F$ is discontinuous on a surface of radial derivative vectors $\vx$ where $0$ enters or exits the limiting bulk spectrum of $\nabla^2_\sph H(\bx)$, and this can be detected by $\vm(0;\vx)$ (or equivalently $\vu(0;\vv)$) becoming non-real.
This boundary divides $\bbR^r$ into several regions and is depicted in Figures~\ref{subfig:two-species} and \ref{subfig:two-species-weird} as the blue curve.
A natural approach to ruling out the extraneous critical points would be to show that, analogously to above, $F$ is locally concave outside this boundary (in the regions containing the $2^r$ true maxima) and locally nonconcave inside it.
However, this characterization is surprisingly not true.
While $F$ is indeed locally concave outside the boundary --- if $M(\vu) \succ 0$, then \eqref{eq:oF''} implies $\nabla^2 \oF(\vv) \preceq 0$ --- it is possible for $F$ to also be locally concave inside it, for example in the purple regions in Figure~\ref{subfig:two-species-weird}.

This counterexample rules out attempts to argue globally about convexity. This led us to the more direct approach of finding, at each extraneous critical point, a direction along which $\nabla^2 F$ is positive.

\section{Approximate Critical Point Control from Kac--Rice Estimates}
\label{sec:approximate}

By Markov's inequality, negativity of the annealed Kac--Rice estimate \eqref{eq:annealed-kac-rice-atypical} implies that $H_N$ has no $\eps$-atypical critical points (with high probability).
The following main result of this section shows this implication is robust in some sense: the same Kac--Rice estimate also implies non-existence of certain \emph{approximate} critical points.

\begin{proposition}
\label{prop:approx-crits}
    For strictly super-solvable $\xi$ and any $\ups>0$ there exists $\eps=\eps(\xi,\ups)$ such that with probability $1-e^{-cN}$, all $\eps$-approximate critical points are $\ups$-good (recall Definition~\ref{def:good}).
\end{proposition}

Our approach proceeds as follows.
Given $H_N$ and $\delta>0$, define the rerandomization
\begin{equation}
\label{eq:H-N-delta}
    H_{N,\delta}(\bx)
    =
    \sqrt{1-\delta}\, H_N(\bx)
    +
    \sqrt{\delta}\, H_N'(\bx)
\end{equation}
for $H_N'$ an independent copy of $H_N$.
Let $\Crt_{N,\delta}$ be the set of critical points for $H_{N,\delta}$.
Our goal will be to show that if $\|\nabla_{\sph} H_N(\bx)\|_2 \leq \eps\sqrt{N}$ and $H_N$ lies in a typical set (of probability $1-e^{-cN}$), then for suitable $\delta,\iota$ tending to $0$ with $\eps$,
\[
    \bbE\lt[|\Crt_{N,\delta}\cap B_{\iota\sqrt{N}}(\bx)|~\big|~H_N\rt]\geq e^{-o_{\delta}(N)}
\]
This implies that if $H_N$ has an $\eps$-critical point that is not $\ups$-good (for some $\ups$ also tending to $0$ with $\eps$), then the rerandomized Hamiltonian $H_{N,\delta}$ has on average at least $e^{-o_{\delta}(N)}$ critical points which are not $\ups/2$-good.
Combining with Theorem~\ref{thm:annealed-crits}, which shows the number of such critical points is exponentially small, will then yield Proposition~\ref{prop:approx-crits}.

In fact, this argument proves the following much more general result, which we believe is of significant independent interest.
Let $\cJ$ consist of all compact subsets of $\bbR$, equipped with the Hausdorff metric. We consider a non-empty subset
\[
\ocD\subseteq \bbR^r\times \bbR^r\times \bbR\times \cJ\times \bbW_1(\bbR),
\]
where the right-hand product is equipped with the supremum metric over its five factors.

Given $\iota,\eps\geq 0$, we define the set $\Crt_N^{\ocD,\eps,\ups}(H_N)\subseteq\cS_N$ of $\eps$-approximate critical points $\bx\in\cS_N$ for $H_N$ which are $\ups$-far from being described by an element of $\ocD$, i.e. which satisfy
\begin{equation}
\label{eq:ups-far-from-ocD}
    d\Big(
    \Big(
    \nabla_{\rd}H_N(\bx),
    \vR(\bx,\bG^{(1)}),
    \frac{H_N(\bx)}{N},
    \spec_{H_N}(\bx),
    \wh\mu_{H_N}(\bx)
    \Big)
    ,\ocD
    \Big)
    \geq \ups.
\end{equation}
Recall \eqref{eq:Hessian-spectral-defs} for definitions of $\spec_{H_N}(\cdot)$ and $\wh\mu_{H_N}(\cdot)$.
As usual, the distance from a point to a set is the infimal point-to-point distance; recall also the definition of $\nabla_{\rd}$ near Fact~\ref{fac:riemannian-to-euclidean}.
Note that $\Crt_N^{\ocD,\eps,\iota}(H_N)$ is an infinite set with positive probability unless $\eps=0$.

\begin{theorem}
\label{thm:approx-crits-from-annealed}
    Suppose $\xi,\ocD,\eps,\ups,c_0$ are such that for $N$ large enough,
    \begin{equation}
    \label{eq:annealed-crits-close-to-ocD}
        \bbE \big|\Crt_N^{\ocD,0,\ups/2}(H_N)\big|\leq e^{-c_0 N}.
    \end{equation}
    Then for some small $\eps>0$ depending only on $(\xi,\ups,c_0)$, for some $c>0$ and all $N$ large enough:
    \[
        \bbP\big[\big|\Crt_N^{\ocD,\eps,\ups}(H_N)\big|\geq 1\big]\leq e^{-cN}.
    \]
\end{theorem}

\begin{proof}[Proof of Proposition~\ref{prop:approx-crits} from Theorem~\ref{thm:approx-crits-from-annealed}:]
    Let $\ocD$ be the size $2^r$ set:
    \[
    \ocD
    =
    \ocD(\xi)
    =
    \lt\{
    \Big(
    \vx(\vDelta),
    \vR(\vDelta),
    E(\vDelta),
    S(\vDelta),
    \mu(\vx(\vDelta))
    \Big)
    ~:~\vDelta\in \{-1,1\}^r\rt\}.
    \]
    Then Theorem~\ref{thm:annealed-crits}\ref{itm:crt-bad} implies the condition \eqref{eq:annealed-crits-close-to-ocD} for any $\ups>0$, for some correspondingly small $c>0$.
    Therefore we may apply Theorem~\ref{thm:approx-crits-from-annealed} which completes the proof.
\end{proof}

We note that although $\ocD$ is defined in near-maximal generality above, examining just one of its $5$ components yields interesting consequences.
For instance Section~\ref{sec:approx-local-max-E-infty} considers only the radial derivative.

\subsection{Technical Properties of the Conditional Vector Dyson Equation}

In this somewhat technical subsection, we study the vector Dyson equation corresponding to \eqref{eq:H-N-delta}.
Note that when both $H_N$ and $H_N'$ are treated as random, the relevant Dyson equation is exactly as in our usual setting explained in Subsection~\ref{subsec:prelim-rmt}.
However if one first conditions on $H_N$, then \eqref{eq:H-N-delta} yields a different vector Dyson equation with an apriori somewhat different solution that depends on $H_N$.

We show below that with extremely high probability over $H_N$, the latter ``conditional'' solution is uniformly close to that of the unconditional equation.
This is crucial because, since we do wish to condition on $H_N$ in our main argument, we need to apply the asymptotic determinant results from \cite{arous2021exponential} to this conditioned vector Dyson equation.
The main idea is that with or without conditioning, concentration of the empirical spectrum of $H_{N,\delta}$ implies it is well-described by the Dyson equation's solution with very high probability.
Hence these solutions must in fact be similar, with high probability over $H_N$.

Recall the definition of $C$-regular probability measure from Definition~\ref{def:C-regular-mu}.
We next define a class of $C$-regular random matrix models, which have $C'$-regular spectral measure in some sense; see Lemma~\ref{lem:local-law-boosted} just below.

\begin{definition}
\label{def:C-regular-M}
    The (law of the) random symmetric matrix $M_N\in\bbR^{N\times N}$ is \textbf{$C$-regular} if it is given by $M_N=W_N+A_N$ where:
    \begin{enumerate}[label=(\alph*)]
        \item $A_N$ is deterministic and $\|A_N\|_{\op}\leq C$.
        \item $W_N$ is a centered Gaussian matrix with independent entries on and above the diagonal.
        \item Each entry of $W_N$ has variance in $\lt[\frac{1}{CN},\frac{C}{N}\rt]$.
    \end{enumerate}
\end{definition}

Given a $C$-regular random matrix $M_N$, for each $z\in\bbH$ we let $G_N(z)\in\bbC^{N\times N}$ be the unique solution to the equation
\begin{equation}
\label{eq:MDE}
    I_N + (zI_N - A_N + \bbE[W_N G_N(z)W_N])G_N(z)=0
\end{equation}
with the constraint that the imaginary part $\Im (G_N(z))$ is a strictly positive definite matrix.
We let $\mu_{M_N}\in\cP(\bbR)$ be the (unique) probability measure with Stieltjes transform $\Tr(G_N(z))/N$.
Such $G_N(z)$ and $\mu_{M_N}$ exist and are unique by e.g. \cite[Proposition 5.1~(i),(ii)]{erdos2019random}.

\begin{lemma}
\label{lem:local-law-boosted}
    If $M_N$ is $C_1$-regular, then $\mu_{M_N}$ is $C_2$-regular for $C_2$ depending only on $C_1$.
    Further, for any event $E$ with $\bbP[E]\geq 1/2$,
    \begin{align}
    \label{eq:W1-approx-MDE}
    \bbW_1\big(
    \mu_{M_N},
    \bbE\big[
    \wh{\mu}_{M_N}
    \big]
    \big)
    &\leq
    \delta_N,
    \\
    \label{eq:shifted-determinant-formula}
    \lt|
    \fr1N \log \bbE \lt[
        1_E\cdot
        |\det M_N|
    \rt]
    -
    \int \log|\lambda| \de\mu_{M_N}(\lambda)
    \rt|
    &\leq \delta_N
    \end{align}
    for a sequence $\delta_N\to 0$ depending only on $C_1$.
\end{lemma}

\begin{proof}
    The first assertion follows from \cite[Theorem 2.6]{ajanki2017singularities}.
    Next, \eqref{eq:W1-approx-MDE} follows from
    \cite[Theorem 2.1(4b)]{erdos2019random} and \cite[Proposition 3.1]{arous2021exponential} as in \cite[Proof of Corollary 1.9.B]{arous2021exponential}.
    The second part \eqref{eq:shifted-determinant-formula} is a rewriting of \cite[Corollary 1.9.A]{arous2021exponential} except for the presence of the event $E$. This additional ingredient follows by the same proof since e.g. at the end of \cite[Proof of Theorem 1.2]{arous2021exponential}, the probabilities of all good events $\cE_{\text{Lip}},\cE_{\text{gap}},\cE_{\text{b}}$ are shown to tend to $1$.
    Indeed as $\bbP[E]\geq 1/2$, the factor $1_E$ only affects said lower bound by an additive $O(1/N)$.
\end{proof}



\begin{lemma}
\label{lem:hoffman-wielandt}
    Suppose $M_N=W_N+A_N$ and $M_N'=W_N'+A_N'$ are $C_1$-regular and $W_N\stackrel{d}{=}W_N'$.
    Then
    \[
    \bbW_1\lt(
        \bbE[\wh{\mu}_{M_N}],
        \bbE[\wh{\mu}_{M_N'}]
    \rt)
    \leq
    \fr{1}{\sqrt{N}} \tnorm{A_N-A_N'}_F.
    \]
\end{lemma}

\begin{proof}
    By coupling $W_N=W_N'$ and then using the Hoffman--Wielandt lemma (see e.g. \cite[Lemma 2.1.19]{Guionnet}), one finds:
    \[
    \bbW_2\lt(
        \bbE[\wh{\mu}_{M_N}],
        \bbE[\wh{\mu}_{M_N'}]
    \rt)^2
    \leq
    \fr{1}{N} \tnorm{A_N-A_N'}_F^2.
    \]
    The Cauchy--Schwarz inequality implies that $\bbW_1$ distance is smaller than $\bbW_2$ distance, completing the proof.
\end{proof}

Given $\mu\in\cP(\bbR)$, let $\mu^{(C)}$ be the pushforward of $\mu$ under $x\mapsto \min(C,\max(-C,x))$.

\begin{lemma}
\label{lem:spectral-concentration}
    Suppose $M_N=W_N+A_N$ is $C_1$-regular. Then for some $\eps_0>0$ and any $C_2,\delta>0$, there exists $C_3=C_3(C_1,C_2,\delta)$ such that with $N$ sufficiently large:
    \[
    \bbP\lt[
    \bbW_1\big(\wh{\mu}_{M_N}^{(C_2)}
    ,\bbE[\wh{\mu}_{M_N}^{(C_2)}]\big)
    \geq\delta
    \rt]
    \leq
    C_3 e^{-c(C_1,C_2,\delta)N^{1+\eps_0}}
    .
    \]
\end{lemma}

\begin{proof}
    For any $1$-Lipschitz test function $f$, it is shown in condition (L) in \cite[Equation (1.11)]{arous2021exponential} that the concentration
    \begin{equation}
    \label{eq:the-concentration}
    \bbP\lt[
    \lt|
    \bbE^{\wh\mu_{M_N}}[f]
    -
    \bbE\big[\bbE^{\wh\mu_{M_N}}[f]\big]
    \rt|\geq \delta/10
    \rt]
    \leq
    C_3 e^{-c(C_1,C_2,\delta)N^{1+\eps_0}}
    .
    \end{equation}
    Here our $\eps_0$ is their $\eps_0-\zeta$; in our setting this statement follows by a Herbst argument as in \cite[Proof of Corollary 1.9.B]{arous2021exponential}.
    Since $x\mapsto \min(C_2,\max(-C_2,x))$ is $1$-Lipschitz, the composition $f^{(C_2)}(x)\equiv f\big(\min(C_2,\max(-C_2,x))\big)$ is as well.
    To obtain the claimed Wasserstein bound, recall that the $\bbW_1$ distance makes $\cP([-C_2,C_2])$ a compact metric space, and coincides with the bounded Lipschitz metric:
    \begin{equation}
    \label{eq:kantorovich-rubenstein}
    \bbW_1(\mu,\wt\mu)=\sup_{\substack{f:\bbR\to\bbR\\ Lip(f)\leq 1}} \big(\bbE^{\mu}[f]-\bbE^{\wt\mu}[f]\big).
    \end{equation}
    Hence for any $\delta>0$ and $C_2$, we may choose a finite $\delta/10$-net $\cN\subseteq \cP([-C_2,C_2])$ with respect to $\bbW_1$.
    Then for each distinct pair $\mu_i,\mu_j\in\cN$, we may choose a $1$-Lipschitz $f_{i,j}:\bbR\to\bbR$ such that $|\bbE^{\mu_i}[f_{i,j}]-\bbE^{\mu_j}[f_{i,j}]|=\bbW_1(\mu_i,\mu_j)$.
    Since $|\cN|$ is independent of $N$, the event in \eqref{eq:the-concentration} holds for all $f_{i,j}$ simultaneously with probability $1-C_3' e^{-c(C_1,C_2,\delta)N^{1+\eps_0}}$.
    Finally, we choose $i,j$ so that
    $\bbW_1(\mu_i,\wh\mu_{M_N}^{(C_2)})\leq \delta/10$ and $\bbW_1(\mu_j,\bbE \wh\mu_{M_N}^{(C_2)})\leq \delta/10$.
    On the event that \eqref{eq:the-concentration} applies to $f_{i,j}$, we thus find:
    \begin{align*}
    \bbW_1\big(\wh{\mu}_{M_N}^{(C_2)}
    ,\bbE[\wh{\mu}_{M_N}^{(C_2)}]\big)
    &\leq
    \bbW_1(\wh\mu_{M_N}^{(C_2)},\mu_i)
    +
    \bbW_1\big(\mu_i
    ,\mu_j\big)
    +
    \bbW_1(\mu_j,\bbE\wh\mu_{M_N}^{(C_2)})
    \\
    &\leq
    |\bbE^{\mu_i}[f_{i,j}]-\bbE^{\mu_j}[f_{i,j}]|
    +
    \frac{\delta}{5}
    \\
    &\stackrel{\eqref{eq:kantorovich-rubenstein}}{\leq}
    \big|\bbE^{\wh\mu_{M_N}^{(C_2)}}[f_{i,j}]-\bbE\big[\bbE^{\wh\mu_{M_N}^{(C_2)}}[f_{i,j}]\big]\big|
    +
    \frac{2\delta}{5}
    \\
    &\stackrel{\eqref{eq:the-concentration}}{\leq}
    \delta/2.
    \end{align*}
    This completes the proof, since the value $|\cN|$ depends only on $C_2$ and $\delta$, hence can be absorbed into the value $C_3$.
\end{proof}

For our Kac--Rice application, we will fix some $H_N\in K_N$ (recall Proposition~\ref{prop:gradients-bounded}) and condition also on $\va(\bx)=\nabla_{\rd} H_{N,\delta}(\bx)$, where $H_{N,\delta}(\bx)$ is as in equation \eqref{eq:H-N-delta}.
Let us assume $\tnorm{\va(\bx)}_\infty \le C$, which holds with probability $1-e^{-cN}$ for some constant $C$ by Proposition~\ref{prop:gradients-bounded}.
Then the law of $\nabla^2_{\sph}H_{N,\delta}(\bx)$ conditionally on $\va(\bx)$ is a $C_1$-regular matrix
\[
    \nabla^2_{\sph}H_{N,\delta}(\bx)
    =
    \underbrace{\sqrt{\delta}\, \nabla^2_{\cT \times \cT} H_N'(\bx)}_{W_N}
    +
    \underbrace{
    \sqrt{1-\delta}\,\nabla^2_{\cT \times \cT} H_N(\bx)
    -
    \diag(\Lambda^{-1/2} \va(\bx) \diamond \bone_\cT)
    }
    _{A_N}
    .
\]
We write $\mu_{H_N,\va,\bx}^{\delta}\in\cP(\bbR)$ for the probability measure with Stieltjes transform the corresponding solution to \eqref{eq:MDE} for $\nabla_{\sph}^2 H_{N,\delta}(\bx)$.
Also let $\mu_{\xi,\va}=\mu_{H_N,\va,\bx}^{1}$.
However, we emphasize that $\mu_{H_N,\va,\bx}^{\delta}$ makes sense even for $\va\neq \nabla_{\rd} H_{N,\delta}(\bx)$. In fact in the arguments below, we will obtain control on $\mu_{H_N,\va,\bx}^{\delta}$ by estimating it locally uniformly in $\va$, and only substituting $\va=\va(\bx)$ at the end.
First we show that for \textbf{fixed} $\va$, the measure $\mu_{H_N,\va,\bx}^{\delta}$ concentrates sharply around $\mu_{\xi,\va}$. The idea is to apply Lemma~\ref{lem:spectral-concentration} both before and after conditioning on $H_N$: since it yields concentration of the spectral measure in both cases, they must concentrate around approximately the same measure.

\begin{lemma}
\label{lem:high-prob-MDE}
    There is some $\eps_0>0$ such that for any $C,\delta>0$ there is $C',c>0$ such that the following holds. For each fixed $\bx\in\cS_N$ and fixed $C$-bounded $\va\in \bbR^r$,
    \[
    \bbP\lt[
    H_N\in K_N~~\text{and}~~
    \bbW_1\big(
    \mu_{H_N,\va,\bx}^{\delta}
    ,
    \mu_{\xi,\va}
    \big)
    \geq \delta
    \rt]
    \leq
    C'e^{-cN^{1+\eps_0}}
    \]
    for $N$ sufficiently large.
\end{lemma}

\begin{proof}
    Lemma~\ref{lem:local-law-boosted} shows $\mu_{\xi,\va}$ is $C_2$-regular. Moreover if $H_N\in K_N$ and $\|\va\|_\infty \leq C$, it implies that $\mu_{H_N,\va,\bx}^{\delta}$ is $C_2$-regular.
    Let $M_N^{\delta}\equiv\nabla_{\sph}^2 H_{N,\delta}(\bx)$; then applying \eqref{eq:W1-approx-MDE} conditionally on $H_N$ shows
    \[
    \bbW_1\big(
    \bbE[\wh\mu_{M_N^{\delta}}\,|\,H_N],
    \mu_{H_N,\va,\bx}^{\delta}
    \big)
    \leq \delta/4.
    \]
    Applying Lemma~\ref{lem:spectral-concentration} also conditionally on $H_N$ gives:
    \[
    \bbP\lt[H_N,\frac{H_{N,\delta}}{2}\in K_N~~\text{and}~~
    \bbW_1\big(\wh{\mu}_{M_N^{\delta}},
    \bbE[\wh\mu_{M_N^{\delta}}\,|\,H_N]\big)\geq\delta/4\rt]
     \leq
    C'e^{-cN^{1+\eps_0}}/4
    .
    \]
    Here $\frac{H_{N,\delta}}{2}\in K_N$ was used to imply $\wh{\mu}_{M_N^{\delta}}=\wh{\mu}_{M_N^{\delta}}^{(C_2)}$.
    Lemma~\ref{lem:spectral-concentration} applied without conditioning on $H_N$ similarly yields
    \[
    \bbP\lt[H_N,\frac{H_{N,\delta}}{2}\in K_N~~\text{and}~~
    \bbW_1\big(
    \wh{\mu}_{M_N^{\delta}},
    \bbE[\wh{\mu}_{M_N^{\delta}}]
    \big)\geq\delta/4\rt]
     \leq
    C'e^{-cN^{1+\eps_0}}/4
    .
    \]
    Finally since $\bbE[\wh{\mu}_{M_N^{\delta}}]=\bbE[\wh{\mu}_{M_N}]$, it follows from \eqref{eq:W1-approx-MDE} (now applied unconditionally) that
    \[
    \bbW_1\big(
    \bbE[\wh{\mu}_{M_N}],
    \mu_{\xi,\va}
    \big)
    \leq
    \delta/4.
    \]

    Combining using the triangle inequality yields
    \[
    \bbP\lt[
    H_N,\frac{H_{N,\delta}}{2}\in K_N~~\text{and}~~
    \bbW_1\big(
    \mu_{H_N,\va,\bx}^{\delta}
    ,
    \mu_{\xi,\va}
    \big)
    \geq \delta
    \rt]
    \leq
    C'e^{-cN^{1+\eps_0}}/2
    .
    \]
    It remains to observe that the event $\bbW_1\big(
    \mu_{H_N,\va,\bx}^{\delta}
    ,
    \mu_{\xi,\va}
    \big)
    \geq \delta$ is determined by $H_N$, and that
    \begin{equation}
    \label{eq:two-inequalities}
    \bbP\lt[\frac{H_{N,\delta}}{2}\in K_N\,|\,H_N\rt]\geq \bbP[H_N'\in K_N]\geq 1/2
    \end{equation}
    for any $H_N\in K_N$.
    The former inequality in \eqref{eq:two-inequalities} follows by writing $H_{N,\delta}/2=\frac{\sqrt{1-\delta}}{2}H_N + \frac{\sqrt{\delta}}{2}H_N'$ because $\frac{\sqrt{1-\delta}}{2}+\frac{\sqrt{\delta}}{2}\leq 1$.
    Indeed since $K_N$ is a symmetric convex set it contains the origin, so if $H_N,H_N'\in K_N$ then also $H_{N,\delta}/2\in K_N$.
    The latter inequality in \eqref{eq:two-inequalities} is one of the defining properties of $K_N$ in Proposition~\ref{prop:gradients-bounded}, since $H_N'$ is an \iid copy of $H_N$.
\end{proof}

Next in Lemma~\ref{lem:MDE-uniform-bound-final} and Proposition~\ref{prop:KN-eps-high-prob}, we exhaust all bounded $\va$ in Lemma~\ref{lem:high-prob-MDE}. Combined with the continuity shown in Lemma~\ref{lem:hoffman-wielandt}, this shows validity of Lemma~\ref{lem:high-prob-MDE} uniformly in bounded $\va$.

\begin{lemma}
\label{lem:MDE-uniform-bound-final}
    For any $C,\eps > 0$,
    with probability $1-e^{-cN}$, $H_N\in K_N$ and
    \[
    \sup_{\bx\in\cS_N,\|\va\|_{\infty}\leq C}
    \bbW_1\big(\mu_{H_N,\bx,\va}^{\delta}
    ,
    \mu_{\xi,\va}
    \big)
    \leq
    \eps.
    \]
\end{lemma}

\begin{proof}
    Let $\cN_N$ be an $N^{-10}$-net for $\cS_N$ and $\cA_{N}$ an $N^{-10}$-net for the $C$-bounded vectors $\va$. Note that $|\cN_{N}\times \cA_N|\leq N^{O(CN)}$.
    Since $\bbP[H_N\in K_N]\geq 1-e^{-cN}$, union-bounding over the events in Lemma~\ref{lem:high-prob-MDE} over $\cN_N\times \cA_N$ implies that with probability $1-e^{-c'N}$,
    \begin{equation}
    \label{eq:MDE-bound-on-net}
    \sup_{\bx\in\cN_N,\va\in\cA_N}
    \bbW_1\big(\mu_{H_N,\bx,\va}^{\delta}
    ,
    \mu_{\xi,\va}
    \big)
    \leq
    \eps/2.
    \end{equation}
    Next assuming again that $H_N\in K_N$, for any $\hat\bx\in \cS_N$ and $C$-bounded $\hat\va$, let $\bx,\va$ be the nearest points in $\cN_N,\cA_N$. Let
    $M_N^{\delta}=M_N^{\delta}(\bx),
    \wh M_N^{\delta}=M_N^{\delta}(\hat\bx)$
    be the associated random Riemannian Hessians of $H_{N,\delta}$ given $H_N$. Then Lemma~\ref{lem:hoffman-wielandt} implies
    \[
    \bbW_1\big(\bbE[\wh{\mu}_{M_N^{\delta}}]
    ,
    \bbE[\wh{\mu}_{\wh M_N^{\delta}}]
    \big)
    \leq
    N^{-3}
    .
    \]
    Recalling the deterministic bound \eqref{eq:W1-approx-MDE}, we have
    \[
    \bbW_1\big(\mu_{M_N^{\delta}},\mu_{\wh M_N^{\delta}}\big)
    \leq 2\delta_N+N^{-3}
    \]
    whenever $H_N\in K_N$.
    Combining with \eqref{eq:MDE-bound-on-net} completes the proof, as $\mu_{M_N^\delta} = \mu^\delta_{H_N,\bx,\va}$.
\end{proof}

For $\bx\in\cS_N$, define the distorted Hessian
\begin{align}
    \label{eq:wt-Hessian}
    J_N(\bx)
    &\equiv
    \Xi \diamond \nabla^2_{\sph} H_N(\bx)\,, &
    \Xi = (\xi')^{-1/2} \otimes (\xi')^{-1/2} \in \bbR^{r\times r}.
\end{align}
This is again $C$-regular conditionally on $\va=\nabla_{\rd}H_N(\bx)$, and we let $\wt{\mu}_{\xi,\va}$ be the corresponding solution to \eqref{eq:MDE}.

\begin{definition}
\label{def:KN-eps}
    The set $K_N(\eps)\subseteq K_N$ consists of all $H_N\in\sH_N$ satisfying:
\begin{enumerate}[label=(\alph*)]
    \item
    \label{it:eps-C-bounded}
    $H_N\in K_N$.
    \item
    \label{it:actual-hessian-close}
    For all $\bx\in\cS_N$, with $M_N=\nabla^2_{\sph}H_N(\bx)$,
    \[
    \bbW_1(\wh{\mu}_{M_N},
    \mu_{\xi,\nabla_{\rd}H_N(\bx)})
    \leq \eps.
    \]
    \item
    \label{it:deformed-hessian-close}
    For all $\bx\in\cS_N$, with
    $\wt M_N=J_N(\bx)$,
    \[
    \bbW_1(\wh{\mu}_{\wt M_N},\wt\mu_{\xi,
    \nabla_{\rd}H_N(\bx)})
    \leq \eps.
    \]
    \item
    \label{it:MDE-bulk-close}
    For all $\bx\in\cS_N$ and $\|\va\|_{\infty}\leq C$ we have
    \[
    \bbW_1\big(
    \mu_{H_N,\bx,\va}^{\delta}
    ,
    \mu_{\xi,\va}
    \big)
    \leq
    \eps.
    \]
\end{enumerate}
\end{definition}

\begin{proposition}
\label{prop:KN-eps-high-prob}
    For any $\eps>0$, we have $\bbP[H_N\in K_N(\eps)]\geq 1-e^{-c(\eps)N}$.
\end{proposition}

\begin{proof}
    Part~\ref{it:eps-C-bounded} follows from Proposition~\ref{prop:gradients-bounded}.
    Part~\ref{it:MDE-bulk-close} follows by Lemma~\ref{lem:MDE-uniform-bound-final}.
    A similar argument implies parts \ref{it:actual-hessian-close} and \ref{it:deformed-hessian-close}.
\end{proof}

\subsection{Main Argument}

We fix a small constant $\delta$ and set
\begin{equation}
\label{eq:small-constants}
    (\eps,\alpha,\eta,\iota)
    =
    \big(\delta^{10},
    \delta^{1/3},
    \delta^{1/10},
    \delta^{1/100}
    \big)
    .
\end{equation}

\begin{lemma}
\label{lem:determinant-of-pertubation}
    Fix $H_N\in K_N(\eps)$. For each $\by\in\cS_N$ and $C$-bounded $\va$, with $N$ sufficiently large, let $E_{\by}$ be an event satisfying
    \[
    \inf_{\by\in\cS_N}
    \bbP
    \Big[E_{\by}~\Big|~
    \big(
    \nabla_{\sph}
    H_{N,\delta}(\by)
    ,
    \nabla_{\rd}
    H_{N,\delta}(\by)
    ,
    H_N
    \big)\Big]
    \geq
    1/2.
    \]
    Then including $1_{E_{\by}}$ within the determinant expectation for $\nabla^2_{\sph}
    H_{N,\delta}(\by)$ has a negligible effect, in the sense that uniformly in $\by$:
    \begin{align}
    \label{eq:determinant-of-pertubation}
    &\Bigg|
    \frac{1}{N}
    \log
    \bbE
    \lt[
    1_{E_{\by}}\cdot
    \big|
    \det
    \nabla^2_{\sph}
    H_{N,\delta}(\by)
    \big|
    ~\Big|~
    \big(
    \nabla_{\sph}
    H_{N,\delta}(\by)
    =
    0,
    \nabla_{\rd}
    H_{N,\delta}(\by)
    =
    \va,
    H_N
    \big)
    \rt] \\
    \notag
    &-
    \int \log|\lambda|\de\mu_{\xi,\va}(\lambda)
    \Bigg|
    \leq
    o_{\delta}(1).
    \end{align}
\end{lemma}

\begin{proof}
    By Lemma~\ref{lem:derivative-laws}, $\nabla^2_{\cT \times \cT} H_{N,\delta}(\by)$ and $\nabla H_{N,\delta}(\by)$ are independent.
    Recalling \eqref{eq:H-N-delta} the conditional law of $\nabla^2_{\sph}H_{N,\delta}(\by)$ agrees with that of
    \[
    \sqrt{1-\delta}\nabla^2_{\cT \times \cT} H_N(\by)
    +
    \sqrt{\delta} \nabla^2_{\cT \times \cT} H_N'(\by)
    -
    \diag(\Lambda^{-1/2} \va \diamond \bone_\cT)
    \]
    which is $C(\delta)$-regular for $H_N\in K_N$.
    Upper-bounding the left-hand side of \eqref{eq:determinant-of-pertubation} by
    \begin{align*}
    &\Bigg|
    \frac{1}{N}
    \log
    \bbE
    \lt[
    1_{E_{\by}}\cdot
    \big|
    \det
    \nabla^2_{\sph}
    H_{N,\delta}(\by)
    \big|
    ~\Big|~
    \big(
    \nabla_{\sph}
    H_{N,\delta}(\by)
    =
    0,
    \nabla_{\rd}
    H_{N,\delta}(\by)
    =
    \va,
    H_N
    \big)
    \rt] \\
    &\qquad -
    \int \log|\lambda|\de\mu_{H_N,\bx,\va}^{\delta}(\lambda)
    \Bigg|
    +
    \lt|
    \int \log|\lambda|\de\mu_{H_N,\bx,\va}^{\delta}(\lambda)
    -
    \int \log|\lambda|\de\mu_{\xi,\va}(\lambda)
    \rt|,
    \end{align*}
    we can apply Lemma~\ref{lem:local-law-boosted} to bound the first term and Lemma~\ref{lem:log-integral-convergence} (using part~\ref{it:MDE-bulk-close} of Definition~\ref{def:KN-eps}) to the second.
\end{proof}

The next lemma gives a Taylor expansion estimate for $\|\nabla_{\sph} H_N(\by)\|_2$.
In it, we let $\gamma_{\bx\to \by}:[0,1]\to \cS_N$ be the shortest path geodesic with $\gamma_{\bx\to \by}(0)=\bx$ and $\gamma_{\bx\to \by}(1)=\by$ (say, whenever $\by\in B_{\alpha\sqrt{N}}(\bx)$ so the shortest path is clear). Note that $\gamma_{\bx\to\by}'(0)$ is approximately $\by-\bx$; the result below holds with this replacement as well, but for the application $\gamma_{\bx\to\by}'(0)$ will be more convenient since it is in the tangent space $T_{\bx} \cS_N$.

\begin{lemma}
\label{lem:grad-y-taylor}
    Let $H_N \in K_N$. For $\by\in B_{\alpha\sqrt{N}}(\bx)\cap \cS_N$ and bounded $\vv\in\bbR^{r}$, we have
    \begin{equation}
    \label{eq:claim-grad-y-taylor}
        \tnorm{\vv\diamond\nabla_{\sph} H_N(\by)}_2
        \le
        \tnorm{
            \vv\diamond\nabla_{\sph} H_N(\bx) +
            \vv\diamond\nabla^2_{\sph} H_N(\bx)\cdot \gamma_{\bx\to\by}'(0)
        }_2
        + C\alpha^2 \sqrt{N}\,.
    \end{equation}
\end{lemma}

\begin{proof}
    Let $P_t:T_{\gamma_{\bx\to\by}(t)}\cS_N\to T_{\bx}\cS_N$ be the associated parallel transport map on tangent spaces.
    This is a product of isometries on subspaces of each $\bbR^{\cI_s}$, so
    \[
    \tnorm{\vv\diamond\nabla_{\sph} H_N(\gamma(t))}_2
    =\tnorm{\vv\diamond P_t\big(\nabla_{\sph} H_N(\gamma(t))\big)}_2.
    \]
    Moreover,
    \[
        \frac{\de}{\de t}
        \lt[
        \vv\diamond
        P_t\big(\nabla_{\sph} H_N(\gamma_{\bx\to\by}(t))\big)
        \rt]\big|_{t=0}
        =
        \vv\diamond
        \nabla^2_{\sph}
        H_N(\gamma_{\bx\to\by}(t)) \gamma_{\bx\to\by}'(t)\,.
    \]
    It remains to apply Taylor's theorem to $\vv\diamond P_t\big(\nabla_{\sph} H_N(\gamma(0)))$ with derivative bounds from Proposition~\ref{prop:gradients-bounded} as $H_N\in K_N$.
\end{proof}

\begin{lemma}
\label{lem:approx-to-exact-computation}
    Fix $H_N\in K_N(\eps)$ and $\bx$ such that
    \begin{equation}
    \label{eq:grad-H-small}
        \tnorm{\nabla_{\sph} H_N(\bx)}_2
        \leq
        \eps\sqrt{N}\,.
    \end{equation}
    Then in expectation over $H_N'$, there are at least $e^{-o_{\delta}(N)}$ points $\by\in B_{\alpha\sqrt{N}}(\bx)\cap \cS_N$ such that
    \begin{align}
    \label{eq:H-y-good}
        |H_N(\bx)-H_{N,\delta}(\by)|&\leq \iota^2 N,
        \\
    \label{eq:grad-y-zero}
        \nabla_{\sph} H_{N,\delta}(\by)
        &=0,
        \\
    \label{eq:rad-y-good}
        \tnorm{\nabla_{\rd} H_N(\bx)-\nabla_{\rd} H_{N,\delta}(\by)}_\infty
        &\leq\iota^2,
        \\
    \label{eq:Hess-y-good}
        \|\nabla^2_{\sph}H_N(\bx)-\nabla^2_{\sph}H_{N,\delta}(\by)\|_{\op}
        &\leq \iota
        .
    \end{align}
\end{lemma}

\begin{proof}
    Throughout the proof we fix and condition on $H_N\in K_N(\eps)$. Then with high conditional probability, the events \eqref{eq:H-y-good}, \eqref{eq:rad-y-good}, and \eqref{eq:Hess-y-good} all occur (recall from \eqref{eq:small-constants} that $\iota$ is much larger than $\alpha$ and $\delta$).
    Let $E_1$ be the event that \eqref{eq:H-y-good} and \eqref{eq:rad-y-good} hold, and $E_{2,\by}$ the event that \eqref{eq:Hess-y-good} holds.

    By considering only the contribution from $E_1\cap E_{2,\by}$ and recalling Lemma~\ref{lem:derivative-laws} part~\ref{it:tangential-derivative-law}, we find that
    the expected number of such critical points is at least
    \begin{align}
    \label{eq:conditional-Kac--Rice-integral}
    &(2\pi\delta)^{-\frac{N-r}{2}}
    \prod_{s\in\sS} (\xi^s)^{\frac{N_s-1}{2}}
    \int_{\by\in B_{\alpha\sqrt{N}}(\bx)\cap \cS_N}
    \Bigg(
    \exp\lt(
    \frac{-(1-\delta)}
    {2\delta}
    \|(\xi')^{-1/2}\diamond \nabla_{\sph} H_N(\by)\|_2^2
    \rt)
    1_{E_1}
    \\
    \nonumber
    &\qquad
    \times\min_{\|\va-\nabla_{\rd}H_N(\bx)\|_\infty\leq\iota^2}
    \bbE
    \bigg[
    1_{E_{2,\by}}\cdot
    \big|
    \det
    \nabla^2_{\sph}
    H_{N,\delta}(\by)
    \big| \\
    \notag
    &\qquad \qquad \qquad
    \,\Big|\,
    \big(
    \nabla_{\sph}
    H_{N,\delta}(\by)
    =0,
    \nabla_{\rd}
    H_{N,\delta}(\by)
    =\va
    ,
    H_N
    \big)
    \bigg]
    \Bigg)
    ~\de \by
    \end{align}
    If in addition to $H_N$ one conditions on any $(H_{N,\delta}(\by),\nabla_{\rd}H_{N,\delta}(\by))$ satisfying $E_1$, we claim the conditional probability of $E_{2,\by}$ (i.e. \eqref{eq:Hess-y-good}) is at least $1/2$.
    Indeed, the conditional mean of $\nabla_{\sph}^2 H_{N,\delta}(\by)$ is given by $\sqrt{1-\delta}\nabla_{\sph}^2 H_{N,\delta}$, plus an additive shift of operator norm $O(\iota^2)\ll \iota$ (coming from linear regression via Lemma~\ref{lem:derivative-laws}). The conditionally random part of $\nabla_{\sph}^2 H_{N,\delta}(\by)$ consists of an additive $\sqrt{\delta}\nabla_{\cT\times\cT}^2 H_N'(\by)$, which has operator norm $O(\sqrt{\delta})\ll\iota$ with high probability.

    Hence Lemma~\ref{lem:determinant-of-pertubation} (which holds uniformly in $\by$) implies that the latter part of the integrand is, for each $\va$ such that $\|\va-\nabla_{\rd}H_N(\bx)\|_\infty \leq\iota^2$:
    \begin{equation}
        \label{eq:determinant-of-perturbation-application}
        \exp\lt(
        N
        \int \log |\lambda|\,
        [\mu_{\xi,\nabla_{\rd}H_N(\bx)}](\de\lambda)
        \pm o_{\iota}(N)
        \rt)
        .
    \end{equation}
    Indeed, \eqref{eq:shifted-determinant-formula} shows it is suitably close to
    $\exp\lt(N\int \log |\lambda|\,[\mu_{\xi,\va}](\de\lambda)\pm o_{\iota}(N)\rt)$
    and combining Lemmas~\ref{lem:hoffman-wielandt} and \ref{lem:log-integral-convergence} allows us to replace $\va$ by $\nabla_{\rd}H_N(\bx)$.
    It remains to integrate the other term, namely the origin density of $\nabla_{\sph}H_{N,\delta}(\by)$ conditional on $H_N$:
    \begin{equation}
    \label{eq:origin-density-y}
    (2\pi\delta)^{-\frac{N-r}{2}}
    \prod_{s\in\sS} (\xi^s)^{\frac{N_s-1}{2}}
    \exp\lt(
    \frac{-(1-\delta)}
    {2\delta}
    \|(\xi')^{-1/2}\diamond \nabla_{\sph} H_N(\by)\|_2^2
    \rt)
    .
    \end{equation}
    Using \eqref{eq:claim-grad-y-taylor} with $\vv=(\xi')^{-1/2}$ and $\tnorm{\bu+\bv}_2^2\le (1+\iota)\tnorm{\bu}_2^2 + (C/\iota)\tnorm{\bv}_2^2$ we get:
    \begin{align*}
        \|(\xi')^{-1/2}\diamond \nabla_{\sph} H_N(\by)\|_2^2
        \leq
        (1+\iota)
        \tnorm{(\xi')^{-1/2}\diamond\nabla^2_{\sph} H_N(\bx)\cdot \gamma_{\bx\to\by}'(0)}_2^2
        +
        \frac{C(\eps^2+\alpha^4)}{\iota} N\,.
    \end{align*}
     Recalling \eqref{eq:origin-density-y}, we integrate over a subset of $\by\in B_{\alpha\sqrt{N}}(\bx)$ chosen as follows.
     Let $T_1(\bx)$ be the span of the eigenvectors of
     $J_N(\bx)$
     (recall \eqref{eq:wt-Hessian})
     with eigenvalues inside $[-\eta,\eta]$, and $T_2(\bx)$ the orthogonal complement in the tangent space $T_x \cS_N$.
     Since $H_N\in K_N(\eps)$ and $\mu_{\xi,\va}$ is $C_1$-regular uniformly in $C$-bounded $\va$, we have $\dim(T_1(\bx))\leq O(N\eta)$.

     Write $\wt\gamma_{\bx\to\by}'(0)=(\xi')^{1/2}\diamond \gamma_{\bx\to\by}'(0)$.
     Let $S(\bx)$ be the Minkowski sum of a radius $\alpha^2\sqrt{N/2}$ ball $S_1(\bx)\subseteq T_1(\bx)$ and a radius $\alpha\sqrt{N/2}$ ball $S_2(\bx)\subseteq T_2(\bx)$.
     We consider $\by$ for which $\wt\gamma_{\bx\to\by}'(0)\in S(\bx)$ and accordingly write $\wt\gamma_{\bx\to\by}'(0)=\bs_1+\bs_2$ for $\bs_i\in S_i(\bx)$.
     It is easy to see that the map $\by\to \gamma_{\bx\to\by}'(0)$ is a diffeomorphism on $\by\in B_{\alpha\sqrt{N}}(\bx)$ with Jacobian determinant $1\pm o_{\alpha}(1)$ uniformly.
     We will thus freely switch to integration over $\gamma_{\bx\to\by}'(0)$, which is equivalent to integration over $\wt\gamma_{\bx\to\by}'(0)$ after picking up a factor
     \[
     \prod_{s\in\sS}
     (\xi^s)^{(|\cI_s|-1)/2}
     =
     e^{o(N)}
     \cdot
     \prod_{s\in\sS}
     (\xi^s)^{\lambda_s N/2}.
     \]
     Continuing,
     \begin{align*}
         \|J_N(\bx)
         \cdot\wt\gamma_{\bx\to\by}'(0)\|_2^2
         &=
         \|J_N(\bx)\cdot(\bs_1+\bs_2)\|_2^2
         \\
         &=
         \|J_N(\bx)\cdot\bs_1\|_2^2
         +
         \|J_N(\bx)\cdot\bs_2\|_2^2
         \\
         &\leq
         C\eta^2\alpha^4N
         +
         \|J_N(\bx)\cdot\bs_2\|_2^2
        .
     \end{align*}
    From \eqref{eq:small-constants}, we have $\eta^2\alpha^4\ll \delta$ so the first term will be negligible below.
     By definition of $S_2(\bx)$, the vector
     $\bv_2=J_N(\bx)\cdot\bs_2$
     ranges over a superset of the ball of radius $\alpha\eta\sqrt{N}/C$ in $S_2(\bx)$.
     Since $\alpha\eta\gg \sqrt{\delta}$ from \eqref{eq:small-constants}, this captures at least $1/2$ of the Gaussian integral mass for \eqref{eq:origin-density-y}. Writing $S_{1,2}$ for the product $S_1(\bx)\times S_2(\bx)$, we find that
     \begin{align}
        \notag
         &
         (2\pi\delta)^{-\frac{N-r}{2}}
        \prod_{s\in\sS} (\xi^s)^{-\frac{N_s-1}{2}}
         \int\limits_{B_{\alpha\sqrt{N}}(\bx)}
         \exp\lt(
        \frac{-(1-\delta)}{2\delta}
        \|(\xi')^{-1/2}\diamond \nabla_{\sph} H_N(\by)\|_2^2
        \rt)
        \de \by
        \\
        \notag
        &\geq
        (2\pi\delta)^{-\frac{N-r}{2}}
        \prod_{s\in\sS}(\xi^s)^{-\lambda_s N} \\
        \notag
        &\qquad \times
        \int\limits_{S_{1,2}}
         \exp\bigg(
        \frac{-(1+\iota)\|J_N(\bx)\cdot(\bs_1+\bs_2)
        \|_2^2
        -
        C(\eps^2+\alpha^4)N/\iota
        }
        {2\delta}
        \bigg)
        \,\de\bs_2\de \bs_1
        \\
        \notag
        &=
        (2\pi\delta)^{-\frac{N-r}{2}}
        \prod_{s\in\sS}(\xi^s)^{-\lambda_s N}
        \exp\bigg(
        -\frac{C(\eps^2+\alpha^4)N}{2\delta\iota}
        \bigg) \\
        \notag
        &\qquad \times
        \int\limits_{S_{1,2}}
         \exp\bigg(
        \frac{-\|(1+\iota)
        J_N(\bx)\cdot(\bs_1+\bs_2)
        \|_2^2
        }
        {2\delta}
        \bigg)
        \,\de\bs_2\de \bs_1
        \\
        \notag
        &\geq
        (2\pi\delta)^{-\frac{N-r}{2}}
        \prod_{s\in\sS}(\xi^s)^{-\lambda_s N}
        \exp\bigg(
        -\frac{C(\eps^2+\alpha^4)N}{2\delta\iota}
        -
        \frac{(1+\iota)
        C\eta^2\alpha^4 N}
        {\delta}
        \bigg) \\
        \notag
        &\qquad \times
        \int\limits_{S_{1,2}}
         \exp\bigg(
        \frac{-\|(1+\iota)
        J_N(\bx)\cdot\bs_2
        \|_2^2
        }
        {2\delta}
        \bigg)
        \,\de\bs_2\de \bs_1
        \\
        \notag
        &\geq
        \prod_{s\in\sS}(\xi^s)^{-\lambda_s N}
        \exp(-\iota N)
        \Vol(S_1(\bx))
        ~\Big|
        \det\lt(J_N(\bx)\big|_{S_2(\bx)}\rt)^{-1}
        \Big| \\
        \notag
        &\qquad \times
        \int\limits_{
            \substack{\bv_2\in S_2(\bx),
                \\
            \|\bv_2\|_2\leq \alpha\eta\sqrt{N}/C}
        }
        (2\pi\delta)^{-\frac{N-r}{2}}
         \exp\lt(
        \frac{
        -\|(1+\iota)
        \bv_2
        \|_2^2
        }
        {2\delta}
        \rt)
        ~\de\bv_2
        \\
        \label{eq:zero-density-integral}
        &\geq
        \prod_{s\in\sS}(\xi^s)^{-\lambda_s N}
        \exp(-\iota N)
        \Vol(S_1(\bx))
        ~\Big|
        \det\lt(J_N(\bx)\big|_{S_2(\bx)}\rt)^{-1}
        \Big|~
        (1+\iota)^{-\frac{N-r}{2}}
        /2.
    \end{align}
    Recalling \eqref{eq:small-constants} and that $\dim(T_1(\bx))\leq O(N\eta)$, we find $\Vol(S_1(\bx))\geq \alpha^{O(N\eta)}\geq e^{-\iota N}$.
    Uniformly over $H_N\in K_N(\eps)$ and $\bx\in\cS_N$, from Definition~\ref{def:KN-eps} part~\ref{it:deformed-hessian-close} we have
    \[
    \lim_{\delta\to 0}
    \lim_{N\to\infty}
    \bigg|
    \frac{1}{N}\log\det\lt(J_N(\bx)\big|_{S_2(\bx)}\rt)
    -
    \lt(
    \int \log|\lambda|\,
    [\mu_{\xi,\nabla_{\rd} H_N(\bx)}](\de\lambda)
    -
    \sum_{s\in\sS}
    \lambda_s
    \log(\xi^s)
    \rt)
    \bigg|
    =0
    .
    \]
    Indeed up to factors of $e^{o_{\delta}(N)}$, uniformly over these sets we have
    \begin{align*}
    \det\lt(J_N(\bx)\big|_{S_2(\bx)}\rt)
    &\approx
    \bbE
    \lt[
    \det\lt(J_N(\bx)\rt)
    ~\big|~
    \nabla_{\rd}H_N(\bx)
    \rt]
    \\
    &=
    \bbE
    \lt[
    \det\lt(\nabla^2_{\sph}(H_N(\bx))\rt)
    ~\big|~
    \nabla_{\rd}H_N(\bx)
    \rt]
    \prod_{s\in\sS}(\xi^s)^{-\lambda_s N}
    \\
    &\approx
    \exp\lt(N\int \log|\lambda|\,
    [\mu_{\xi,\nabla_{\rd} H_N(\bx)}](\de\lambda)\rt)
    \prod_{s\in\sS}(\xi^s)^{-\lambda_s N}.
    \end{align*}
    Thus \eqref{eq:zero-density-integral} equals
    \[
        e^{-o_\delta(N)}
        \exp\lt(-N\int \log|\lambda|\,
        [\mu_{\xi,\nabla_{\rd} H_N(\bx)}](\de\lambda)\rt)\,.
    \]
    This cancels the expected determinant given approximately by \eqref{eq:determinant-of-perturbation-application}, leaving $e^{-o_{\delta}(N)}$ and completing the proof.
\end{proof}

\begin{remark}
We restricted attention to small $\|\bs_1\|_2$ above because this causes $\|\nabla^2_{\sph} H_N(\bx)\cdot \bs_1\|_2^2$ to be negligible. The trade-off is that $\Vol(S_1(\bx))$ becomes smaller. However since $\dim(T_1(\bx))\leq O(N\eta)$, this volumetric factor is also irrelevant because all small parameters were polynomially related.
\end{remark}

Using Lemma~\ref{lem:approx-to-exact-computation}, we now deduce Theorem~\ref{thm:approx-crits-from-annealed} from the start of this section.

\begin{proof}[Proof of Theorem~\ref{thm:approx-crits-from-annealed}]
    Let $K_N(\eps,\ups)$ consist of those $H_N\in K_N(\eps)$ such that $|\Crt_N^{\ocD,\eps,\ups}(H_N)|\geq 1$.
    Let $H_N \in K_N(\eps,\ups)$ and let $\bx \in \Crt_N^{\ocD,\eps,\ups}(H_N)$ be an $\eps$-critical point satisfying \eqref{eq:ups-far-from-ocD}.
    For $\eps$ small enough that (recalling \eqref{eq:small-constants}) $\iota\leq \ups/C$, we claim that Lemma~\ref{lem:approx-to-exact-computation} implies
    \begin{equation}
    \label{eq:previous-display-averaged}
    \bbE
    \lt[
    \big|\Crt_N^{\ocD,0,\ups/2}(H_N)\big|
    \,|\,
    H_N
    \rt]
    \geq e^{-o_{\eps}(N)}
    \cdot \bone\{H_N\in K_N(\eps,\ups)\}.
    \end{equation}
    Indeed \eqref{eq:rad-y-good} and \eqref{eq:H-y-good} handle the radial derivative and energy errors between $\bx$ and $\by$.
    Proposition~\ref{prop:spectral-perturbation} is used exploit the hypothesis \eqref{eq:Hess-y-good}; the resulting $\bbW_{\infty}$ estimate controls both the distances in $\cJ$ (Hausdorff distance between the spectral supports) and in $\bbW_1(\bbR)$.
    Finally the overlap with $\bG^{(1)}$ is controlled by the simple bound
    \[
    \by\in B_{\alpha\sqrt{N}}(\bx)\cap \cS_N
    \quad\implies\quad
    \tnorm{\vR(\bx,\bG^{(1)})-\vR(\by,\bG^{(1)})}\leq O(\alpha)\ll \upsilon,
    \]
    where we used that $H_N\in K_N(\eps)\subseteq K_N$ implies $\|\bG^{(1)}\|\leq O(\|\nabla H_N(\bzero)\|)\leq O(\sqrt{N})$.

    Averaging \eqref{eq:previous-display-averaged} over $H_N$ and applying
    the hypothesis~\eqref{eq:annealed-crits-close-to-ocD},
    we find
    \[
        e^{-c_0 N}
        \geq
        \bbE
        \big|\Crt_N^{\ocD,0,\ups/2}(H_N)\big|
        \geq
        e^{-o_{\eps}(N)}\cdot
        \bbP[H_N\in K_N(\eps,\ups)].
    \]
    Choosing $\eps$ to also be sufficiently small depending on $c_0$ and rearranging yields $\bbP[H_N\in K_N(\eps,\ups)]\leq e^{-cN}$.
    Recalling that $\bbP[H_N\in K_N(\eps)]\geq 1-e^{-cN}$ by Proposition~\ref{prop:KN-eps-high-prob} completes the proof.
\end{proof}

\subsection{Failure of Annealed Topological Trivialization for Sub-Solvable $\xi$}
\label{subsec:annealed-complexity-positive}

We showed in Section~\ref{sec:variational} that annealed topological trivialization occurs for super-solvable $\xi$ (recall Definition~\ref{def:xi'}).
In this subsection we prove the strictly sub-solvable case \ref{itm:crt-tot-sub-solvable} of Theorem~\ref{thm:annealed-crits}, which is equivalent by Proposition~\ref{prop:annealed-kac-rice} to the following.
Recall the reparameterized form $\oF$ of the complexity functional $F$ defined in Subsection~\ref{subsec:stationarity}.
\begin{proposition}
\label{prop:sub-solvable-complexity-positive}
    If $\xi$ is strictly sub-solvable, then $\sup_{\vv\in\bbR^r} \oF(\vv)>0$.
\end{proposition}

We will require a computation from our concurrent paper \cite{huang2023optimization} as well as Theorem~\ref{thm:approx-crits-from-annealed}.
The point is that in \cite{huang2023optimization}, we gave an explicit algorithm to construct approximate critical points for $H_N$ whenever $\xi$ is strictly sub-solvable.
Applying Theorem~\ref{thm:approx-crits-from-annealed} then shows $\oF$ is non-negative at the radial derivative of such points (which is computed explicitly therein).
While we cannot show this input makes $\oF$ strictly positive, we do show it is not a stationary point of $\oF$, which suffices.

The algorithm from \cite{huang2023optimization} relies on a certain coordinate-wise increasing $C^1$ path $\Phi:[0,1]\to [0,1]^r$.
In short, it proceeds by outward exploration in $\bbR^N$ starting from $\bzero$, greedily optimizing $H_N$ in each step similarly to \cite{subag2018following} (though implemented with approximate message passing as in \cite{montanari2021optimization,ams20,sellke2024optimizing}).
The function $\Phi$ determines the schedule at which exploration occurs in the $r$ species (which is trivial in the single-species setting).
The optimal choice of $\Phi$ obeys stationarity conditions, which were established in \cite{huang2023algorithmic} and exploited in \cite{huang2023optimization}.
In particular the optimal $\Phi$ exhibits a phase transition when $\xi$ shifts from super-solvable to sub-solvable (which in fact motivated the present paper).

$\Phi$ does not seem to have an explicit formula in the sub-solvable case, but is given by any maximizer of a $(\xi,\vlambda)$-dependent functional $\bbA$ defined in \cite[Equation (1.6)]{huang2023algorithmic}.
(As explained in Remark~\ref{rem:external-field}, the fact that these companion results technically use deterministic external field instead of Gaussian $\bG^{(1)}$ is inconsequential.)
$\Phi$ satisfies the normalization $\la\vlambda,\Phi'(q)\ra=1$ for all $q\in [0,1]$, and moreover $\Phi_s'(1)> 0$ for all $s\in\sS$.
The input from \cite{huang2023optimization} is as follows.

\begin{proposition}[{\cite[Proposition 3.3]{huang2023optimization}}]
\label{prop:IAMP}
    For non-degenerate and strictly super-solvable $\xi$, and $\Phi$ as above, and any $\eps>0$, with probability $1-e^{-cN}$ there exists an $\eps$-approximate critical point $\bx_*\in \cS_N$ such that
    \begin{equation}
        \label{eq:amp-gives-approx-crit}
        \|\Lambda^{1/2} \nabla_{\rd} H_N(\bx_*)
        -
        \vv_*(\Phi)\|_2
        \leq \eps,
    \end{equation}
    where $\vv_*(\Phi) = (v_{*,s}(\Phi))_{s\in \sS}$ is given by
    \begin{align}
        \label{eq:v-*-def}
        v_{*,s}(\Phi)
        &=
        \lambda_s f_s^{-1}+\sum_{s'\in\sS}
        \xi''_{s,s'}f_{s'}\,;
        \\
        \label{eq:f-s-def}
        f_s
        &=
        \sqrt{\frac{\Phi_s'(1)}{(\xi^s\circ\Phi)'(1)}}.
    \end{align}
\end{proposition}
Proposition~\ref{prop:sub-solvable-complexity-positive} is a direct consequence of the following result.

\begin{theorem}
\label{thm:annealed-complexity-positive}
    If $\xi$ is non-degenerate and strictly sub-solvable, then for any $\Phi$ as above, we have $\oF(\vv_*)\geq 0$ and $\nabla \oF(\vv_*)\neq \vzero$. Hence $\sup_{\vv\in\bbR^r}\oF(\vv)>0$, i.e. the annealed complexity is strictly positive.
\end{theorem}

\begin{proof}
    We apply Theorem~\ref{thm:approx-crits-from-annealed} with
    \[
    \ocD=\{\Lambda^{-1/2} \vv_*(\Phi)\}\times \bbR^r\times\bbR\times \cJ\times \bbW_1(\bbR).
    \]
    (I.e. we consider only the radial derivative and ignore the remaining components of $\ocD$.)
    The high-probability existence of $\bx_*$ obeying \eqref{eq:amp-gives-approx-crit} for arbitrarily small $\eps$, combined with continuity of $\oF$, yields $\oF(\vv_*)\geq 0$.

    Let $\vu_* = -\vec f$.
    We claim that $\vu(0;\vv_*) = \vu_*$.
    Due to the formula \eqref{eq:v-*-def}, $\vu_*$ satisfies \eqref{eq:dyson-equation-2}, so by Lemma~\ref{lem:u-manifold} case \ref{itm:semicircle-flat} it suffices to check $M(\vu_*) \succeq 0$.
    This follows by Lemma~\ref{lem:diagonally-signed-hs23}, as $\Phi'(1) \succ 0$ and $M(\vu_*)\Phi'(1) = \vzero$ by inspection.

    Next, Lemma~\ref{lem:stationary-condition} implies that stationary points with $\vu\in\bbR^r$ satisfy $u_s^2 = \frac{1}{\xi^s(\vone)}$. Hence if $\vv_*$ were stationary, rearranging using the definition \eqref{eq:f-s-def} of $f_s$ would directly yield
    \begin{align*}
        f_s^2
        =
        \frac{\Phi_s'(1)}{(\xi^s\circ\Phi)'(1)}
        &=
        \frac{1}{\xi^s(\vone)}\quad\forall s\in\sS
        \\
        \implies
        \big(\diag(\xi') - \xi''\big)\Phi'(1) &= 0.
    \end{align*}
    Since $\Phi'(1)\succ 0$, Lemma~\ref{lem:diagonally-signed-hs23} implies $\diag(\xi') - \xi'' \succeq 0$, contradicting that $\xi$ is strictly sub-solvable.

    We conclude that $\nabla\oF(\vv_*)\neq \vzero$. Combined with the fact that $\oF(\vv_*)\geq 0$ immediately yields $\sup_{\vv\in\bbR^r}\oF(\vv)>0$ as desired.
\end{proof}

\begin{remark}
\label{rem:positive-complexity}
We expect that for all (or at least almost all) strictly sub-solvable $\xi$ one has
\[
    \oF(\vv_*)>\oF(\vv_*,E_*)>0,
\]
where $E_*=\bbA(\Phi)\approx H_N(\bx_*)/N$ is the associated energy of $\Phi$ as described in \cite[Equation (1.6)]{huang2023algorithmic}.\footnote{As written therein $E_*=\bbA(p,\Phi;q_0)$; both $p:[0,1]\to [0,1]$ and $q_0\in [0,1]$ are implicitly determined by $\Phi$ via \cite[Theorem 3]{huang2023algorithmic}.} However both inequalities seem much more involved to prove.
Given the branching tree construction of algorithmic maximizers in \cite{huang2023optimization}, it is natural to speculate that $\oF(\vv_*,E_*)$ strictly increases along the tree-descending part $q\in [q_0,1]$ of any $\Phi$ satisfying the conclusions of \cite[Theorem 3]{huang2023algorithmic}.
\end{remark}

\begin{remark}
\label{rem:positive-complexity-2}
As we recalled in Proposition~\ref{prop:quenched-failure-of-STT}, our work \cite{huang2023optimization} actually constructs $\exp(\delta N)$ points $\bx_*$ satisfying the conditions of Proposition~\ref{prop:IAMP}, with all pairwise distances at least $\sqrt{N}/C(\xi)$. If one had $\eps$ sufficiently small given $\delta$, then Theorem~\ref{thm:approx-crits-from-annealed} would imply that $\oF(\vv_*)>0$. However \cite{huang2023optimization} only guarantees $\delta>0$ is positive for each $\eps$, which does not yield strict inequality.
On the other hand as explained in the introduction, it does imply \emph{quenched} failure of \emph{strong} topological trivialization.
\end{remark}

Finally we show that $\vv_*$ above corresponds to the top of the bulk spectrum of $\nabla_{\sph}^2 H_N(\bx_*)$ equalling zero. Informally, this means $\bx_*$ is on the verge of being a local maximum. More formally, it is an $\eps$-marginal local maximum as defined in Section~\ref{sec:approx-local-max-E-infty}.

\begin{proposition}
\label{prop:tree-descending-approx-maximum}
    For $\vv_*$ as above, $\max \supp(\omu_{\xi}(\vv_*))=0$.
\end{proposition}

\begin{proof}
    Consider a path $\vu(t)=(1+t)^{-1}\vu_*$ for $t\geq 0$ and with $u_{*,s}=-f_s$ as above.
    We verified above that $M(\vu_*)\succeq 0$, and from the definition \eqref{eq:def-oM} of $M$ it follows that $M(\vu(t))\succeq 0$ for all $t\geq 0$.
    By Lemma~\ref{lem:u-manifold} this means that for all $t\ge 0$, $\vu(t) = \vu(0;\vv(t))$ for some $\vv(t) \in \bbR^r$.

    For $t$ sufficiently large, \eqref{eq:dyson-equation-1} yields $v_s(t)\leq -\lambda_s/u_s \leq -C$ for large $C=C(\xi)>0$.
    At this point, \eqref{eq:mu-vx-approx} implies
    \[
    \bbW_{\infty}
    \Big(\omu(\vv(t)),
    \sum_{s\in\sS} \lambda_s \delta_{-v_s/\lambda_s}
    \Big)
    \leq C
    \]
    whence $\max\supp(\omu(\vv(t))<0$.

    Since $\vv(t)\in\bbR^r$, it follows that $\omu(\vv(t))$ always has density $0$ at $0$.
    By a continuity argument via Lemma~\ref{lem:Dyson-weakly-continuous}, if $\max\supp(\omu(\vv_*))>0$ held, then there would exist $t$ such that $\max\supp(\omu(\vv(t))=\delta$ for arbitrarily small $\delta>0$.
    It follows by \cite[Eq. (2.15)]{ajanki2019quadratic} that $\omu(\vv(t))$ must have positive density at $0$ for such $t$ when $\delta$ is taken sufficiently small, a contradiction.

    For the opposite direction, suppose that $\max\supp(\omu(\vv_*))<0$.
    Then $0\notin \supp(\omu(\vv_*))$, which implies that $M(\vu_*)$ is invertible by Proposition~\ref{prop:edge-vs-cusp} or Lemma~\ref{lem:M-singular-edge-cusp}.
    However $M(\vu_*)\Phi'(1)=\vzero$ so this cannot hold.
\end{proof}

\section{Locating the Critical Points}
\label{sec:2^r}

In this section, we complete the proof of Theorem~\ref{thm:precise-landscape} by combining the description of $\eps$-approximate critical points from Proposition~\ref{prop:approx-crits} with a recursive argument that localizes all approximate critical points.
Throughout this section we assume $\xi$ is strictly super-solvable.

In Subsection~\ref{subsec:6-1}, we define the \emph{type} $\vDelta\in\{-1,1\}^r$ of an approximate critical point based on its radial derivative.
It follows by the previous section that with high probability, all critical points have a well-defined type.
Next in Subsection~\ref{subsec:6-2} we explain the conditional law of $H_N$ on subspherical bands.
This lets us analyze the recursive algorithm of Subsection~\ref{subsec:6-3}.
In Subsection~\ref{subsec:6-4} we deduce that all approximate critical points of each type $\vDelta$ are localized inside a single small (random) subset of $\cS_N$.
Subsection~\ref{subsec:6-5} uses this to deduce existence and uniqueness of type of (exact) critical point.
Finally Subsection~\ref{subsec:index} determines the exact index of each critical point by gradually perturbing $\xi$ and arguing that eigenvalues do not cross $0$.


\subsection{Critical Points of Type $\vDelta$}
\label{subsec:6-1}

Our argument will separately localize each critical point of type $\vDelta \in \{-1,1\}^r$, defined as follows.
\begin{definition}
    \label{def:type}
    Let $\ups = o_\eps(1)$ be given by Proposition~\ref{prop:approx-crits}.
    Say $\bx \in \cS_N$ is a \textbf{$\eps$-critical point of type $\vDelta$}, or alternatively a \textbf{$(\eps,\vDelta)$-critical point}, if it is an $\eps$-critical point (recall Definition~\ref{def:approx}) and
    \begin{equation}
        \label{eq:type-radial-deriv}
        \tnorm{\nabla_{\rd} H_N(\bx) - \vx(\vDelta)}_\infty \le \ups\,.
    \end{equation}
\end{definition}
\begin{fact}
    \label{fac:type-partition}
    There exists $\eps_0 = \eps_0(\xi)$ such that with probability $1-e^{-cN}$ the following holds.
    For all $\eps \le \eps_0$, all $\eps$-critical points of $H_N$ are $(\eps,\vDelta)$-critical points for a unique $\vDelta \in \{-1,1\}^r$.
\end{fact}
\begin{proof}
    Immediate from Proposition~\ref{prop:approx-crits}.
    The signs $\vDelta$ are unique since for small $\eps$, the $\ups$-balls around the $\vx(\vDelta)$ are disjoint.
\end{proof}
\begin{definition}[Species-wise rescaling]
    Let $\bv \in \bbR^N$ such that $\bv_s \neq \bzero$ for all $s\in \sS$.
    For $\vq \in [0,1]^r$, $\vDelta \in \{-1,1\}^r$, define $\scl(\bv;\vDelta,\vq)$ to be the vector $\bu \in \bbR^N$ with
    \[
         \bu_s = \Delta_s \sqrt{q_s \lambda_s N} \fr{\bv_s}{\norm{\bv_s}_2}\,.
    \]
\end{definition}
That is, $\bu$ is the vector parallel to $\bv$ in each species with $\vR(\bu,\bu) = \vq$, whose species-$s$ component is correlated (resp. anti-correlated) with that of $\bv$ if $\Delta_s=1$ (resp. $-1$).
The following corollary of Proposition~\ref{prop:approx-crits} shows that $(\eps,\vDelta)$-critical points have nearly constant correlation with $\nabla H_N(\bzero) = \Gamma^{(1)} \diamond \bG^{(1)}$, the $1$-spin part of $H_N$.
\begin{corollary}
    \label{cor:approx-crits}
    For any $\eps>0$, there exists $\ups = o_\eps(1)$ such that with probability $1-e^{-cN}$ the following holds.
    For any $(\eps,\vDelta)$-critical point $\bx$, let $\by$ be its species-wise projection onto $\nabla H_N(\bzero)$, i.e.
    \[
        \by_s
        = \fr{\la \bx_s, (\nabla H_N(\bzero))_s \ra}{\tnorm{(\nabla H_N(\bzero))_s}_2^2} (\nabla H_N(\bzero))_s
    \]
    for all $s\in \sS$. Then
    \[
        \norm{\by - \scl\lt(\nabla H_N(\bzero); \vDelta, \nabla \xi(\vzero) / \nabla \xi(\vone)\rt)}_2
        \le \ups \sqrt{N}\,.
    \]
\end{corollary}
\begin{proof}
    We can write
    \[
        \by_s
        = \fr{\la \bx_s, \bG^{(1)}_s \ra}{\tnorm{\bG^{(1)}_s}_2^2} \bG^{(1)}_s
        = \fr{\la \bx_s, \bG^{(1)}_s \ra}{\tnorm{\bG^{(1)}_s}_2}
        \cdot \fr{\bG^{(1)}_s}{\tnorm{\bG^{(1)}_s}_2}\,.
    \]
    By Proposition~\ref{prop:approx-crits}, with probability $1-e^{-cN}$ all $(\eps,\vDelta)$-critical point $\bx$ satisfies $\la \bx_s, \bG^{(1)}_s \ra = \fr{\Delta_s \gamma_s}{\sqrt{\xi'_s}} \cdot \lambda_s N (1 + o_\eps(1))$, and by a standard concentration bound $\tnorm{\bG^{(1)}_s}_2 = \sqrt{\lambda_s N} (1 + o_\eps(1))$.
    So up to $1+o_\eps(1)$ multiplicative error
    \[
        \by_s
        = \fr{\Delta_s\gamma_s}{\sqrt{\xi'_s}} \cdot \sqrt{\lambda_s N} \fr{\bG^{(1)}_s}{\tnorm{\bG^{(1)}_s}_2}
        = \Delta_s \sqrt{\fr{\partial_s \xi(\vzero)}{\partial_s \xi(\vone)} \lambda_s N}
        \fr{\bG^{(1)}_s}{\tnorm{\bG^{(1)}_s}_2}\,.
    \]
    The result follows because
    \[
      \scl\lt(\nabla H_N(\bzero); \vDelta, \nabla \xi(\vzero) / \nabla \xi(\vone)\rt)
      = \scl\lt(\bG^{(1)}; \vDelta, \nabla \xi(\vzero) / \nabla \xi(\vone)\rt)\,.\qedhere
  \]
\end{proof}

\subsection{Conditional Band Models}
\label{subsec:6-2}

Our arguments rely on a self-similarity in law obtained by restriction to a band.
The point is that bands inside $\cS_N$ are still of the same form as our original model, with $N_s$ replaced by $N_s -1$ and $\xi$ replaced by a new mixture function. This lets us apply the preceding results of this paper to said bands.
This idea has been used extensively in recent work by Subag, e.g. \cite{subag2018free,subag2021tap1,subag2021tap2}.

For the below definitions, $U$ is a species-aligned subspace (recall Definition~\ref{dfn:species-aligned}), which in our applications will always be of dimension $O(1)$.
Let
\[
    \cB_N = \lt\{
        \bx \in \bbR^N :
        \tnorm{\bx_s}_2^2 \le \lambda_s N \quad
        \forall s\in \sS
    \rt\}
\]
be the convex hull of $\cS_N$.
Over the course of the recursive argument, we will be interested in the landscape of various Hamiltonians in the following domains where we project out the subspace $U$.
The original model corresponds to $U = \emptyset$.
\begin{definition}
    Let $\cS_N^U = \cS_N \cap U^\perp$ and $\cB_N^U = \cB_N \cap U^\perp$ (recalling the notation \eqref{eq:multi-perp}).
\end{definition}
\begin{definition}
    Let $\bm \in \cB_N^U$ such that $\bm_s \neq \bzero$ for all $s\in \sS$.
    The band of $\cS_N^U$ centered at $\bm$ is
    \[
        \Band^U(\bm) = \lt\{
            \bsig \in \cS_N^U : \vR(\bsig - \bm, \bm) = \vzero
        \rt\}\,.
    \]
\end{definition}
Note that for $\vq = \vR(\bm,\bm) \in [0,1]^r$ and $U \bowtie \bm = \Span(U,\bm_1,\ldots,\bm_s)$,
\begin{equation}
    \label{eq:band-affine}
    \Band^U(\bm) = (\vone - \vq)^{1/2} \diamond \cS_N^{U \bowtie \bm} + \bm\,.
\end{equation}
We will be interested in the following bands, whose centers are rescalings of $\nabla H_N(\bzero)$ (projected to $U^\perp$).
\begin{definition}
    \label{def:gradient-aligned-band}
    For $\vq \in [0,1]^r$ and $\vDelta \in \{-1,1\}^r$, define
    \begin{align*}
        \bm^U_{\vDelta,\vq}(H_N) &= \scl(\Proj_{U^\perp} \nabla H_N(\bzero);\vDelta,\vq)\,, &
        \Band^U_{\vDelta,\vq}(H_N) &= \Band^U(\bm^U_{\vDelta,\vq}(H_N))\,.
    \end{align*}
\end{definition}
We will abbreviate these $\bm^U_{\vDelta,\vq}$ and $\Band^U_{\vDelta,\vq}$ when $H_N$ is clear.
The following corollary shows that all critical points of $H_N$ in $\cS_N^U$ lie near one of these bands, given by a specific $\vq$.
\begin{corollary}
    \label{cor:crit-pt-band-depth}
    There exists $\ups = o_\eps(1)$ such that the following holds.
    Let $U$ be a species-aligned subspace of dimension $O(1)$.
    With probability $1-e^{-cN}$, all $(\eps,\vDelta)$-critical points of the restriction of $H_N$ to the manifold $\cS_N^U$
    lie within $\ups \sqrt{N}$ of $\Band^U_{\vDelta,\vq}(H_N)$, where $\vq = \nabla \xi(\vzero)/ \nabla \xi(\vone)$.
\end{corollary}
\begin{proof}
    Immediate from Corollary~\ref{cor:approx-crits}.
    Since $U$ has dimension $O(1)$, $H_N$ restricted to $\cS_N^U$ is a multi-species spin glass whose species dimensions $N'_s = N_s - O(1)$ still satisfy $N'_s / N \to \lambda_s$.
    Thus restricting the model to $\cS_N^U$ does not affect the result.
\end{proof}

Finally let $U' = U \bowtie \bm^U_{\vDelta,\vq}$~.
We define the centered band Hamiltonian $H^U_{N,\vDelta,\vq}(\bsig) : \cS_N^{U'} \to \bbR$ by
\begin{equation}
    \label{eq:band-hamiltonian}
    H^U_{N,\vDelta,\vq}(\bsig)
    = H_N\lt( (\vone-\vq)^{1/2} \diamond \bsig + \bm^U_{\vDelta,\vq}\rt)
    - \lt\la \Proj_{U^\perp} \nabla H_N(\bzero),
        (\vone-\vq)^{1/2} \diamond \bsig + \bm^U_{\vDelta,\vq}
    \rt\ra\,.
\end{equation}

The following lemma shows that conditional on $\Proj_{U^\perp} \nabla H_N(\bzero)$, $H^U_{N,\vDelta,\vq}$ is itself a multi-species spin glass.
Note that the last term of \eqref{eq:band-hamiltonian} is constant for $\bsig \in \cS_N^{U'}$, and that $(\vone-\vq)^{1/2} \diamond \bsig + \bm^U_{\vDelta,\vq}$ ranges over $\Band_{\vDelta,\vq}^U$ as $\bsig$ ranges over $\cS_N^{U'}$.
Thus $H^U_{N,\vDelta,\vq}$ is the remaining randomness of $H_N$ on this band.
\begin{lemma}
    \label{lem:conditional-formula}
    Conditionally on $\Proj_{U^\perp} \nabla H_N(\bzero)$, $H^U_{N,\vDelta,\vq}$ is a centered Gaussian process on $\cS_N^{U'}$ with covariance
    \begin{align*}
        \E[H^U_{N,\vDelta,\vq}(\bsig) H^U_{N,\vDelta,\vq}(\brho)]
        &=
        N\xi_{\vq}(\vR(\bsig,\brho))\,,
        \qquad \text{where} \\
        \xi_{\vq}(\vx)
        &= \xi\lt((\vone-\vq)\odot \vx + \vq\rt)
        - \lt\la \nabla \xi(\vzero), (\vone - \vq) \odot \vx + \vq \rt\ra\,.
    \end{align*}
\end{lemma}
\begin{proof}
    For $\bsig \in \cS_N^U$, we have
    \[
        H_N(\bsig) = \la \Proj_{U^\perp} \nabla H_N(\bzero), \bsig \ra + H_{N,\ge2}(\bsig)\,,
    \]
    where $H_{N,\ge2}$ consists of the interactions of degree at least $2$.
    These two constituent functions are, respectively, $\Proj_{U^\perp} \nabla H_N(\bzero)$-measurable and independent of $\Proj_{U^\perp} \nabla H_N(\bzero)$, while $\bm^U_{\vDelta,\vq}$ and $\Band^U_{\vDelta,\vq}$ are both $\Proj_{U^\perp} \nabla H_N(\bzero)$-measurable.
    Since
    \[
        H^U_{N,\vDelta,\vq}(\bsig)
        = H_{N,\ge2} \lt( (\vone-\vq)^{1/2} \diamond \bsig + \bm^U_{\vDelta,\vq}\rt),
    \]
    this is a centered Gaussian process.
    Moreover, $H_{N,\ge2}$ has covariance
    \begin{align*}
        \E[H_{N,\ge2}(\bsig) H_{N,\ge2}(\brho)] &= N\xi_{\ge2}(\vR(\bsig,\brho))\,,
        &&\text{where}&
        \xi_{\ge2}(\vx) &= \xi(\vx) - \la \nabla \xi(\vzero), \vx\ra\,.
    \end{align*}
    The covariance formula for $H^U_{N,\vDelta,\vq}$ now follows because for $\bsig,\brho \in \cS_N^{U'}$, we have $\bsig,\brho \in (\bm^U_{\vDelta,\vq})^\perp$, and so
    \[
        \vR\lt(
            (\vone-\vq)^{1/2} \diamond \bsig + \bm^U_{\vDelta,\vq}\,,
            (\vone-\vq)^{1/2} \diamond \brho + \bm^U_{\vDelta,\vq}
        \rt)
        = (\vone - \vq) \diamond \vR(\bsig,\brho) + \vq\,.
    \qedhere
    \]
\end{proof}

\subsection{Recursive Algorithm}
\label{subsec:6-3}

We now consider a recursive \emph{critical point finding} algorithm.
Roughly speaking, Corollary~\ref{cor:crit-pt-band-depth} shows that all $(\eps,\vDelta)$-critical points of $H_N$ lie near a band $\Band_1(\vDelta) = \Band^\emptyset_{\vDelta,\vq}$, for a deterministic $\vq$.
Lemma~\ref{lem:conditional-formula} shows that $H_N$ restricted to this band is conditionally another multi-species spin glass.
So, Corollary~\ref{cor:crit-pt-band-depth} implies all $(\eps,\vDelta)$-critical points lie near a sub-band $\Band_2(\vDelta) \subseteq \Band_1(\vDelta)$.
Repeating this argument, all $(\eps,\vDelta)$-critical points of $H_N$ lie near a nested sequence of bands $\cS_N \supseteq \Band_1(\vDelta) \supseteq \Band_2(\vDelta) \cdots$, and we will show these bands' diameters shrink to $0$.
After a large constant number of recursions, this shows all $(\eps,\vDelta)$-critical points lie in a region of diameter $o_\eps(\sqrt{N})$, and the well-conditionedness of $\nabla^2_\sph H_N$ (by Proposition~\ref{prop:approx-crits}) shows there is at most one critical point in this region.

We now define the recursive bands, starting with sequence of radii of their centers.
Define $\vR^0 = \vzero$ and recursively
\[
    \vR^{k+1} = \nabla \xi(\vR^k) / \nabla \xi(\vone)\,.
\]
Because $\xi$ is coordinate-wise increasing, the sequence $\vR^k$ is coordinate-wise increasing up to some limit in $[0,1]^r$.
The following lemma, which relies on super-solvability of $\xi$, shows this limit is $\vone$.
That is, the band centers approach the surface $\cS_N$ of $\cB_N$ and the band diameters limit to zero.

\begin{lemma}[{\cite[Lemma 2.3]{huang2023optimization}}]
    \label{lem:recursion-termination}
    We have that $\lim_{k\to\infty} \vR^k=\vone$.
\end{lemma}
Fix $\vDelta\in\{-1,1\}^r$.
Let $\bm^0(\vDelta)=\bzero$ and $\Band_0(\vDelta)=\cS_N$.
Recursively for $k\geq 1$ define
\begin{equation}
    \label{eq:def-recursive-algorithm}
    \begin{aligned}
    \bm^k(\vDelta)
    &=
    \bm^{k-1}(\vDelta) + \scl(\bg^{k-1}(\vDelta); \vDelta, \vR^k-\vR^{k-1} ) \,,
    \\
    \bg^k(\vDelta)
    &=
    \Proj_{U_k(\vDelta)^{\perp}} \nabla H_N(\bm^k(\vDelta))\,,
    \\
    U_k(\vDelta)
    &=
    \Span(\bm^j_s(\vDelta))_{1\le j\le k, s\in \sS}\,,
    \\
    \Band_k(\vDelta)
    &=
    \cS_N\cap \big(\bm^k(\vDelta) + U_k(\vDelta)^{\perp}\big)\,.
\end{aligned}
\end{equation}
Note that because $\bg^k_s(\vDelta) \in U_k(\vDelta)^{\perp}$ for all $s\in \sS$, we have $\bm^k(\vDelta) \in (\bm^{k-1}(\vDelta) + U_{k-1}(\vDelta)^{\perp})$, so $\Band_k(\vDelta) \subseteq \Band_{k-1}(\vDelta)$.
Also, by induction $\vR(\bm^k(\vDelta),\bm^k(\vDelta))=\vR^k$ for each $k\ge 0$, so analogously to \eqref{eq:band-affine},
\begin{align}
    \label{eq:recursive-band}
    \Band_k(\vDelta) &= \phi_{k,\vDelta}(\cS_N^{U_k(\vDelta)})\,,
    &\text{where}&&
    \phi_{k,\vDelta}(\bsig) = (\vone - \vR^k)^{1/2} \diamond \bsig + \bm^k(\vDelta)\,.
\end{align}
Define the band Hamiltonian $\wt H_{N,\vDelta,k}: \cS_N^{U_k(\vDelta)} \to \bbR$ by
\begin{equation}
    \label{eq:recursive-band-hamiltonian}
    \wt H_{N,\vDelta,k}(\bsig)
    = H_N\lt(\phi_{k,\vDelta}(\bsig)\rt)
    - \lt\la
        \bg^{k-1}(\vDelta),
        \phi_{k,\vDelta}(\bsig)
    \rt\ra
    + \sum_{i=1}^{k-1}
    \lt\la
        \bg^i(\vDelta)-\bg^{i-1}(\vDelta),
        \bm^i(\vDelta)
    \rt\ra\,,
\end{equation}
whose meaning is explained in Lemma~\ref{lem:recursive-conditional-formula} below.
We first show that the bands $\Band_k(\vDelta)$ can be constructed by applying the construction from Definition~\ref{def:gradient-aligned-band} recursively.
The radius $\vq^k$ of the band center in the next lemma is chosen to be near all $(\eps,\vDelta)$-critical points of the Hamiltonian restricted to $\Band_k(\vDelta)$, as will be explained in Corollary~\ref{cor:recursive-crit-pt-band-depth} below.
\begin{lemma}
    \label{lem:bands-are-recursive}
    Let $\vq^k = (\vR^{k+1} - \vR^k) / (\vone - \vR^k)$ and note that $\phi_{k,\vDelta}^{-1}$ maps $\Band_k(\vDelta)$ to $\cS_N^{U_k(\vDelta)}$.
    Then $\phi_{k,\vDelta}^{-1}(\Band_{k+1}(\vDelta)) = \Band^{U_k(\vDelta)}_{\vDelta,\vq^k}(\wt H_{N,\vDelta,k})$, where the latter band is defined in Definition~\ref{def:gradient-aligned-band}.
\end{lemma}
\begin{proof}
    Since $\phi_{k,\vDelta}^{-1}(\bsig) = (1-\vR^k)^{-1/2} \diamond (\bsig - \bm_k(\vDelta))$, we have
    \begin{align}
        \notag
        \phi_{k,\vDelta}^{-1}(\bm^{k+1}(\vDelta))
        &= (\vone-\vR^k)^{-1/2} \diamond \scl(\bg^k(\vDelta); \vDelta,\vR^{k+1}-\vR^k) \\
        \label{eq:band-step}
        &= \scl\lt(
            \bg^k(\vDelta);
            \vDelta,
            \vq^k
        \rt)
        = \bm^{U_k(\vDelta)}_{\vDelta,\vq^k}(\wt H_{N,\vDelta,k})\,.
    \end{align}
    Let us denote this point $\bm$.
    Moreover,
    \begin{equation}
        \label{eq:bowtie-equivalence}
        \phi_{k,\vDelta}^{-1}(\Band_{k+1}(\vDelta))
        = \phi_{k,\vDelta}^{-1}(\phi_{k+1,\vDelta}(\cS_N^{U_{k+1}(\vDelta)}))
        = \lt(\fr{\vone - \vR^k}{\vone - \vR^{k+1}}\rt)^{1/2} \diamond \cS_N^{U_{k+1}(\vDelta)}
        + \bm\,.
    \end{equation}
    Also,
    \[
        U_{k+1}(\vDelta)
        = U_k(\vDelta) \bowtie \bm^k(\vDelta)
        = U_k(\vDelta) \bowtie \bg^k(\vDelta)
        = U_k(\vDelta) \bowtie \bm\,,
    \]
    so elements of $\cS_N^{U_{k+1}(\vDelta)}$ are orthogonal to $\bm$.
    This implies the conclusion.
\end{proof}
Let $F_k(\vDelta)=(\bg^0(\vDelta), \bg^1(\vDelta), \ldots \bg^{k-1}(\vDelta))$.
The following lemma shows that conditional on $F_k(\vDelta)$, $\wt H_{N,\vDelta,k}$ is a multi-species spin glass.
Moreover, all terms on the right-hand side of \eqref{eq:recursive-band-hamiltonian} except $H_N(\phi_{k,\vDelta}(\bsig))$ are constant on $\cS_N^{U_k(\vDelta)}$, so $\wt H_{N,\vDelta,k}$ is the remaining randomness of $H_N$ on $\cS_N^{U_k(\vDelta)}$.
\begin{lemma}
    \label{lem:recursive-conditional-formula}
    Conditional on $F_k(\vDelta)$, $\wt H_{N,\vDelta,k}$ is a centered Gaussian process on $\cS_N^{U_k(\vDelta)}$ with covariance
    \[
        \E\lt[
            \wt H_{N,\vDelta,k} (\bsig)
            \wt H_{N,\vDelta,k} (\brho)
        \rt]
        = N\xi_k(\vR(\bsig,\brho))\,,
    \]
    where
    \begin{align*}
        \xi_k(\vx)
        &= \xi\lt((\vone - \vR^k) \odot \vx + \vR^k\rt)
        - \lt\la
            \nabla \xi(\vR^{k-1}),
            (\vone - \vR^k)\odot \vx + \vR^k
        \rt\ra \\
        &\qquad + \sum_{i=1}^{k-1}
        \lt\la
            \nabla \xi(\vR^i) - \nabla \xi(\vR^{i-1}),
            \vR^i
        \rt\ra\,.
    \end{align*}
\end{lemma}
\begin{proof}
    We induct on $k$.
    Assume the claim holds for $k$ and let $\bm = \scl(\bg^k(\vDelta);\vDelta,\vq^k)$ as in \eqref{eq:band-step}.
    Lemma~\ref{lem:conditional-formula} implies that conditional on $(F_k(\vDelta), P_{U_k(\vDelta)^\perp} \nabla \wt H_{N,\vDelta,k}(\bzero))$ the function $\wt H_{N,\vDelta,k}^\circ : \cS_N^{U_k(\vDelta) \bowtie \bm}\to\bbR$ given by
    \begin{align*}
        \wt H_{N,\vDelta,k}^\circ(\bsig)
        &=
        \wt H_{N,\vDelta,k}\lt((\vone - \vq^k)^{1/2} \diamond \bsig + \bm\rt) \\
        &\qquad - \lt\la
            P_{U_k(\vDelta)^\perp} \nabla \wt H_{N,\vDelta,k}(\bzero),
            (\vone - \vq^k)^{1/2} \diamond \bsig + \bm
        \rt\ra
    \end{align*}
    is a centered Gaussian process.
    We calculate that
    \begin{align*}
        \Proj_{U_k(\vDelta)^\perp} \nabla \wt H_{N,\vDelta,k}(\bzero)
        &= \Proj_{U_k(\vDelta)^\perp}
        (\vone - \vR^k)^{1/2} \diamond \lt(
            \nabla H_N(\bm^k(\vDelta)) -
            \bg^{k-1}(\vDelta)
        \rt) \\
        &= (\vone - \vR^k)^{1/2} \diamond (\bg^k(\vDelta) - \bg^{k-1}(\vDelta))\,.
    \end{align*}
    Thus conditioning on $(F_k(\vDelta), P_{U_k(\vDelta)^\perp} \nabla \wt H_{N,\vDelta,k}(\bzero))$ is equivalent to conditioning on $F_{k+1}(\vDelta)$.
    Also, \eqref{eq:bowtie-equivalence} shows $U_k(\vDelta) \bowtie \bm = U_{k+1}(\vDelta)$.
    For all $\bsig \in \cS_N^{U_{k+1}(\vDelta)}$,
    \[
        \phi_{k,\vDelta}\lt((\vone - \vq^k)^{1/2} \diamond \bsig + \bm\rt)
        = (\vone - \vR^{k+1})^{1/2} \diamond \bsig + (\vone - \vR^k)^{1/2} \diamond \bm + \bm^k(\vDelta)\,,
    \]
    and this equals $\phi_{k+1,\vDelta}(\bsig) = (\vone - \vR^{k+1})^{1/2} \diamond \bsig + \bm^{k+1}(\vDelta)$ because
    \[
        (\vone - \vR^k)^{1/2} \diamond \bm
        = \scl(\bg^k(\vDelta);\vDelta,\vR^{k+1}-\vR^k)
        = \bm^{k+1}(\vDelta) - \bm^k(\vDelta)\,.
    \]
    Moreover
    \begin{align*}
        &\lt\la
            P_{U_k(\vDelta)^\perp} \nabla \wt H_{N,\vDelta,k}(\bzero),
            (\vone - \vq^k)^{1/2} \diamond \bsig + \bm
        \rt\ra \\
        &=
        \lt\la
            (\vone - \vR^k)^{1/2} \diamond (\bg^k(\vDelta) - \bg^{k-1}(\vDelta)),
            (\vone - \vq^k)^{1/2} \diamond \bsig + \bm
        \rt\ra \\
        &= \lt\la
            \bg^k(\vDelta) - \bg^{k-1}(\vDelta),
            (\vone - \vR^{k+1})^{1/2} \diamond \bsig + \bm^{k+1}(\vDelta) - \bm^k(\vDelta)
        \rt\ra \\
        &= \lt\la
            \bg^k(\vDelta) - \bg^{k-1}(\vDelta),
            \phi_{k+1,\vDelta}(\bsig)
        \rt\ra
        - \lt\la
            \bg^k(\vDelta) - \bg^{k-1}(\vDelta),
            \bm^k(\vDelta)
        \rt\ra.
    \end{align*}
    Combining the above,
    \begin{align*}
        \wt H_{N,\vDelta,k}^\circ(\bsig)
        &= H_N(\phi_{k+1,\vDelta}(\bsig))
        - \lt\la
            \bg^{k-1}(\vDelta),
            \phi_{k+1,\vDelta}(\bsig)
        \rt\ra
        + \sum_{i=1}^{k-1} \lt\la
            \bg^i(\vDelta) - \bg^{i-1}(\vDelta),
            \bm^i
        \rt\ra \\
        &\qquad - \lt\la
            \bg^k(\vDelta) - \bg^{k-1}(\vDelta),
            \phi_{k+1,\vDelta}(\bsig)
        \rt\ra
        + \lt\la
            \bg^k(\vDelta) - \bg^{k-1}(\vDelta),
            \bm^k(\vDelta)
        \rt\ra \\
        &= \wt H_{N,\vDelta,k+1}(\bsig)\,.
    \end{align*}
    This proves that conditional on $F_{k+1}(\vDelta)$, $\wt H_{N,\vDelta,k}$ is a centered Gaussian process on $\cS_N^{U_{k+1}(\vDelta)}$.
    The covariance formula is shown by similarly verifying that
    \[
        \xi_k\lt((\vone - \vq^k)\odot \vx + \vq^k\rt)
        - \lt\la
            \nabla \xi_k(\vzero),
            (\vone - \vq^k)\odot \vx + \vq^k
        \rt\ra
        = \xi_{k+1}\lt((\vone - \vq^{k+1}) \odot \vx + \vq^{k+1}\rt)
    \]
    This completes the induction.
\end{proof}

We next verify that if $\xi$ is strictly super-solvable, then so is $\xi_k$.
This fact is key for our recursion.

\begin{proposition}
\label{prop:xi-k-supersolvable}
    For each $k$, the model $\xi_k$ is strictly super-solvable.
\end{proposition}

\begin{proof}
    The definition $\vR^k = \nabla \xi(\vR^{k-1}) / \nabla \xi(\vone)$ rearranges to
    \begin{equation}
        \label{eq:amp-identity}
        \nabla \xi(\vone) - \nabla \xi(\vR^{k-1})
        = (\vone - \vR^k) \odot \nabla \xi(\vone)\,.
    \end{equation}
    Thus,
    \begin{align}
        \label{eq:xik-deriv-vone}
        \nabla \xi_k(\vone)
        &=
        (\vone-\vR^k) \odot
        (\nabla \xi(\vone) - \nabla \xi(\vR^{k-1}))
        = (\vone-\vR^k)^2 \odot \nabla \xi(\vone)\,, \\
        \label{eq:xik-2deriv-vone}
        \nabla^2 \xi_k(\vone)
        &=
        (\vone-\vR^k)^{\otimes 2} \odot \nabla^2 \xi''(\vone)\,.
    \end{align}
    Combining the above gives
    \[
        \diag(\nabla \xi_k(\vone))
        = (\vone-\vR^k)^{\otimes 2} \odot \diag(\nabla \xi(\vone))
        \succ (\vone-\vR^k)^{\otimes 2} \odot \nabla^2 \xi''(\vone)
        = \nabla^2 \xi_k(\vone)\,.
    \qedhere
    \]
\end{proof}
Finally the next corollary explains the choice of radius $\vq^k$.
Combined with Lemma~\ref{lem:bands-are-recursive}, this explains the choice of radii $\vR^k$ in the construction of the bands $\Band_k(\vDelta)$: $\vR^{k+1}$ is chosen so that $\Band_{k+1}(\vDelta)$ is the sub-band of $\Band_k(\vDelta)$ orthogonal to $\bg^k$ which lies near all approximate critical points of type $\vDelta$.
\begin{corollary}
    \label{cor:recursive-crit-pt-band-depth}
    With probability $1-e^{-cN}$, all $(\eps,\vDelta)$-critical points of $\wt H_{N,\vDelta,k}$ on $U_k(\vDelta)$ lie within $\ups \sqrt{N}$ of $\Band^{U_k(\vDelta)}_{\vDelta,\vq^k}(\wt H_{N,k,\vDelta})$ for some $\ups = o_\eps(1)$.
\end{corollary}
\begin{proof}
    Since $\xi_k$ is strictly super-solvable by Proposition~\ref{prop:xi-k-supersolvable}, Corollary~\ref{cor:crit-pt-band-depth} applies.
    Recalling \eqref{eq:amp-identity}, we have
    \begin{align*}
        \nabla \xi_k(\vzero)
        &= (\vone-\vR^k) \odot (\nabla \xi(\vR^k) - \nabla \xi(\vR^{k-1})) \\
        &= (\vone-\vR^k) \odot \lt(
            (\nabla \xi(\vone) - \nabla \xi(\vR^{k-1})) -
            (\nabla \xi(\vone) - \nabla \xi(\vR^{k}))
        \rt) \\
        &= (\vone-\vR^k) \odot (\vR^{k+1} - \vR^k) \odot \nabla \xi(\vone)\,.
    \end{align*}
    Recalling \eqref{eq:xik-deriv-vone}, this implies $\nabla \xi_k(\vzero) / \nabla \xi_k(\vone) = (\vR^{k+1}-\vR^k)/(\vone - \vR^k) = \vq^k$.
    The result follows from Corollary~\ref{cor:crit-pt-band-depth}.
\end{proof}

\subsection{Localization of Approximate Critical Points}
\label{subsec:6-4}

For $S\subseteq \bbR^N$ and $\iota > 0$, let $B_\iota(S) \subseteq \bbR^N$ denote the set of points whose distance to $S$ is at most $\iota$.
The following proposition localizes all $(\eps,\vDelta)$-critical points of $H_N$.
Note that by Fact~\ref{fac:type-partition}, all $\eps$-critical points of $H_N$ are described by this proposition.
\begin{proposition}
    \label{prop:crit-pts-localization}
    For any constant $k \in \bbN$, $\eps > 0$, there exists $\iota_k = o_\eps(1)$ (depending on $k$) such that with probability $1-e^{-cN}$, all $(\eps,\vDelta)$-critical points of $H_N$ lie in $B_{\iota_k \sqrt{N}}(\Band_k(\vDelta))$.
\end{proposition}
We begin by relating the $(\eps,\vDelta)$-critical points of $H_N$ to those of $\wt H_{N,\vDelta,k}$.
\begin{lemma}
    \label{lem:Delta-consistency}
    For any $k\ge 1$, $\eps > 0$, $\vDelta \in \{-1,1\}^r$ the following holds.
    If $\bsig\in \Band_k(\vDelta)$ is an $(\eps,\vDelta)$-critical point of $H_N$, then $\phi_{k,\vDelta}^{-1}(\bsig) \in \cS_N^{U_k(\vDelta)}$ is a $(\eps,\vDelta)$-critical point of $\wt H_{N,\vDelta,k}$.
\end{lemma}
\begin{proof}
    Let $\vx(\vDelta,\xi) \in \bbR^r$ be the radial derivative defined in \eqref{eq:ideal-stats}, where we make the dependence on $\xi$ explicit.
    By definition of $(\eps,\vDelta)$-approximate critical point,
    \[
        \tnorm{\nabla H_N(\bsig) - \Lambda^{-1/2}\vx(\vDelta,\xi) \diamond \bsig}_2 \le \ups \sqrt{N}\,.
    \]
    We write $\bsig = \phi_{k,\vDelta}(\brho)$ for $\brho \in \cS_N^{U_k(\vDelta)}$.
    Let $\nabla_{U_k(\vDelta)^\perp}$ denote the Euclidean gradient projected into the subspace $U_k(\vDelta)^\perp$.
    Because this projection is $1$-Lipschitz, we also have
    \begin{equation}
        \label{eq:Delta-consistency-1}
        \tnorm{
            \nabla_{U_k(\vDelta)^\perp} H_N(\bsig)
            - \Lambda^{-1/2}\vx(\vDelta,\xi) \diamond \Proj_{U_k(\vDelta)^\perp} \bsig
        }_2 \le \ups \sqrt{N}\,.
    \end{equation}
    Taking this gradient of \eqref{eq:recursive-band-hamiltonian} yields
    \[
        \nabla_{U_k(\vDelta)^\perp}
        \wt H_{N,\vDelta,k}(\brho)
        = (\vone - \vR^k)^{1/2} \diamond
        \nabla_{U_k(\vDelta)^\perp}
        H_N(\phi_{k,\vDelta}(\brho))
        = (\vone - \vR^k)^{1/2} \diamond
        \nabla_{U_k(\vDelta)^\perp}
        H_N(\bsig)\,,
    \]
    as the gradient contribution from $\bg^{k-1}(\vDelta)$ projects to zero.
    Moreover,
    \[
        \Proj_{U_k(\vDelta)^\perp} \bsig
        = \Proj_{U_k(\vDelta)^\perp} \phi_{k,\vDelta}(\brho)
        = (\vone - \vR^k)^{1/2} \diamond \brho\,.
    \]
    From \eqref{eq:xik-deriv-vone} and \eqref{eq:xik-2deriv-vone}, it readily follows that
    \[
        \vx(\vDelta,\xi_k) = (\vone - \vR^k) \odot \vx(\vDelta,\xi)\,.
    \]
    Thus \eqref{eq:Delta-consistency-1} implies
    \begin{align*}
        \ups \sqrt{N}
        &\ge
        \tnorm{
            (\vone - \vR^k)^{-1/2} \diamond
            \nabla_{U_k(\vDelta)^\perp} \wt H_{N,\vDelta,k}(\brho)
            - (\Lambda^{-1/2} \vx(\vDelta,\xi) \odot
            (\vone - \vR^k)^{1/2}) \diamond \brho
        }_2 \\
        &= \tnorm{
            (\vone - \vR^k)^{-1/2} \diamond
            (\nabla_{U_k(\vDelta)^\perp} \wt H_{N,\vDelta,k}(\brho)
            - \Lambda^{-1/2} \vx(\vDelta,\xi_k) \diamond \brho)
        }_2 \\
        &\ge \tnorm{
            \nabla_{U_k(\vDelta)^\perp} \wt H_{N,\vDelta,k}(\brho)
            - \Lambda^{-1/2} \vx(\vDelta,\xi_k) \diamond \brho
        }_2\,.
    \end{align*}
    So, $\brho = \phi_{k,\vDelta}^{-1}(\bsig)$ is an $(\eps,\vDelta)$-critical point of $\wt H_{N,\vDelta,k}$.
\end{proof}

\begin{proof}[Proof of Proposition~\ref{prop:crit-pts-localization}]
    Throughout we assume $H_N \in K_N$, which holds with probability $1-e^{-cN}$ by Proposition~\ref{prop:gradients-bounded}.
    We induct on $k$.
    Suppose the claim holds for $k$, so all $(\eps,\vDelta)$-critical points of $H_N$ lie in $B_{\iota_k\sqrt{N}} (\Band_k(\vDelta))$.
    Let $\bsig$ be one such critical point, and let $\brho \in \Band_k(\vDelta)$ be its projection in $\Band_k(\vDelta)$.
    Because $H_N \in K_N$, $\brho$ is a $(\eps',\vDelta)$-critical point of $H_N$ for some $\eps' = o_k(1)$.

    By Lemma~\ref{lem:Delta-consistency}, $\btau = \phi_{k,\vDelta}^{-1}(\brho) \in \cS_N^{U_k(\vDelta)}$ is a $(\eps',\vDelta)$-critical point of $\wt H_{N,k,\vDelta}$.
    By Corollary~\ref{cor:recursive-crit-pt-band-depth}, (with probability $1-e^{-cN}$)
    \[
        \btau \in B_{\ups \sqrt{N}}(\Band^{U_k(\vDelta)}_{\vDelta,\vq^k}(\wt H_{N,k,\vDelta}))
    \]
    for some $\ups = o_{\eps'}(1) = o_\eps(1)$.
    By Lemma~\ref{lem:bands-are-recursive},
    \[
        \phi_{k,\vDelta}(\Band^{U_k(\vDelta)}_{\vDelta,\vq^k}(\wt H_{N,k,\vDelta})) = \Band_{k+1}(\vDelta)\,,
    \]
    so (letting $C_k$ be the Lipschitz constant of $\phi_{k,\vDelta}$)
    \[
        \brho = \phi_{k,\vDelta}(\btau) \in B_{C_k \ups\sqrt{N}}(\Band_{k+1}(\vDelta))\,.
    \]
    It follows that $\bsig \in B_{\iota_{k+1}\sqrt{N}}$ for $\iota_{k+1} = C_k \ups + \iota_k$.
    Over this argument we union bounded over $k+1=O(1)$ events with probability $1-e^{-cN}$, so the conclusion holds with probability $1-e^{-cN}$.
\end{proof}

\subsection{Existence and Uniqueness of Exact Critical Points}
\label{subsec:6-5}

So far we have established that all $(\eps,\vDelta)$-critical points of $H_N$ are close together.
This easily implies that each $\vDelta$ has at most $1$ associated (exact) critical point.

\begin{definition}
    Let $\eps>0$ be a sufficiently small constant independent of $N$.
    A $\vDelta$-critical point $\bx_\vDelta$ of $H_N$ is a critical point that is also a $(\eps,\vDelta)$-critical point (i.e. whose radial derivative $\nabla_\rd H_N(\bx_\vDelta)$ satisfies \eqref{eq:type-radial-deriv}).
\end{definition}

\begin{proposition}
\label{prop:crits-unique}
    With probability $1-e^{-cN}$, for each $\vDelta\in\{-1,1\}^r$ there is at most one $\vDelta$-critical point of $H_N$.
\end{proposition}

\begin{proof}
    Let $\bx_{\vDelta},\bx_{\vDelta}'$ be two such critical points of $H_N$.
    Then they are both $(\eps,\vDelta)$ critical points for small $\eps>0$.
    By Fact~\ref{fac:type-partition} and Proposition~\ref{prop:crit-pts-localization}, we find that for any $\delta>0$ independent of $N$,
    \[
        \|\bx_{\vDelta}-\bx_{\vDelta}'\|_2 \leq \delta \sqrt{N}
    \]
    holds with probability $1-e^{-cN}$.
    Moreover \eqref{eq:annealed-kac-rice-atypical} and Lemma~\ref{lem:spec-no-0} together imply that with the same probability, the spherical Hessians of $H_N$ at both points are $C(\xi)$ well-conditioned, with all eigenvalues inside $\pm [C^{-1},C]$.
    For $\delta$ small enough and $H_N\in K_N$, this is impossible since $C$-well-conditioned critical points cannot be arbitrarily close together (as can be shown by Taylor expanding along a geodesic as in Lemma~\ref{lem:grad-y-taylor}).
\end{proof}

To show existence we appeal to Morse theory, which shows the total number of critical points is almost surely at least $2^r$ just from the geometry of $\cS_N$.

\begin{proposition}
\label{prop:morse}
    Almost surely, $H_N$ has at least $2^r$ critical points on $\cS_N$. Hence by Fact~\ref{fac:type-partition} and Proposition~\ref{prop:crits-unique}, with probability $1-e^{-cN}$, $H_N$ has a unique $\vDelta$-critical point of each type $\vDelta\in \{-1,1\}^r$, and no other critical points.
\end{proposition}

\begin{proof}
    It suffices to show the first claim. Recall that any sphere has two non-zero homology groups, each of dimension $1$. Hence by the K{\"u}nneth formula, the sum of the dimensions of the homology groups for $\cS_N$ is $2^r$. Finally by the Morse inequalities (see e.g. \cite{milnor1963morse}), this sum lower bounds the number of critical points of any Morse function, in particular $H_N$.
\end{proof}

Putting everything together, we obtain most of Theorem~\ref{thm:precise-landscape}.

\begin{proof}[Proof of Theorem~\ref{thm:precise-landscape} except for part~\ref{it:index}]
    Existence and uniqueness of each $\bx_{\vDelta}$ have just been shown.
    As above, \eqref{eq:annealed-kac-rice-atypical} and Lemma~\ref{lem:spec-no-0} imply the well-conditioning.
    Proposition~\ref{prop:crit-pts-localization} implies part~\ref{it:bsig-in-union}.
    Part~\ref{it:ground-states} is immediate from part~\ref{it:bsig-in-union} since all approximate ground states of $H_N\in K_N$ are approximate critical points.
\end{proof}

\begin{remark}
\label{rem:no-morse}
    As an alternative to the Morse inequalities, we could instead use \cite[Proposition 3.2]{huang2023optimization} which, for each $\vDelta\in\{-1,1\}^r$ and $\eps>0$, explicitly constructs a $(\vDelta,\eps)$-approximate critical point $\wt\bx_{\vDelta}\in\cS_N$ (with probability $1-e^{-cN}$).
    Since the limiting spectral support $S(\vDelta)$ is bounded away from $0$ by Proposition~\ref{lem:spec-no-0}, Proposition~\ref{prop:approx-crits} implies that for $\eps$ small enough, each $\wt\bx_{\vDelta}$ has well-conditioned Hessian.
    Then Newton's method can be used to locate a nearby exact critical point $\bx_{\vDelta}$. This route is more cumbersome than the one taken above, but has a chance to work in situations where the number of critical points in the trivial regime is larger than the lower bound from the Morse inequalities.
\end{remark}

\subsection{The Index of Each Critical Point}
\label{subsec:index}

Finally we compute the index of each critical point, which is the only remaining part of Theorem~\ref{thm:precise-landscape}.
We use a ``critical point following'' argument, showing that critical points move stably as $H_N$ is gradually deformed into a linear function, while their indices remain fixed.
This can easily be turned into an efficient algorithm to locate each $\bx_{\vDelta}$ as mentioned in the introduction, as the proof of \cite[Lemma 3.1]{montanari2023solving} used below is via projected gradient descent on $\|\nabla H_N(\cdot)\|_2^2$ (i.e. Newton's method).

\begin{proposition}
\label{prop:newton}
    For any $(\vlambda,\xi)$ and $\iota>0$ there is $\eps>0$ such that the following holds.
    Suppose $H_N\in K_N$, and $\nabla_{\sph}^2 H_N(\bx)$ is an $\iota$-well-conditioned $\eps$-approximate critical point. Then there exists an exact critical point $\by\in \cS_N$ such that $\|\bx-\by\|_2 \leq C(\vlambda,\xi,\iota)\eps \sqrt{N}$.
\end{proposition}

\begin{proof}
    This follows by \cite[Lemma 3.1]{montanari2023solving} applied to $\nabla H_N$; the constants $J_n, L_n, M_n$ are bounded as $H_N\in K_N$. (The stated result is for a single sphere, but the extension to a finite products of spheres poses no issues.)
\end{proof}

\begin{proof}[Proof of Theorem~\ref{thm:precise-landscape}\ref{it:index}]
Fix $(\vlambda,\xi)$.
Writing $\wt H_N$ for the degree two and higher terms in $H_N$, for $t\in [0,1]$ we set
\[
    H_{N,t}=\la \bG^{(1)},\bx\ra +t\wt H_N(\bx).
\]
The marginal distribution of $H_{N,t}$ thus corresponds to the mixture $\xi^{(t)}(\vx)=(1-t^2) \xi'(0) \odot\vx + t^2 \xi(\vx)$.
It is easy to see that if $\xi$ is strictly super-solvable then so is $\xi^{(t)}$ for each $t\in [0,1]$.
Moreover the proof of Lemma~\ref{lem:spec-no-0} holds uniformly on $(\xi^{(t)})_{t\in [0,1]}$, implying that for some $c>0$,
\[
    S_t(\vDelta)\cap [-c,c]=\emptyset
\]
holds simultaneously for all $t\in [0,1]$ and $\vDelta\in \{-1,1\}^r$.
In particular, our results then imply the following. Fix a small unit fraction $\delta>0$ depending on $(\vlambda,\xi,c)$, and let $\bx_{\vDelta,k\delta}$ be the corresponding critical point for $H_{N,k\delta}$ (which exists with probability $1-e^{-cN}$). Then for $k\geq 0$, if $H_N,\wt H_N\in K_N$:
\begin{align}
\nonumber
    \|\nabla_{\sph} H_{N,(k+1)\delta}(\bx_{\vDelta,k\delta})\|_2 \leq C\delta\sqrt{N},
    \\
\label{eq:uniformly-well-conditioned}
    \spec\big(\nabla_{\sph}^2 H_{N,(k+1)\delta}(\bx_{\vDelta,k\delta})\big)\cap [-c/2,c/2]=\emptyset.
\end{align}
For $H_N,\wt H_N\in K_N$, the above two estimates imply via Proposition~\ref{prop:newton} the existence of a nearby critical point $\by_{\vDelta,(k+1)\delta}$ for $H_{N,(k+1)\delta}$ such that, for a constant $C_1=C_1(\vlambda,\xi,c)$:
\begin{align}
\label{eq:critical-point-slow-movement}
    |\la \by_{\vDelta,(k+1)\delta}
    -
    \bx_{\vDelta,k\delta},\bG^{(1)}\ra|
    &\leq
    C_1\delta,
    \\
\label{eq:critical-point-spectral-movement}
    \|\nabla_{\sph}^2 H_{N,(k+1)\delta}(\by_{\vDelta,(k+1)\delta})-\nabla_{\sph}^2 H_{N,k\delta}(\bx_{\vDelta,k\delta})\|_{\op}
    &\leq
    C_1\delta.
\end{align}
Recalling \eqref{eq:ideal-stats} and \eqref{eq:1spin-correlation-good}, it follows from \eqref{eq:critical-point-slow-movement} that $\by_{\vDelta,(k+1)\delta}=\bx_{\vDelta,(k+1)\delta}$ is a critical point of the same $\vDelta$.
Combining \eqref{eq:uniformly-well-conditioned} and \eqref{eq:critical-point-spectral-movement}, we see that $\nabla_{\sph}^2 H_{N,(k+1)\delta}(\bx_{\vDelta,(k+1)\delta})$ and $\nabla_{\sph}^2 H_{N,k\delta}(\bx_{\vDelta,k\delta})$ have the same number of positive eigenvalues for each $\delta$. Taking $k=0$, it is easy to see that this number is $\sum_{s:\vDelta_s=-1} N_s$. Taking $k\delta=1$ shows that the same holds for $\nabla_{\sph}^2 H_N(\bx_{\vDelta})$ as desired.
\end{proof}

\section{Estimates for Approximate Critical Points in Single-Species Models}
\label{sec:approx-local-max-E-infty}

In this section we detail further consequences of Lemma~\ref{lem:approx-to-exact-computation} which are of independent interest, and of relevance for several concurrent works.
Below, we restrict our attention to single-species models without external field (i.e. $r=1$, $\xi'(0)=0$) for which the relevant Kac--Rice estimates are known from previous work.
In particular we consider mixture functions of the form $\xi(t)=\sum_{p=2}^P \gamma_p^2 t^p$ for $\gamma_2,\dots,\gamma_P\geq 0$.
We assume for sake of normalization that $\xi(1)=1$ and similarly to Definition~\ref{def:xi'}, we write $\xi'=\xi'(1),\xi''=\xi''(1)$ and $\alpha^2=\xi''+\xi'-(\xi')^2$ (unrelated to \eqref{eq:small-constants}).
Recall from \cite{auffinger2013complexity} the thresholds:
\[
    E_{\infty}^{\pm}(\xi)
    \equiv
    \frac{2 \xi' \sqrt{\xi''} \pm \sqrt{4 \xi'' (\xi')^2-\left(\xi''+\xi'\right)\left(2\left(\xi''-\xi'+(\xi')^2\right)-\alpha^2 \log \frac{\xi''}{\xi'}\right)}}{\xi'+\xi''}
    .
\]
One always has $\alpha\geq 0$, with equality exactly in the pure case $\xi(t)=t^p$ for some $p$. In this case, the thresholds $E_{\infty}^{\pm }$ agree at the value $E_{\infty}(p)=2\sqrt{\frac{p-1}{p}}$ from \cite{auffinger2013random}.

We give the relevant Kac--Rice result in Proposition~\ref{prop:kac-rice-E-infinity} below after recalling some definitions and results from \cite{auffinger2013complexity}. For open $\cD,\cD_{\rd}\subseteq \bbR$ we let $\Crt_{N}(\cD;\cD_{\rd})\subseteq \cS_N$ consist of all critical points with
\[
    H_N(\bx)/N\in \cD,
    \quad\quad
    \nabla_{\rd} H_N(\bx) \in \cD_{\rd}.
\]
As in Equation (1.21) therein, for $\gamma \in (0,1)$ define $s_{\gamma}\in (-\sqrt{2},\sqrt 2)$ as the rescaled semicircular law quantile satisfying:
\[
    \gamma
    =
    \frac{1}{\pi}
    \int_{-\sqrt{2}}^{-s_{\gamma}} \sqrt{2-x^2}
    ~\de x.
\]
Moreover, define the function
\begin{equation}
\label{eq:Theta-formula}
    \Theta(s)=\lt(
    -\frac{|s|\sqrt{s^2-2}}{2}
    +
    \log
    \lt(
    \frac{|s|+\sqrt{s^2-2}}{\sqrt{2}}
    \rt)
    \rt) 1_{|s|\geq \sqrt 2}
    \leq 0.
\end{equation}
The critical point complexity functional at $\big(H_N(\bx)/N,\nabla_{\rd}H_N(\bx)\big)\approx \big(y,s\sqrt{2\xi''}\big)$ and its quadratic upper bound are given by:
\begin{equation}
\label{eq:single-species-complexity-function}
\begin{aligned}
    F(s,y)
    &=
    \frac{1}{2}
    \lt(
    \log\frac{\xi''}{\xi'}
    +
    s^2-y^2
    -
    \frac{2\xi''}{\alpha^2}
    \lt(
    s-\frac{y\xi'}{\sqrt{2\xi''}}
    \rt)^2
    +
    \Theta(s)
    \rt),
    \\
    \wt F(s,y)
    &=
    \frac{1}{2}
    \lt(
    \log\frac{\xi''}{\xi'}
    +
    s^2-y^2
    -
    \frac{2\xi''}{\alpha^2}
    \lt(
    s-\frac{y\xi'}{\sqrt{2\xi''}}
    \rt)^2
    \rt).
\end{aligned}
\end{equation}
(If $\alpha=0$, we interpret $-0/0=0$ and $-x/0=-\infty$ for $x>0$.)

Indeed the following holds as a direct consequence of Proposition~\ref{prop:annealed-kac-rice}, see also the proof of \cite[Theorem 1.3]{auffinger2013complexity}. (In fact our scaling of $\nabla_{\rd}H_N(\bx)$ by $\sqrt{2\xi''}$ is chosen to enforce agreement with the latter formula).

\begin{proposition}
\label{prop:auffinger-ben-arous-rewrite}
    For any $\gamma\in (0,1)$ and open $\cD,\cD_{\rd}\subseteq\bbR$:
    \[
    \lim_{N\to\infty}
    \frac{1}{N}
    \log
    \bbE
    \big|\Crt_{N}\big(
    \cD~;\cD_{\rd}
    \big)\big|
    =
    \sup_{
    \substack{y\in \cD,
    \\
    s\sqrt{2\xi''}\in\cD_{\rd}}
    }
    F(s,y).
    \]
\end{proposition}

Define the set
\[
    S=S_{\xi}=\{(s,y)\in \bbR^2~:~ \wt F(s,y)\geq 0\}.
\]
We now make an important observation on the function $\wt F$.

\begin{proposition}
\label{prop:ellipsoid}
    The function $\wt F$ is negative definite, i.e. $S$ is a centered ellipsoid (which degenerates to a line-segment if $\alpha=0$).
    Moreover the major axis of $S$ has ``positive'' slope in $[0,\pi/2]$.
\end{proposition}

\begin{proof}
    We begin with negative definiteness, assuming $\alpha>0$ as the pure case is easy. Note the first three terms of $\wt F(s,y)$ are a quadratic of type $(1,1)$, while the last term subtracts a positive-semidefinite quadratic. Hence $\wt F$ cannot be positive definite, so it suffices to prove it has positive discriminant.
    After some easy computation, the discriminant's positivity reduces to proving that
    \[
    (2\xi'' - \alpha^2)((\xi')^2 +\alpha^2) \stackrel{?}{>} 2\xi'' (\xi')^2.
    \]
    Dividing by $\alpha^2$, this reduces to showing $\alpha^2 > 2\xi''-(\xi')^2$. This in turn rearranges to $\xi''>\xi'$ which is clear.

    The latter assertion holds as if $(s,y)\in S$ with $sy\leq 0$ then also $(s,-y),(-s,y)\in S$.
\end{proof}

In the next proposition, we use the notation $GS(\xi)=\plim_{N\to\infty} \max_{\bx\in\cS_N} H_N(\bx)/N$ for the ground state energy.

\begin{proposition}
\label{prop:kac-rice-E-infinity}
    For any $\ups>0$, and for $\iota$ small enough depending on $(\xi,\ups)$:
    \begin{equation}
    \label{eq:kac-rice-E-infinity-minus}
    \lim_{N\to\infty}
    \frac{1}{N}
    \log
    \bbE
    \big|\Crt_{N}\big(
    (-\infty,E_{\infty}^- -\ups)~;~(\sqrt{2\xi''(1)}-\iota,\infty)
    \big)\big|
    <-c(\xi,\ups)<0.
    \end{equation}
    Furthermore, either $GS(\xi)\leq E_{\infty}^+$ or
    \begin{equation}
    \label{eq:kac-rice-E-infinity-plus}
    \lim_{N\to\infty}
    \frac{1}{N}
    \log
    \bbE
    \big|\Crt_{N}\big(
    (E_{\infty}^+ +\ups,\infty)~;~(-\infty,\sqrt{2\xi''(1)}+\iota)
    \big)\big|
    <-c(\xi,\ups)<0.
    \end{equation}
\end{proposition}

\begin{proof}
    We assume $\alpha>0$ as otherwise the statement is easy. Note that $\sqrt{2\xi''(1)}$ in the statement corresponds to $s=\sqrt{2}$.
    By inspection, the line $s=\sqrt{2}$ intersects the boundary of $S$ at $(E_{\infty}^-,\sqrt 2), (E_{\infty}^+,\sqrt{2})$.
    The remainder of the proof is an elementary two-dimensional geometry argument depicted in Figure~\ref{fig:ellipsoid}.\footnote{In many cases $E_{\infty}^-\geq 0$, but the picture is drawn to emphasize that we do not require it. The red region is non-empty only when $GS(\xi)\leq E_{\infty}^+$.
    }
    \begin{figure}[h!]
    \centering
    \includegraphics[width=0.8\textwidth]{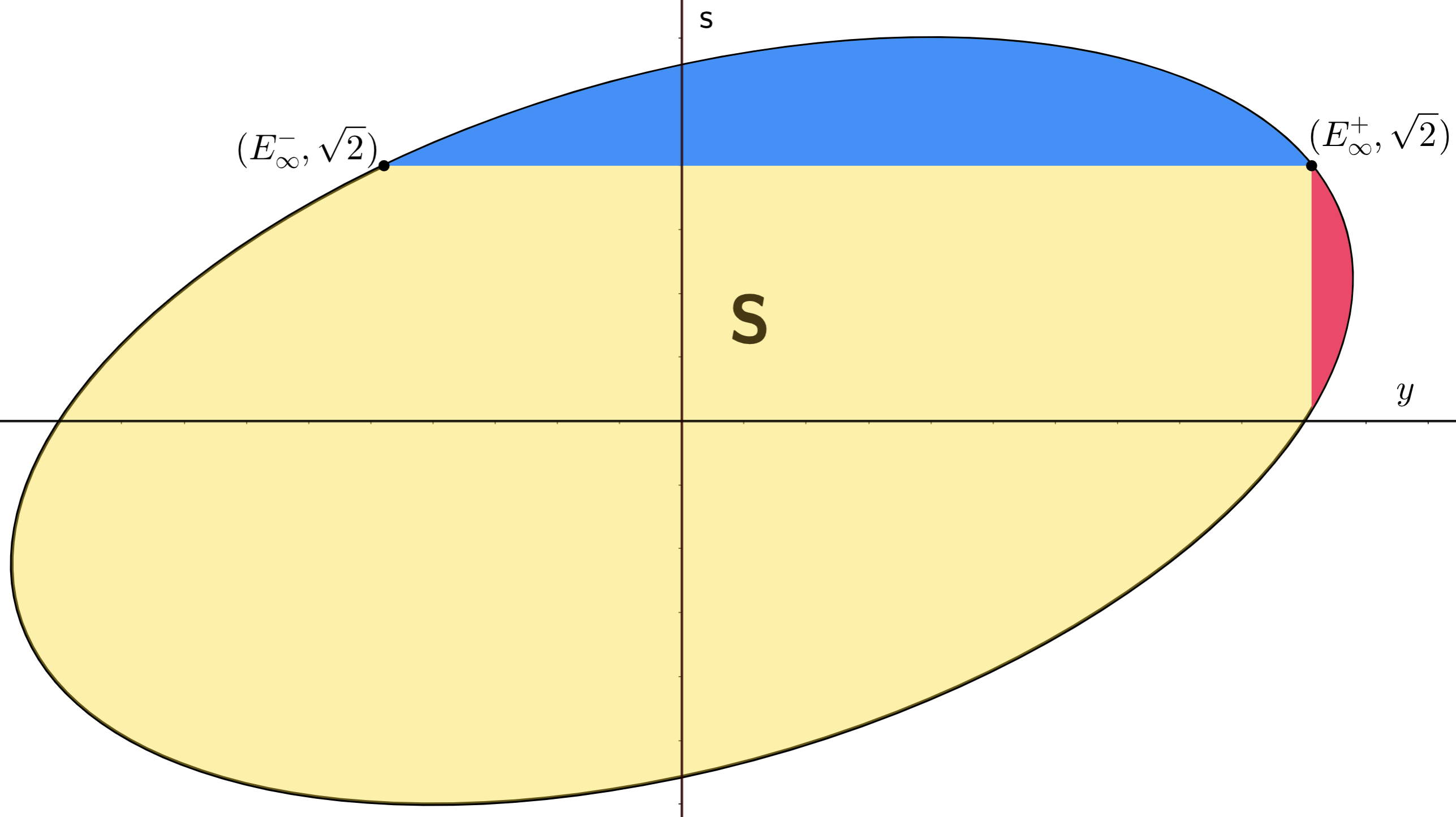}
    \caption{A diagram of the ellipsoid $S$, used in the proof of Proposition~\ref{prop:kac-rice-E-infinity}. If the tangent line to $S$ at $(E_{\infty}^+,\sqrt 2)$ has positive slope, then the red region is empty and we conclude \eqref{eq:kac-rice-E-infinity-plus}. If not, all local maxima correspond to the blue region, hence have energy at most $E_{\infty}^+ + o_N(1)$. This implies that $GS(\xi)\leq E_{\infty}^+$.}
    \label{fig:ellipsoid}
    \end{figure}
    First, because the major axis of $S$ has positive slope in $[0,\pi/2]$, the point in $S$ with minimal $y$ coordinate must have negative $s$ coordinate, while the point in $S$ with maximal $s$ coordinate must have positive $y$ coordinate. For all points on the boundary of $S$ between these two points, in particular $(E_{\infty}^-,\sqrt 2)$, the tangent line to $S$ has positive slope. Therefore
    \[
    S\cap \Big((-\infty,E_{\infty}^-)\times (\sqrt{2},\infty)\Big)=\emptyset,
    \]
    which easily implies the first claim.

    For the second claim, suppose in the first case that the the tangent line to $(E_{\infty}^+,\sqrt 2)$ has slope in $[0,\pi/2]$. Then the result follows similarly to the first part of the proof.
    However as shown in the diagram, it may be that this tangent slope is strictly negative, in $(\pi/2,\pi)$. In this case, we observe (see e.g. Proposition~\ref{prop:radial-deriv-eigenvalues} below) that with probability at least $1-e^{-cN}$, all local maxima of $H_N$ have $s\geq \sqrt{2}-o_N(1)$.
    And if the tangent slope at $(E_{\infty}^+,\sqrt 2)$ is negative,
    \[
    S\cap \Big( \bbR\times (\sqrt 2,\infty)\Big) \subseteq (-\infty,E_{\infty}^+)\times \bbR.
    \]
    Since the global maximum of $H_N$ is a local maximum, we conclude $GS(\xi)\leq E_{\infty}^+$, completing the proof.
\end{proof}

The following fact was used above.

\begin{proposition}[{\cite[Lemma 3]{subag2018following}}]
\label{prop:radial-deriv-eigenvalues}
    For any $\eps>0$, there exists $c,\delta$ such that
    \[
    \bbP\lt[\sup_{\bx\in\cS_N}
    \big|\blambda_{\lfloor \delta N\rfloor}(\nabla^2_{\sph} H_N(\bx))-\sqrt{2\xi''(1)} \nabla_{\rd}H_N(\bx) \big|\leq \eps
    \rt]
    \geq 1-e^{-cN}.
    \]
\end{proposition}

Proposition~\ref{prop:radial-deriv-eigenvalues} also justifies the following definitions. The last is motivated in part by \cite{franz2021surfing}, which suggests that optimization algorithms for general mean-field disordered systems ought to get stuck in $\eps$-marginal local maxima.

\begin{definition}
\label{def:approx-local-max}
We say the $\eps$-critical point $\bx\in\cS_N$ is:
\begin{itemize}
    \item An \textbf{$\eps$-approximate local maximum} if $\nabla_{\rd} H_N(\bx)\geq \sqrt{2\xi''(1)}-\eps$.
    \item An \textbf{$\eps$-approximate local non-maximum} if $\nabla_{\rd} H_N(\bx)\leq \sqrt{2\xi''(1)}+\eps$.
    \item An \textbf{$\eps$-marginal local maximum} if both preceding estimates hold: $|\nabla_{\rd} H_N(\bx)-\sqrt{2\xi''(1)}|\leq \eps$.
\end{itemize}
\end{definition}

We now use Lemma~\ref{lem:approx-to-exact-computation} to control the energy levels at which such $\eps$-critical points can exist.
(One could also directly apply Theorem~\ref{thm:approx-crits-from-annealed}, but this leads to some notational burden.)

\begin{corollary}
\label{cor:zero-eigenvalue-ground-states}
    Fix any $\ups>0$. For sufficiently small $\eps$, with probability $1-e^{-cN}$ all $\eps$-marginal local maxima satisfy
    \begin{equation}
    \label{eq:zero-eigenvalue-ground-states}
    E_{\infty}^- -\ups
    \leq H_N(\bx)/N
    \leq
    E_{\infty}^+ +\ups.
    \end{equation}
    In fact the lower bound holds for all $\eps$-approximate local maxima, while the upper bound holds for all $\eps$-approximate local non-maxima.
\end{corollary}

\begin{proof}
    We proceed in two similar cases, first showing the left-hand side of \eqref{eq:zero-eigenvalue-ground-states} for $\eps$-approximate local maxima.

    \paragraph*{Case $1$: Lower Bound}
    Let $K_N^{\max}(\eps,\ups)$ consist of those $H_N\in K_N(\eps)$ for which there exists an $\eps$-approximate local maximum $\bx\in\cS_N$ with
    \[
    H_N(\bx)/N \leq E_{\infty}^- -\ups.
    \]
    We apply Lemma~\ref{lem:approx-to-exact-computation} as in the proof of Proposition~\ref{prop:approx-crits}.
    We find that for some $\iota=o_{\eps}(1)$,
    \begin{equation}
    \label{eq:same-resampling}
    \bbE
    \big|\,
    \Crt_{N}\big(
    (-\infty,E_{\infty}^- -\ups/2)~;~
    (\sqrt{2\xi''(1)}-\iota,\infty)
    \big)
    \,\big|
    \geq
    e^{-o_{\eps}(N)}
    \cdot
    \bbP[H_N\in K_N^{\max}(\eps,\ups)].
    \end{equation}
    Again using Proposition~\ref{prop:KN-eps-high-prob}, it suffices to show the left-hand side above is at most $e^{-c(\xi,\ups)N}$ for $\eps$ sufficiently small, which is the statement of \eqref{eq:kac-rice-E-infinity-minus}. This proves the left-hand inequality of \eqref{eq:zero-eigenvalue-ground-states} in the claimed sense.

    \paragraph*{Case $2$: Upper Bound}
    Let $K_N^{\nonmax}(\eps,\ups)$ consist of those $H_N\in K_N(\eps)$ for which some $\eps$-local non-maximum $\bx\in\cS_N$ satisfies
    \[
    H_N(\bx)/N \geq E_{\infty}^+ +\ups.
    \]
    Again using Lemma~\ref{lem:approx-to-exact-computation}, we find
    \begin{equation}
    \label{eq:same-resampling-2}
    \bbE
    \big|
    \Crt_{N}\big(
    (E_{\infty}^+ +\ups/2,\infty)~;~
    (-\infty,\sqrt{2\xi''(1)}+\iota)
    \big)
    \big|
    \geq
    e^{-o_{\eps}(N)}
    \cdot
    \bbP[H_N\in K_N^{\nonmax}(\eps,\ups)].
    \end{equation}
    In the case that \eqref{eq:kac-rice-E-infinity-plus} holds, the proof is as in the first case.
    If $GS(\xi)\leq E_{\infty}^+$, then the upper bound holds trivially.
\end{proof}

We briefly summarize the applications of Corollary~\ref{cor:zero-eigenvalue-ground-states} in our concurrent works. See the individual papers for more detail. As partially mentioned in Remark~\ref{rem:positive-complexity}, our work \cite{huang2023optimization} uses approximate message passing to construct $\eps$-marginal local maxima at the algorithmic threshold energy $\ALG$ for any strictly sub-solvable $\xi$.
For $r=1$, Corollary~\ref{cor:zero-eigenvalue-ground-states} thus implies that $\ALG\in [E_{\infty}^-,E_{\infty}^+]$, generalizing the fact that $\ALG=E_{\infty}$ for pure models.
%
%
Separately, \cite{sellke2023threshold} proves that spherical Langevin dynamics at large inverse temperature $\beta$ rapidly climbs to and stays above the energy of the lowest lying $\eps$-approximate local maximum, up to error $o_{\beta\to\infty}(1)$.
Corollary~\ref{cor:zero-eigenvalue-ground-states} thus gives $E_{\infty}^- - o_{\beta}(1)$ as an explicit energy lower bound.

Finally we present a consequence for approximate critical points of finite index. Recall from \cite{auffinger2013complexity} the positive thresholds $(E_k)_{k\geq 0}$, defined so that $E_k$ is the larger of two zeros for the index $k$ critical point complexity function $\theta_{k,\xi}$ (where we have implicitly negated $H_N$ to make all energies positive).
\cite{auffinger2013complexity} deduced from Markov's inequality that $H_N$ has no index $k$ critical points at energies strictly above $E_k$, and we extend this to approximate critical points.
Note that we consider positive energy values, so our signs are switched.

\begin{corollary}
\label{eq:strongly-concave}
    Let $E>E_k$ for fixed $k$ and let $\eps(E,k)$ be sufficiently small.
    Then with probability $1-e^{-cN}$, all $\eps$-approximate critical points $\bx$ with $H_N(\bx)/N\geq E$ have index at most $k$.
    Furthermore if $k=1$, then all such $\bx$ are within distance $\eta\sqrt{N}$ from a local maximum where $\eta=\eta(E,k,\eps)\to 0$ as $\eps\to 0$ for each fixed $E,k$.
\end{corollary}

\begin{proof}
    It follows from the smoothness and monotonicity properties for $\theta_{k,\xi}$ in \cite[Proposition 1]{auffinger2013complexity} that for some $\delta(E,k)>0$, the expected number of critical points $\bx\in\cS_N$ with $\blambda_k(\nabla_{\sph}^2 H_N(\bx))\geq -2\delta$ is at most $e^{-cN}$.
    Another similar application of Lemma~\ref{lem:approx-to-exact-computation} for sufficiently small $\eps$ implies with probability $1-e^{c' N}$, all $\eps$-approximate critical points $\bx\in\cS_N$ satisfy $\blambda_k(\nabla_{\sph}^2 H_N(\bx))\leq -\delta$.
    This completes the proof.
     For the second claim, note that when $k=1$, we have just showed $\supp\big(\nabla_{\sph}^2 H_N(\bx)\big)\cap [-\delta,\delta]=\emptyset$.
     Hence for $\eps$ small compared to $\delta$, Proposition~\ref{prop:newton} yields the result.
\end{proof}

\subsection*{Acknowledgements}

Thanks to Paul Bourgade, Ji\v{r}\'{i} \v{C}ern\'{y}, Sinho Chewi, L{\'a}szl{\'o} Erd{\H o}s, Mufan Li, David Benjamin Lim, Benjamin McKenna, Eliran Subag, and the anonymous referees for helpful discussions, suggestions and comments.
B.H. was supported by an NSF Graduate Research Fellowship, a Siebel scholarship, NSF awards CCF-1940205 and DMS-1940092, and NSF-Simons collaboration grant DMS-2031883.

\appendix

\footnotesize
\bibliographystyle{alpha}
\bibliography{bib}

\normalsize

\section{Properties of Solutions to the Vector Dyson Equation}
\label{app:dyson}

In this appendix, we establish properties of the vector Dyson equation \eqref{eq:dyson-equation-1} that we use in the paper.
It will be useful to rename $\vv + z \vlambda$ to $\vv$ and allow $\vv$ to vary in all of $\obbH^r$.
Thus we study the equation
\begin{equation}
    \label{eq:dyson-equation-complex}
    v_s = -\fr{\lambda_s}{u_s} - \sum_{s'\in \sS} \xi''_{s,s'} u_s.
\end{equation}
The right-hand side of \eqref{eq:dyson-equation-complex} is a function of $\vu$, which we will denote $\vv(\vu)$.
There will be no confusion with the notations $\vv(\vDelta)$, $\vu(\vDelta)$ (defined in Corollary~\ref{cor:only-maximizers} and equation \eqref{eq:def-u-Delta}), which do not appear in this appendix.
\begin{lemma}[{\cite[Section 3]{helton2007operator}}]
    \label{lem:dyson-solution-unique}
    For any $\vv \in \bbH^r$, there exists a unique solution $\vu = \vu(\vv) \in \bbH^r$ to \eqref{eq:dyson-equation-complex}.
\end{lemma}
As a result, for $z \in \bbH$ the solution $\vu(z;\vv)$ to the Dyson equation \eqref{eq:dyson-equation} is given by
\begin{equation}
    \label{eq:u-identification}
    \vu(z;\vv) = \vu(\vv + z \vlambda).
\end{equation}
The first result of this Appendix establishes continuity of $\vu(\cdot)$ in $\vv$.
This extends the $1/3$-H{\"o}lder continuity in $z$ proved in \cite{ajanki2017singularities}, which corresponds for us to varying $\vv$ along certain $1$-dimensional subspaces.
See also \cite[Section 10]{alt2020dyson} for continuity properties in $\xi$.
We note that, importantly, these works treat extremely general models with a continuum of ``species'' parametrized by a probability measure. By contrast we will not hesitate to use the assumption that $r$ is finite (e.g. in \eqref{eq:V-S-def}).

\begin{theorem}
    \label{thm:continuity}
    The solution $\vu(\vv)$ to \eqref{eq:dyson-equation-complex} identified by Lemma~\ref{lem:dyson-solution-unique} extends to a $1/3$-H\"older continuous function $\vu : \obbH^r \to \obbH^r$.
\end{theorem}
Thus the identification \eqref{eq:u-identification} remains true as $z$ tends to the real line. I.e. as $z \to \gamma \in \bbR$, the limit $\vu(\gamma;\vv)$ of $\vu(z;\vv)$ (well-defined by Proposition~\ref{prop:MDE-basic}) equals $\vu(\vv + \gamma \vlambda)$.

The proof of Theorem~\ref{thm:continuity} consists of two steps. We first show $\Im\vu$ is $1/3$-H\"older, and then we extend this to $\Re \vu$.
The first step is handled similarly to \cite{ajanki2017singularities}, though care must be used to handle $\vv$ with imaginary parts of very different sizes.
In the second step, we start by deducing via Stieltjes transforms that $\vu$ can be extended in a H{\"o}lder continuous way within certain $1$-dimensional subspaces.
To glue these extensions together, we employ results from harmonic analysis on the boundary behavior of harmonic functions.
In particular the consistency of these extensions on different lines intersecting at a common point $\vv\in\bbR^r$ follows from the existence of non-tangential limits.

\begin{remark}
    By continuity of $\vv(\cdot)$, for any $\vv \in \obbH^r$ the point $\vu = \vu(\vv)$ defined by Theorem~\ref{thm:continuity} is a solution to the Dyson equation $\vv = \vv(\vu)$.
    However, for $\vv \in \obbH^r \setminus \bbH^r$, this solution is not necessarily the unique preimage of $\vv$ in $\obbH^r$.
\end{remark}
For $\vu \in \obbH^r$, recall from \eqref{eq:def-oM} the definitions:
\begin{align*}
    M(\vu) &= \diag\lt(\fr{\lambda_s}{u_s^2}\rt) - \xi'', &
    \oM(\vu) &= \diag\lt(\fr{\lambda_s}{|u_s|^2}\rt) - \xi''.
\end{align*}
\begin{lemma}
    \label{lem:vu-derivative}
    For any $\vv \in \obbH^r$ such that $M(\vu(\vv))$ is invertible, $\vu(\cdot)$ is differentiable at $\vv$ and $\nabla \vu(\vv) = M(\vu(\vv))^{-1}$.
\end{lemma}

Our next result determines the images $\vu(\obbH^r)$ and $\vu(\bbR^r)$.
In particular, this characterizes which $(\vv,\vu)$ pairs solving \eqref{eq:dyson-equation-complex} for $\vv\in\obbH^r$ are genuine solutions obtainable as limits of solutions with $\vv\in\bbH^r$.

\begin{theorem}
    \label{thm:feasible-region}
    Let $\vu^* \in \obbH^r$.
    \begin{enumerate}[label=(\alph*)]
        \item \label{itm:feasible-interior} There exists $\vv \in \obbH^r$ such that $\vu^* = \vu(\vv)$ if and only if $\oM(\vu^*) \succeq 0$ and $\oM(\vu^*) \Im(\vu^*) \succeq \vzero$.
        \item \label{itm:feasible-boundary} There exists $\vv \in \bbR^r$ such that $\vu^* = \vu(\vv)$ if and only if one of the following conditions holds.
        \begin{enumerate}[label=(\roman*)]
            \item \label{itm:feasible-boundary-real} $\vu^* \in \bbR^r$ and $M(\vu^*) \succeq 0$.
            \item \label{itm:feasible-boundary-imag} $\vu^* \in \bbH^r$, $\oM(\vu^*) \succeq 0$, and $\oM(\vu^*) \Im(\vu^*) = 0$.
        \end{enumerate}
    \end{enumerate}
\end{theorem}
\begin{corollary}
    \label{cor:feasible-region-M-inv}
    If $\vu^* \in \bbH^r$ and there exists $\vv \in \obbH^r$ such that $\vu^* = \vu(\vv)$, then $M(\vu^*)$ is invertible.
\end{corollary}

Next in Proposition~\ref{prop:edge-vs-cusp}, we show that singularity of $M(\vu(\vv))$ (which by the previous corollary requires $\vv \in \bbR^r$) corresponds to $0$ being an edge/cusp of associated spectral measures, and give a precise description of each case.
For any $\vchi \in \bbR_{>0}^r$ with $\|\vchi\|_1 = 1$, the restriction of $\vu(\cdot)$ onto the line $\vv + z \vchi$, $z\in \obbH$ is a rescaled Stieltjes transform of a suitable random matrix.
Indeed, this restriction solves
\[
    v_s + \chi_s z = - \fr{\lambda_s}{u_s(\vv + z \vchi)} - \sum_{s'\in\sS} \xi''_{s,s'} u_{s'}(\vv + z \vchi).
\]
We set
\begin{align}
    \label{eq:dyson-change-of-variables}
    \tilde \xi''_{s,s'} &= \fr{\lambda_s\lambda_{s'} \xi''_{s,s'}}{\chi_s\chi_{s'}}, &
    \tilde m_s(z;\vv) &= \fr{\chi_s}{\lambda_s} u_s(\vv + z \vchi), &
    \tilde x_s &= \fr{v_s \sqrt{\lambda_s}}{\chi_s},
\end{align}
which we note match $\xi''_{s,s'}$, $m_s$, $x_s$ when $\vchi = \vlambda$.
Then the Dyson equation rearranges to
\[
    \fr{\tilde x_s}{\sqrt{\lambda_s}} + z
    = - \fr{1}{\tilde m_s(z;\vv)} - \sum_{s'\in \sS} \fr{\tilde \xi''_{s,s'}}{\lambda_s} \tilde m_{s'}(z;\vv).
\]
Comparing with \eqref{eq:dyson-equation}, we find that $\tilde \vm(z;\vv)$ is the limiting Stieltjes transform of the random matrix
\begin{equation}
    \label{eq:def-M_N-tilted}
    \wt M_N(\tilde \vx) = \wt \bW - \diag(\Lambda^{-1/2} \tilde \vx \diamond \bone_{\cT}),
\end{equation}
where $\wt \bW$ has law \eqref{eq:tangential-hessian-law} but with $\tilde \xi''$ in place of $\xi''$.
Denote the associated limiting spectral measure (cf. \eqref{eq:mu-def})
by
\begin{equation}
\label{eq:wt-mu-def}
    \wt\mu_{\vchi}(\vv)\equiv \mu_{\vchi}( \tilde \vx)
\end{equation}

We recall from \cite[Theorem 2.6]{ajanki2017singularities} that $\wt\mu_{\vchi}(\vv)$ is supported on a finite union of intervals, which is the closure of $\{\gamma \in \bbR : \vu(\vv + \gamma \vchi) \in \bbH^r\}$, with \emph{edges} at the boundary of its support and finitely many \emph{cusps} within the support at which $\vu(\vv + \gamma \vchi)\in\bbR^r$.
We say $0$ is a \emph{left edge} of the support of $\mu_{\vchi}(\tilde \vx)$ if it is an edge and $(0,c)\subseteq \supp \,\mu_{\vchi}(\tilde \vx)$ for small enough $c>0$. A \emph{right edge} is defined similarly.
For $\gamma \in \bbR$ and $\vchi$ as above, we set $\vu^{\gamma}_{\vchi} = \vu(\vv + \gamma \vchi)$.

\begin{proposition}
\label{prop:edge-vs-cusp}
    Suppose $\vv \in \bbR^r$, $\vu = \vu(\vv) \in \bbR^r$, and $M(\vu)$ is singular. Fix $\vchi \in \bbR_{>0}^r$ with $\|\vchi\|_1 = 1$.
    Then $0$ is an edge or cusp of $\wt\mu_{\vchi}(\vv)$.
    In more detail, there exists $\gamma_0>0$ such that for each $\Delta \in \{\pm 1\}$, one of the following holds.
    \begin{enumerate}[label=(\roman*)]
        \item \label{itm:M-singular-real}
        For all $\gamma \in (0,\gamma_0]$, $\vu^{\gamma\Delta}_{\vchi} \in \bbR^r$ and $M(\vu^{\gamma\Delta}_{\vchi}) \succ 0$.
        \item \label{itm:M-singular-imag}
        For all $\gamma \in (0,\gamma_0]$, $\vu^{\gamma\Delta}_{\vchi} \in \bbH^r$.
    \end{enumerate}
    Moreover case \ref{itm:M-singular-imag} holds for at least one $\Delta \in \{\pm 1\}$, and:
    \begin{enumerate}[label=(\alph*)]
    \item
    \label{it:right-edge}
    If case~\ref{itm:M-singular-real} holds for $\Delta=1$, then $0$ is a right edge of $\wt\mu_{\vchi}(\vv)$.
    \item
    \label{it:left-edge}
    If case~\ref{itm:M-singular-real} holds for $\Delta=-1$, then $0$ is a left edge of $\wt\mu_{\vchi}(\vv)$.
    \item
    \label{it:cusp}
    If case~\ref{itm:M-singular-imag} holds for both $\Delta\in \{\pm 1\}$, then $0$ is a cusp of $\wt\mu_{\vchi}(\vv)$.
    \end{enumerate}
    Finally, which of \ref{it:right-edge},\ref{it:left-edge},\ref{it:cusp} occurs for a given $\vv$ does not depend on $\vchi$.
\end{proposition}

\begin{remark}
    The $\vchi$-independence of being an edge or cusp is consistent with Figures~\ref{subfig:two-species} and \ref{subfig:two-species-weird}.
    Recall that in these plots, $\vu(0;\vv)$ is real in the four regions outside the blue boundary and nonreal in the region inside it.
    An edge corresponds to a point on the blue boundary where a positive-slope line in direction $\vchi$ through $\vv$ crosses from a real region to a nonreal region.
    A cusp (two in each plot) corresponds to a point where such a line remains in the nonreal region on either side of $\vv$.
    In both pictures, this property is independent of the slope $\vchi$.
\end{remark}

Our final result computes the annealed exponential growth rate of the determinant of the deformed Gaussian band matrix $M_N(\vx)$ defined in \eqref{eq:def-M_N}.
Namely we give an explicit formula for $\Psi(\vx)$, defined in \eqref{eq:Psi-def}.
Recall from the proof of Proposition~\ref{prop:annealed-kac-rice} that this equals
\[
    \lim_{N\to\infty} \fr{1}{N} \log \bbE |\det M_N(\vx)|.
\]
Recall that $\la \va, \vb \ra = \sum_{i=1}^r a_ib_i$ denotes a bilinear form rather than a complex inner product, even when $\va,\vb$ are complex vectors.
\begin{theorem}
    \label{thm:logdet}
    Let $\vx \in \bbR^r$ and $\vv = \Lambda^{1/2} \vx$.
    Then
    \[
        \Psi(\vx) = \fr12 \Re(\la \vu(\vv), \xi'' \vu(\vv) \ra)
        - \sum_{s\in \sS} \lambda_s \log |u_s(\vv)|.
    \]
\end{theorem}

\subsection{Preliminaries}

\begin{lemma}
    \label{lem:psd-rotate}
    Suppose $A \in \bbR^{r\times r}$ is diagonally signed, $A \succeq 0$, and $A' \in \bbC^{r\times r}$ satisfies:
    \begin{equation}
    \label{eq:psd-rotate}
    \begin{aligned}
    |A'_{s,s}|&\geq A_{s,s},\quad\forall s\in \sS;
    \\
    |A'_{s,s'}|&\leq |A_{s,s'}|,\quad\forall s\neq s'\in \sS.
    \end{aligned}
    \end{equation}
    Then:
    \begin{enumerate}[label=(\alph*)]
        \item \label{itm:psd-inv-positive-entries} If $A$ is invertible, $A^{-1}$ has only positive entries.
        \item \label{itm:inv-domination} If $A$ is invertible, then so is $A'$ and $\|(A')^{-1}\|_{\op}\leq \|A^{-1}\|_{\op}$.
        \item \label{itm:inv-by-strict} If at least one inequality in \eqref{eq:psd-rotate} holds strictly, then $A'$ is invertible.
    \end{enumerate}
\end{lemma}

\begin{proof}
    Suppose $A$ is invertible.
    Then $A \succ 0$, so $A_{s,s} > 0$ for all $s\in \sS$.
    Let $D = \diag(A)^{1/2}$, so $A = D (I-B) D$ for some $B \in \bbR^{r\times r}$ with zero diagonal and positive entries off the diagonal, and with $I-B \succ 0$.
    Let $t = \blambda_{\min}(I-B) \in (0,1)$, so $(1-t)I - B \succeq 0$.
    By Lemma~\ref{lem:psd-sign-flip} (applied to $(1-t)I - B$), $(1-t)I + B \succeq 0$.
    Thus $(1-t)I \succeq B \succeq -(1-t)I$.
    So,
    \[
        A^{-1} = \lt(D(I-B)D\rt)^{-1} = D^{-1} (I + B + B^2 + \cdots) D^{-1},
    \]
    as the geometric series converges.
    Since $D$ and $B$ have positive entries, part \ref{itm:psd-inv-positive-entries} follows.

    There exists diagonal $\tD \in \bbC^{r\times r}$ such that $D^2 \tD^2 = \diag(A')$, and \eqref{eq:psd-rotate} implies $|\tD_{s,s}| \ge 1$ for all $s\in \sS$.
    Then $A' = D\tD (I - \tB) \tD D$.
    Note that for all $s,s'\in \sS$, $|\tB_{s,s'}| \le B_{s,s'}$, and therefore for all $k\ge 1$, $|(\tB^k)_{s,s'}| \le (B^k)_{s,s'}$.
    Thus
    \[
        (A')^{-1} = D^{-1} \tD^{-1} (I +  \tB + \tB^2 + \cdots) \tD^{-1} D^{-1},
    \]
    as the geometric series converges.
    For any $\vx \in \bbC^r$, consider $\vy \in \bbR^r$ defined by $y_s = |x_s|$.
    Then it is clear that for all $s\in \sS$, $|((A')^{-1}\vx)_s| \le (A^{-1}\vy)_s$, so $\norm{(A')^{-1}}_\op \le \norm{A^{-1}}_\op$.
    This proves part \ref{itm:inv-domination}.

    Finally consider the setting of part \ref{itm:inv-by-strict}, where $A\succeq 0$ is not necessarily invertible and at least one inequality in \eqref{eq:psd-rotate} is strict.
    Let $\tA \in \bbR^{r\times r}$ be the diagonally signed matrix with $\tA_{s,s} = |A'_{s,s}|$ and $\tA_{s,s'} = -|A'_{s,s'}|$ for $s\neq s'$.
    Let $\vw$ be the minimal (unit) eigenvector of $\tA$, which by Lemma~\ref{lem:diagonally-signed-hs23} has all positive entries.
    Then
    \[
        \blambda_{\min}(\tA)
        = \la \tA \vw, \vw \ra
        > \la A \vw, \vw \ra
        \ge \blambda_{\min}(A) \ge 0,
    \]
    so $\tA \succ 0$.
    Thus $\tA$ is invertible, and by part \ref{itm:inv-domination} (with $\tA$ for $A$) so is $A'$.
    This proves part \ref{itm:inv-by-strict}.
\end{proof}
The following part of the proof of Theorem~\ref{thm:feasible-region} will be used repeatedly, so we prove it first.
It is related to \cite[Lemma 4.3]{ajanki2017singularities} (namely the operator $F$ appearing there is similar to $\oM$).

\begin{lemma}
    \label{lem:oM-psd}
    For any $\vv \in \bbH^r$, with $\vu = \vu(\vv)$ we have $\oM(\vu) \succ 0$.
\end{lemma}
\begin{proof}
    Taking imaginary parts of \eqref{eq:dyson-equation-complex} yields
    \begin{equation}
        \label{eq:dyson-imaginary}
        \fr{\lambda_s}{|u_s|^2} \Im(u_s)
        - \sum_{s'\in\sS} \xi''_{s,s'} \Im(u_{s'})
        = \Im(v_s).
    \end{equation}
    Since $\Im(v_s) > 0$, we have $\oM(\vu) \Im(\vu) \succ 0$.
    This implies $\oM(\vu) \succ 0$ by Lemma~\ref{lem:diagonally-signed-hs23}.
\end{proof}
\begin{lemma}
    \label{lem:oM-psd-to-M-invertible}
    If $\vu \in \bbH^r$ and $\oM(\vu) \succeq 0$, then $M(\vu)$ is invertible.
\end{lemma}
\begin{proof}
    Since $\oM(\vu) \succeq 0$, $\lambda_s/|u_s|^2 > \xi''_{s,s}$.
    So, for any $s\in \sS$,
    \[
        \lt|\fr{\lambda_s}{u_s^2} -\xi''_{s,s}\rt| > \fr{\lambda_s}{|u_s|^2} -\xi''_{s,s},
    \]
    where the inequality is strict because $\vu\in\bbH^r$.
    Taking $(A,A')=(\oM(\vu),M(\vu))$ in Lemma~\ref{lem:psd-rotate}\ref{itm:inv-by-strict} yields the claim.
\end{proof}

\begin{corollary}
    \label{cor:M-invertible}
    For any $\vv\in\bbH^r$, with $\vu=\vu(\vv)$, $M(\vu)$ is invertible.
\end{corollary}

\begin{proof}
    Follows from Lemmas~\ref{lem:oM-psd} and \ref{lem:oM-psd-to-M-invertible}.
\end{proof}

\begin{lemma}
    \label{lem:u-bounded}
    There exists $C_0>0$ depending on $(\vlambda,\xi'')$ such that for all $\vv\in\bbH^r$, $\|\vu(\vv)\|_{\infty}\leq C_0$.
\end{lemma}

\begin{proof}
    Let $\vu = \vu(\vv)$.
    By Lemma~\ref{lem:oM-psd}, $\oM(\vu) \succ 0$, so $\oM(\vu)_{s,s} > 0$.
    Thus $|u_s| \le \sqrt{\lambda_s / \xi''_{s,s}}$.
\end{proof}

\subsection{Joint Continuity of the Vector Dyson Equation}

In this subsection, we let $C_0$ be as in Lemma~\ref{lem:u-bounded}, let $C_1$ be a sufficiently large constant depending on $(\vlambda,\xi'',C_0)$, and similarly take large $C_2$ depending on $(\vlambda,\xi'',C_0,C_1)$ and $C_3$ large depending on $(\vlambda,\xi'',C_0,C_1,C_2)$.
Given any $S\subseteq\sS$, define
\begin{equation}
\label{eq:V-S-def}
    V_S
    =
    \{
    \vv\in\bbH^r
    ~:~
    |v_s|\leq C_1~\forall s\in S
    \text{ and }
    |v_s|\geq C_1~\forall s\notin S
    \}
    \subseteq\bbH^r.
\end{equation}
We will show H\"older continuity of the restriction of $\Im (\vu)$ to each set $V_S$.

\begin{lemma}
\label{lem:non-degen-per-S}
    For each $S\subseteq \sS$, the restriction of $\vu:\bbH^r\to \bbH^r$ to $V_S$ satisfies
    $\|\Im (\vu)|_{V_S}\|_{C^{1/3}}<\infty$.
\end{lemma}

Lemma~\ref{lem:non-degen-per-S} readily implies that $\|\Im(\vu)\|_{C^{1/3}}<\infty$ holds on all of $\bbH^r$, thus establishing ``half of'' Theorem~\ref{thm:continuity}. Namely given $\vv,\vv'\in \bbH^r$, along the path $(\vv+t(\vv'-\vv))_{t\in [0,1]}$ the $s$-th coordinate's norm switches between $[0,C_1]$ and $[C_1,\infty)$ at most twice. Hence $\Im(\vu(\vv)) -\Im(\vu(\vv'))$ can be bounded by applying Lemma~\ref{lem:non-degen-per-S} at most $2r+1$ times along this path.

We prove Lemma~\ref{lem:non-degen-per-S} after establishing some helpful intermediate results.

\begin{lemma}
    \label{lem:non-degen-u}
    For $C_1$ as described above (i.e. sufficient large depending on $(\vlambda,\xi'',C_0)$), the following holds.
    \begin{enumerate}[label=(\alph*)]
        \item \label{itm:non-degen-u} For all $s\in S$, $|u_s| \ge \lambda_s/2C_1$.
        \item \label{itm:degen-u} For all $s\not\in S$, $|u_s| \le 2\lambda_s / C_1$.
    \end{enumerate}
\end{lemma}
\begin{proof}
    Equation \eqref{eq:dyson-equation-complex} implies
    \begin{equation}
        \label{eq:triangle-ineq}
        |v_s| - \sum_{s'\in \sS} \xi''_{s,s'} |u_{s'}|
        \le
        \fr{\lambda_s}{|u_s|}
        \le
        |v_s| + \sum_{s'\in \sS} \xi''_{s,s'} |u_{s'}|.
    \end{equation}
    In light of Lemma~\ref{lem:u-bounded}, we have
    \[
        \sum_{s'\in \sS} \xi''_{s,s'} |u_{s'}| \le C_1/2
    \]
    for suitably large $C_1$ depending only on $(\vlambda,\xi'')$.
    For $s\in \sS$, the right inequality of \eqref{eq:triangle-ineq} implies $\lambda_s/|u_s| \le 2C_1$, which implies part \ref{itm:non-degen-u}.
    For $s\not\in \sS$, the left inequality implies $\lambda_s/|u_s| \ge C_1/2$, which implies part \ref{itm:degen-u}.
\end{proof}

\begin{lemma}
\label{lem:imaginary-one-scale}
    For $s,s'\in S$,
    \[
    C_2^{-1}\Im(u_s)\leq \Im(u_{s'})\leq C_2 \Im(u_s).
    \]
\end{lemma}
\begin{proof}
    Taking imaginary parts of \eqref{eq:dyson-equation-complex} yields \eqref{eq:dyson-imaginary}.
    In light of Lemma~\ref{lem:non-degen-u}\ref{itm:non-degen-u}, this implies
    \[
        \fr{4C_1^2}{\lambda_s} \Im(u_s)
        \ge \fr{\lambda_s}{|u_s|^2} \Im(u_s)
        \ge \xi''_{s,s'} \Im(u_{s'}).
    \]
    Since such an inequality holds for all $s,s'\in S$ the conclusion follows.
\end{proof}

\begin{lemma}
    \label{lem:u-derivative}
    The function $\vu$ is differentiable on $\bbH^r$ with Jacobian $\nabla \vu(\vv) = M(\vu(\vv))^{-1}$ (which is invertible by Corollary~\ref{cor:M-invertible}).
\end{lemma}
\begin{proof}
    Let $\vv \in \bbH^r$ and $\vu = \vu(\vv) \in \bbH^r$.
    Then $\vv(\vu) = \vv$.
    The function $\vv(\cdot)$ is clearly continuous, so it maps an open neighborhood $\cN \subset \bbH^r$ of $\vu$ into $\vv(\cN) \subset \bbH^r$.
    By Lemma~\ref{lem:dyson-solution-unique}, this is a bijective map with inverse $\vu(\cdot)$.
    Moreover $\vv(\cdot)$ is differentiable, with Jacobian
    \begin{equation}
        \label{eq:vv-deriv}
        \nabla \vv(\vu) = M(\vu),
    \end{equation}
    and this is invertible by Corollary~\ref{cor:M-invertible}.
    The result follows by the inverse function theorem.
\end{proof}

\begin{lemma}
\label{lem:AEK-adaptation}
    For each $s_*\in S$, and distinct $\vv,\wt\vv\in V_S$,
    \begin{equation}
    \label{eq:AEK-adaptation}
    \frac{\|\Im (u_{s_*}(\vv)) - \Im (u_{s_*}(\wt\vv))\|_2}
    {\|\vv-\wt\vv\|_2^{1/3}}
    \leq
    C_3.
    \end{equation}
\end{lemma}

\begin{proof}
    Write $\vu = \vu(\vv)$.
    By Lemma~\ref{lem:u-derivative}, $\nabla \vu(\vv) = M^{-1}(\vu)$.
    We show that for $\vv \in V_S$,
    \begin{equation}
    \label{eq:AEK-bound-op}
    \|M^{-1}(\vu)\|_{\op}
    \leq
    C_3
    \Im(u_{s_*})^{-2}.
    \end{equation}
    To deduce \eqref{eq:AEK-adaptation} from this, first note that for any smooth path $\gamma:[0,1]\to V_S$, \eqref{eq:AEK-bound-op} implies
    \[
    \lt|\frac{\de}{\de t}
    \big[\Im\big(u_{s_*}(\gamma(t))\big)^3\big]
    \rt|
    \leq
    C_3 \gamma'(t).
    \]
    This implies $\Im(u_{s_*})^3$ is Lipschitz on $V_S$ because for any $\vv,\wt\vv\in V_S$ there exists $\gamma$ as above with $\big(\gamma(0),\gamma(1)\big)=(\vv,\wt\vv)$ and $\int_0^1 |\gamma'(t)|\de t\leq 10r \|\vv-\wt\vv\|_2$.
    Since $\Im(u_{s_*})$ is uniformly bounded by Lemma~\ref{lem:u-bounded}, the fact that $\Im(u_{s_*})^3$ is Lipschitz immediately yields \eqref{eq:AEK-adaptation}.

    To show \eqref{eq:AEK-bound-op}, with $\eps=C_2 \Im(u_{s_*})>0$, we have $\Im(u_s)\geq \eps$ for all $s\in S$ by Lemma~\ref{lem:imaginary-one-scale}.
    Define the matrix
    \[
    M^{\dagger}_{s,s'}
    =
    \begin{cases}
        |M_{s,s}|=|\lambda_s/u_s^2 - \xi''_{s,s}|, & s=s' \in\sS,
        \\
        -\xi''_{s,s'}, & s\neq s'\in \sS.
    \end{cases}.
    \]
    Thus $M^{\dagger}$ agrees with $M$ off of the diagonal. On the diagonal, we claim that
    \[
        M^{\dagger}_{s,s}
        \geq
        \oM_{s,s}
        +
        \Omega(\eps^2)
        \cdot \one_{s\in S}.
    \]
    This is easy to see geometrically: given $|u_s|$, the entry $M^{\dagger}_{s,s}$ varies on a circle, and its radius is $\sqrt{\lambda_s}/|u_s|\asymp 1$ since $s\in S$, and its distance from the center is also $\xi''_{s,s}\asymp 1$.
    By Lemma~\ref{lem:psd-rotate}, it follows that $M^{\dagger}$ is strictly positive definite since $\oM\succeq 0$.

    We claim that in fact
    \begin{equation}
        \label{eq:increase-from-good-coords}
        M^{\dagger} \succeq \Omega(\eps^2)I_r.
    \end{equation}
    Indeed let $\vy\in \bbR^r$ be a unit vector and let $\vy_S\in\bbR^r$ agree with $\vy$ on coordinates in $S$ and have zero coordinates otherwise. If $\|\vy_S\|_2^2 \geq 1/2$, then
    \[
    \la \vy, M^{\dagger}\vy\ra
    \geq
    \la \vy, \oM \vy\ra
    +
    \Omega(\eps^2 \|\vy_S\|_2^2)
    \geq \Omega(\eps^2).
    \]
    Otherwise, suppose $\|\vy_{S^c}\|_2^2 \geq 1/2$.
    Define
    \begin{align*}
        C'_0 &= \max_{s,s'\in \sS} \xi''_{s,s'}, &
        C'_1 &= \min_{s\in \sS} \lt\{
            \fr{C_1^2}{4\lambda_s} - \xi''_{s,s}\rt
        \}.
    \end{align*}
    Lemma~\ref{lem:non-degen-u}\ref{itm:degen-u} implies $M^{\dagger}_{s,s}\ge C'_1$ for all $s\notin S$ while $M^{\dagger}_{s,s'}\geq -C'_0$ for all $s,s'\in\sS$.
    So
    \[
        \la \vy, M^{\dagger}\vy\ra
        \geq
        C'_1 \|\vy_{S^c}\|_2^2
        -
        rC'_0
        \geq 1
    \]
    if $C_1$ is suitably large.
    Combining cases proves \eqref{eq:increase-from-good-coords} since $\vy$ was an arbitrary unit vector.

    Thus $\norm{(M^{\dagger})^{-1}}_\op \le O(\eps^{-2})$.
    Applying Lemma~\ref{lem:psd-rotate}\ref{itm:inv-domination} with $(A,A')=(M^{\dagger},M)$ shows that $\|M^{-1}\|_{\op}\leq  \|(M^{\dagger})^{-1}\|_{\op}$, establishing \eqref{eq:AEK-bound-op} as desired.
\end{proof}

\begin{proof}[Proof of Lemma~\ref{lem:non-degen-per-S}]
    Let $B(R)\subseteq\bbC$ denote the radius $R$ ball.
    Note that given \emph{any} $(\vv,\vu_S)\in \bbH^r\times B(C_0)^{|S|}$, the complementary vector $\vu_{S^c}=\vu-\vu_S$ may be defined by the equations in \eqref{eq:dyson-equation-complex} for $s\in S^c$.
    Restricting the domain slightly to $\cD_S=(\bbC\backslash B(C_1))^r\times B(C_0)^{|S|}$, this defines a map
    \[
    \varphi_S:\cD_S\to B(C_0)^{|S^c|}.
    \]
    Note that the restriction $M_{S^c}$ of $M$ to coordinates $S^c\times S^c$ satisfies $\|M_{S^c}(\vu)^{-1}\|_{\op}\geq 1$ on the domain of $\varphi_S$ since $C_1$ is large compared to $C_0$.
    It follows that $\|\nabla\varphi_S\|\leq O(1)$ holds everywhere on $\cD_S$.

    Finally just as in the proof of Lemma~\ref{lem:AEK-adaptation}, for any pair of points in $\cD_S$, there is a smooth path $\gamma:[0,1]\to\cD_S$ connecting them with total length at most the Euclidean distance between them.
    Therefore $\varphi_S$ is $O(1)$-Lipschitz on $\cD_S$, and so using Lemma~\ref{lem:AEK-adaptation}, for any $\vv,\wt\vv\in V_S$
    \begin{align*}
        \|\vu(\vv)-\vu(\wt\vv)\|_2
        &\leq
        \|\vu_S(\vv)-\vu_S(\wt\vv)\|_2
        +
        \|\vu_{S^c}(\vv)-\vu_{S^c}(\wt\vv)\|_2
        \\
        &\lesssim
        \|\vu_S(\vv)-\vu_S(\wt\vv)\|_2
        +
        \big(
        \|\vu_S(\vv)-\vu_S(\wt\vv)\|_2
        +
        \|\vv-\wt\vv\|_2
        \big)
        \\
        &\lesssim
        \|\vv-\wt\vv\|_2^{1/3}\cdot (1+\|\vv-\wt\vv\|_2).
    \end{align*}
    Recalling that $\vu$ is uniformly bounded now completes the proof.
\end{proof}

It follows immediately from the preceding result that $\Im\vu$ extends to a $C^{1/3}$ function on $\obbH^r$.
It remains to show the same for $\Re \vu$. Similarly to \cite[Proposition 5.1]{ajanki2017singularities}, along any given $1$-dimensional subspace of $\bbR^r$, $\Re\vu$ can be obtained via the Stieltjes transform of the continuous boundary extension of $\Im \vu$, which automatically inherits the $1/3$-H{\"o}lder continuity of $\Im \vu$.
Since we aim to show continuity in $\vv\in\obbH^r$, these $1$-dimensional Stieltjes transforms must be patched together.
Consistency of the extensions to intersecting lines in $\bbR^r$ will follow from the existence of non-tangential limits as recalled below.

\begin{definition}
    Given $v\in\bbR$ and $\theta\in (0,\pi/2)$, define the cone
    \[
    \Cone_{\theta}(v)=
    \big\{y\in \bbH~:~\arg(y-v)\in [\pi/2-\theta,\pi/2+\theta]\big\}
    \subseteq \bbH.
    \]
    Given $\vv\in\bbR^r$ and $\theta_1,\dots,\theta_r\in (0,\pi/2)$, define the product cone
    \[
    \Cone_{\vtheta}(\vv)=\prod_{s=1}^r \Cone_{\theta_s}(v_s)\subseteq\bbH^r.
    \]
\end{definition}

\begin{proposition}[{\cite[Special Case of Theorem 3.24 in Chapter 2]{stein1971introduction}}]
\label{prop:non-tangential-limit}
    Let $\vu:\bbH^r\to\bbH^{r}$ be a bounded harmonic function. Then for almost every $\vv\in\bbR^r$, the following \textbf{non-tangential limit} exists and is uniformly bounded for all $\vtheta\in (0,\pi/2)^r$:
    \begin{equation}
    \label{eq:non-tangential-limit}
    \vu^{(\bbR)}(\vv)\equiv \lim_{\substack{\vy\to \vv\\ \vy\in \Cone_{\vtheta}(\vv)}}
    \vu(\vy).
    \end{equation}
    We call those $\vv\in\bbR^r$ with this property \textbf{regular points} for $\vu$.
\end{proposition}

It is well-known that a bounded harmonic function on $\obbH$ can be recovered as the Poisson integral of its boundary values, see e.g. \cite[Theorem 11.30(b)]{rudin1987real}.
Below we give an extension to higher dimensions which suffices for our purposes.
Define for $y\in\bbH,\vy\in\bbH^r$ the univariate and multivariate Poisson kernels:
\begin{equation}
\label{eq:poisson-kernel}
    K(y)
    =
    \frac{1}{\pi\|y\|^2},
    \quad\quad
    K(\vy)
    =
    \prod_{s=1}^r
    \vK(y_s).
\end{equation}
Note that $K(y)$ is a probability density on each shift $\bbR+it$ for $t>0$. We view $K(y)$ as a point mass at $y$ if $y\in\bbR$, and $\vK$ as the corresponding product measure for $\vy\in\obbH^r$.

\begin{proposition}
\label{prop:poisson-integral-representation}
    Suppose $\vu:\bbR^r\to\bbR$ as defined in Proposition~\ref{prop:non-tangential-limit} agrees with a continuous bounded function $\vu^{(\bbR)}:\bbR^r\to\bbR$ almost everywhere. Then $\vu$ extends to a bounded continuous function on $\obbH^r$ agreeing with $\vu^{(\bbR)}$ on $\bbR^r$ and admitting the Poisson integral representation
    \begin{equation}
    \label{eq:poisson-representation}
    \vu(\vy)
    =
    \int_{\bbR^r}
    \vK(\vy-\vv) \vu^{(\bbR)}(\vv)
    ~\de \vv,
    \quad\vy\in\obbH^r.
    \end{equation}
\end{proposition}

\begin{proof}
    First, the above definition of the Poisson integral agrees (in the case that $\vu^{(\bbR)}$ is uniformly bounded) with that of \cite[Chapter 2 page 67]{stein1971introduction} as an iterated application of univariate Poisson integrals.
    By $r$-fold application of \cite[Chapter 2 Theorem 2.1(b)]{stein1971introduction}, it follows that the right-hand side of \eqref{eq:poisson-representation} is continuous and bounded on $\obbH^r$.
    Call this right-hand side $\wt\vu$.
    Since each probability measure $K(y_s-v_s)\de v_s$ converges weakly to a point mass at $\Re y_s$ as $\Im y_s\downarrow 0$, it follows that $\wt\vu$ is continuous on $\obbH^r$.

    It remains to show that $\vu$ and $\wt\vu$ agree on $\bbH^r$.
    Hence fix $\vy\in\bbH^r$.
    Since both functions are harmonic and bounded on $\bbH^r$, we have the upward-shifted Poisson integral representations for $\eps\in (0,\min_s \Im y_s)$:
    \begin{equation}
    \label{eq:shifted-poisson}
    \begin{aligned}
        \vu(\vy)
        &=
        \int_{\bbR^r}
        \vK\big(\vy-\vv-\eps i \vone\big)
        \vu(\vv+ \eps i \vone)
        ~\de\vv,
        \\
        \wt\vu(\vy)
        &=
        \int_{\bbR^r}
        \vK\big(\vy-\vv-\eps i \vone\big)
        \wt\vu(\vv+ \eps i \vone)
        ~\de\vv.
    \end{aligned}
    \end{equation}
    The functions $\vu(\vv+ \eps i \vone)$ and $\wt\vu(\vv+ \eps i \vone)$ on $\bbR^r$ are uniformly bounded and converge almost everywhere to the same limit $\vu^{(\bbR)}$ as $\eps\downarrow 0$.
    Moreover for each fixed $\vy\in\bbH^r$, the kernel densities $\vK(\vy-\vv-\eps i \vone)$ are all probability measures, and they converge in total variation to $\vK(\vy-\vv)$ as $\eps\downarrow 0$.
    It follows that the $\eps\downarrow 0$ limits of the right-hand sides in \eqref{eq:shifted-poisson} agree. Hence the left-hand sides also agree as desired.
\end{proof}

\begin{proof}[Proof of Theorem~\ref{thm:continuity}]
    It follows by Lemma~\ref{lem:non-degen-per-S} and the following discussion that $\vv \mapsto \Im(\vu(\vv))$ is a uniformly $1/3$-H\"older continuous function on $\bbH^r$.
    We use \cite[Theorem 3.5]{garnett2007bounded} which states that if a bounded holomorphic function $\varphi:\bbH\to\bbH$ satisfies
    \begin{equation}
    \label{eq:stieltjes-condition}
    \lim_{A\to\infty} -iA\varphi(iA)=W>0,
    \end{equation}
    then $\varphi$ is the Stieltjes transform of a positive measurable density on $\bbR$ with total integral $W$, which is given as as almost-everywhere limit of functions $\Im(\varphi(x+i\eta))$ as $\eta\downarrow 0$.
    We consider for each $\vy\in [1/2,2]^r$ and $\vz_*\in\bbR^r$ and $s\in\sS$ the function
    \[
    \varphi_{\vy,s}(z)
    =
    \vu_s(\vz_*+z\vy),\quad z\in\bbH.
    \]
    Then it is easy to see that the condition~\eqref{eq:stieltjes-condition} holds with $W=\lambda_s/y_s$.

    Consider now the lines $\ell(\vz_*,\vy)=\{\vz_*+z\vy\}_{z\in\bbR}$ for $\vy\in [1/2,2]^r$ and $\vz_*\in\bbH^r$.
    For each $\ell(\vz_*,\vy)$, taking the Stieltjes transform of $\varphi_{\vy,s}$ gives a function $u_s(\cdot\,;\vz_*,\vy):\ell(\vz_*,\vy)\to\obbH$.
    Recall that Stieltjes transforms increase $C^{1/3}$ norms by at most a constant factor (see e.g. \cite[Section 22]{muskhelishvili2008singular}).
    Since $\Im \vu$ is uniformly $1/3$-H{\"o}lder, it follows that each
    $\vu(\cdot\,;\vz_*,\vy)$ and in particular its real part has uniformly bounded $C^{1/3}$ norm on its corresponding domain $\ell(\vz_*,\vy)$.

    Next whenever $\vz_*+z\vy\in\bbR^r$ is regular, we have
    \begin{equation}
    \label{eq:consistent-def}
    \vu(z;\vz_*,\vy)
    =
    \lim_{\eps\downarrow 0}
    \varphi_{\vy,s}(z+i\eps)
    =
    \vu^{(\bbR)}(\vz_*+z\vy).
    .
    \end{equation}
    (The first equality holds by continuity properties of ordinary Stieltjes transforms, and the second by definition of a regular point.)
    Thus let $\vv,\vv'\in\bbR^r$ be regular points for $\vu$, and let $\wt\vv$ be another regular point such that with $\|\vv-\vv'\|_{\infty}=M$, we have $\wt v_s-\vv_s\in [3M,4M]$ for each $s$.
    (Such $\wt\vv$ exists by Proposition~\ref{prop:non-tangential-limit}.)
    Then $\frac{v_s-\wt v_s}{v_{s'}-\wt v_{s'}}\in [1/2,2]$ for each $s,s'\in\sS$, which means there is some $\ell(\vz_*,\vy)$ passing through $(\vv,\wt\vv)$, and similarly a $\ell(\vz'_*,\vy')$ passing through $(\vv',\wt\vv)$.
    Using $1/3$-H{\"o}lder continuity on these lines together with \eqref{eq:consistent-def} in the latter step, we thus obtain:
    \begin{align*}
    \|\vu^{(\bbR)}(\vv)-\vu^{(\bbR)}(\vv')\|
    &\leq
    \|\vu^{(\bbR)}(\vv)-\vu^{(\bbR)}(\wt\vv)\|
    +
    \|\vu^{(\bbR)}(\vv')-\vu^{(\bbR)}(\wt\vv)\|
    \\
    &\leq
    O(M^{1/3})
    =
    O(\|\vv-\vv'\|_{\infty}).
    \end{align*}
    Hence the restriction of $\vu^{(\bbR)}$ to regular $\vv\in\bbR^r$ is uniformly $C^{1/3}$.
    By Proposition~\ref{prop:non-tangential-limit}, it follows that $\vu^{(\bbR)}$ admits a bounded continuous extension to all of $\bbR^r$.
    By Proposition~\ref{prop:poisson-integral-representation}, $\vu$ extends to a bounded continuous function on $\obbH^r$.

    Finally we show $\vu:\obbH^r\to\obbH^r$ as just defined is uniformly $C^{1/3}$ on its full domain.
    Fix $\vv,\vv'\in\obbH^r$.
    For each $s\in\sS$, let $\bB_s(t)$ be a standard complex Brownian motion. Define the processes
    \[
    \bV_s(t)=v_s + \Im(v_s)\bB_s(t\wedge \tau_s)
    \text{ and }
    \bV'_s(t)=v'_s+\Im(v_s')\bB_s(t\wedge \tau_s)
    \]
    where $\tau_s$ denotes the first time that $\Im \bB_s(t)=-1$;  thus $\Im(\bV_s(\tau_s))=\Im(\bV_s'(\tau_s))=0$.
    Note $\tau\equiv \max_s \tau_s<\infty$ almost surely.

    Since $\vu$ is bounded and holomorphic, it follows that $\vu(\vec\bV(t))$ and $\vu(\vec\bV'(t))$ are both $\bbC$-valued martingales.
    (This also follows directly from \eqref{eq:poisson-kernel}.)
    Using the triangle inequality followed by $C^{1/3}$-boundedness of $\vu$ on $\bbR^r$ gives
    \begin{align*}
        |\vu(\vv)-\vu(\vv')|
        &\leq
        \bbE |\vu(\vec\bV(\tau)))-\vu(\vec\bV'(\tau))|
        \\
        &\lesssim
        \bbE \sum_{s=1}^r
        |\Re(\bV_s(\tau))-\Re(\bV_s'(\tau))|^{1/3}
        \\
        &\lesssim
         \bbE
         \Big[
        \sum_{s=1}^r
        |\Re(v_s)-\Re(v'_s)|^{1/3}
        +
        \sum_{s=1}^r
        |\Im(v_s)-\Im(v'_s)|^{1/3}
        |\Re\bB_s(\tau_s)|^{1/3}
        \Big]
        .
    \end{align*}
    The law of $\Re\bB_s(\tau_s)$ is well known to be a standard symmetric Cauchy random variable (and does not depend on $\vv$ or $\vv'$). In particular it has finite $1/3$ moment.
    Hence we find that $\|\vu(\vv)-\vu(\vv')\|\leq O\big(\|\vv-\vv'\|_{\infty}^{1/3}\big)$, completing the proof.
\end{proof}



\begin{proof}[Proof of Lemma~\ref{lem:vu-derivative}]
    Consider a sequence of functions $\vu^\eps(\vv) = \vu(\vv + \eps i \vone)$.
    By Lemma~\ref{lem:u-derivative}, $\nabla \vu^{\eps}(\vv) = M(\vu^{\eps}(\vv))^{-1}$.
    By Theorem~\ref{thm:continuity} and invertibility of $M(\vu(\vv))$, as $\eps \downarrow 0$ both $\vu^{\eps}(\cdot)$ and $M(\vu^{\eps}(\cdot))^{-1}$ converge locally uniformly to $\vu(\cdot)$ and $M(\vu(\cdot))^{-1}$.
    The result now follows by e.g. \cite[Theorem 7.17]{rudin1976principles}, which states that if a family of functions and their derivatives each converge uniformly, then the derivative of the limiting function is the limit of the derivatives.
\end{proof}

\subsection{Solution Space of the Vector Dyson Equation}

In this subsection we prove Theorem~\ref{thm:feasible-region} and Proposition~\ref{prop:edge-vs-cusp}.
We first establish Theorem~\ref{thm:feasible-region} in the setting $\vu^*,\vv \in \bbH^r$.

\begin{lemma}
    \label{lem:feasible-interior}
    Let $\vu^* \in \bbH^r$.
    There exists $\vv \in \bbH^r$ such that $\vu^* = \vu(\vv)$ if and only if $\oM(\vu^*) \succ 0$ and $\oM(\vu^*) \Im(\vu^*) \succ \vzero$.
\end{lemma}
\begin{remark}
    \label{rmk:redundancy}
    In the setting of this lemma $\Im(\vu^*) \succ \vzero$, so $\oM(\vu^*) \Im(\vu^*) \succ \vzero$ implies $\oM(\vu^*) \succ 0$ by Lemma~\ref{lem:diagonally-signed-hs23}.
    We have written the lemma in this form for consistency with Theorem~\ref{thm:feasible-region}.
\end{remark}

\begin{proof}
    If $\vu^* = \vu(\vv)$ for some $\vv \in \bbH^r$, Lemma~\ref{lem:oM-psd} shows $\oM(\vu^*) \succ 0$, and the proof of that lemma shows $\oM(\vu^*) \Im(\vu^*) \succ \vzero$.
    Conversely, suppose $\oM(\vu^*) \Im(\vu^*) \succ \vzero$.
    Define $\vv = \vv(\vu^*)$.
    Then
    \[
        v_s
        = -\fr{\lambda_s}{|u^*_s|^2} \bar u^*_s -\sum_{s'\in\sS} \xi''_{s,s'} u^*_{s'},
    \]
    so $\Im(\vv) = \oM(\vu^*) \Im(\vu^*) \succ \vzero$.
    Thus $\vv \in \bbH^r$.
    By Lemma~\ref{lem:dyson-solution-unique}, $\vu^* = \vu(\vv)$.
\end{proof}

\begin{proof}[Proof of Theorem~\ref{thm:feasible-region}]
    We first prove the forward directions of both parts.
    For part \ref{itm:feasible-interior}, suppose $\vu^* = \vu(\vv)$ for some $\vv \in \obbH^r$.
    By Lemma~\ref{lem:feasible-interior} and continuity of $\vu$, $\vu^*$ lies in the closure of the set defined by $\oM(\vu^*) \succ 0$ and $\oM(\vu^*) \Im(\vu^*) \succ \vzero$, which implies the conclusion.
    For part \ref{itm:feasible-boundary}, suppose $\vu^* = \vu(\vv)$ for some $\vv \in \bbR^r$.
    Taking imaginary parts of \eqref{eq:dyson-equation-complex} yields \eqref{eq:dyson-imaginary}, which implies $\oM(\vu^*) \Im(\vu^*) = \vzero$.
    Since $\oM(\vu^*)$ is diagonally signed this implies $\Im(\vu^*) = \vzero$ or $\Im(\vu^*) \succ \vzero$, i.e. $\vu^* \in \bbR^r$ or $\vu^* \in \bbH^r$.
    By part \ref{itm:feasible-interior} we also have $\oM(\vu^*) \succeq 0$.
    If $\vu^* \in \bbR^r$, then $M(\vu^*) = \oM(\vu^*) \succeq 0$, so conclusion \ref{itm:feasible-boundary-real} holds.
    If $\vu^* \in \bbH^r$, conclusion \ref{itm:feasible-boundary-imag} holds.

    We turn to the converses, beginning with part \ref{itm:feasible-interior}.
    Suppose $\oM(\vu^*) \succeq 0$ and $\oM(\vu^*) \Im(\vu^*) \succeq \vzero$.
    Let $\vv^* = \vv(\vu^*)$; we will show that $\vu^* = \vu(\vv^*)$.

    Similarly to above, $\oM(\vu^*) \Im(\vu^*) \succeq \vzero$ implies $\vu^* \in \bbR^r$ or $\vu^* \in \bbH^r$.
    Suppose first that $\vu^* \in \bbR^r$, and further assume $\oM(\vu^*) \succ 0$.
    Recall from \eqref{eq:vv-deriv} that $\vv$ has Jacobian $\nabla \vv(\vu) = M(\vu)$.
    Because $\vu^* \in \bbR^r$, we have $M(\vu^*) = \oM(\vu^*) \succ 0$.
    So, $\nabla \vv$ is invertible in a neighborhood of $\vu^*$.
    By the inverse function theorem, there is a local inverse $\vv^{-1}$ of $\vv$ satisfying
    \[
        \nabla \vv^{-1} (\vv^*) = \oM(\vu^*)^{-1}.
    \]
    By Lemma~\ref{lem:psd-rotate}\ref{itm:psd-inv-positive-entries}, $\oM(\vu^*)^{-1}$ has all positive entries.
    For small $\eps>0$ define $\vv^\eps = \vv^* + i \eps \vone$ and note that
    \[
        \fr{\de}{\de \eps} \vv^{-1}(\vv^\eps) \big|_{\eps=0} = i \oM(\vu^*)^{-1} \vone \in \bbH^r.
    \]
    Define $\vu^\eps = \vv^{-1}(\vv^\eps)$.
    Then, for small $\eps>0$ we have $\vu^\eps, \vv^\eps \in \bbH^r$ and $\vv^\eps = \vv(\vu^\eps)$.
    By Lemma~\ref{lem:dyson-solution-unique}, $\vu^\eps = \vu(\vv^\eps)$.
    Taking $\eps \to 0$, continuity of $\vu$ implies $\vu^* = \vu(\vv^*)$.

    Next suppose $\vu^* \in \bbR^r$ and $\oM(\vu^*) \succeq 0$ is singular.
    For small $\eps>0$ define $\vu^{(\eps)} = (1-\eps) \vu^*$ and $\vv^{(\eps)} = \vv(\vu^{(\eps)})$.
    Since $\oM(\vu^{(\eps)}) \succ \oM(\vu^*) \succeq 0$, we have just shown $\vu^{(\eps)} = \vu(\vv^{(\eps)})$.
    Continuity of $\vu$ implies $\vu^* = \vu(\vv^*)$.

    Finally, suppose $\vu^* \in \bbH^r$.
    As above, for small $\eps>0$, $\oM(\vu^{(\eps)}) \succ 0$.
    Moreover, for any $s\in \sS$,
    \[
        \fr{\lambda_s}{|u^{(\eps)}_s|^2} \Im(u^{(\eps)}_s)
        - \sum_{s'\in \sS} \xi''_{s,s'} \Im(u^{(\eps)}_{s'})
        > \fr{\lambda_s}{|u^*_s|^2} \Im(u^*)
        - \sum_{s'\in \sS} \xi''_{s,s'} \Im(u^*_{s'})
        \ge 0,
    \]
    so $\oM(\vu^{(\eps)}) \Im(\vu^{(\eps)}) \succ \vzero$.
    Lemma~\ref{lem:feasible-interior} implies $\vu^{(\eps)} = \vu(\vv^{(\eps)})$.
    Continuity of $\vu$ implies $\vu^* = \vu(\vv^*)$.
    This proves the converse to part \ref{itm:feasible-interior}.

    Finally, we turn to the converse to part \ref{itm:feasible-boundary}.
    If either of \ref{itm:feasible-boundary-real}, \ref{itm:feasible-boundary-imag} holds, then $\oM(\vu^*) \succeq 0$ and $\oM(\vu^*) \Im(\vu^*) \succeq \vzero$.
    We have just shown that $\vu^* = \vu(\vv^*)$, where $\vv^* = \vv(\vu^*)$.
    We easily verify that under \ref{itm:feasible-boundary-real} or \ref{itm:feasible-boundary-imag}, $\vv^* \in \bbR^r$, completing the proof.
\end{proof}

\begin{proof}[Proof of Corollary~\ref{cor:feasible-region-M-inv}]
    Theorem~\ref{thm:feasible-region} implies $\oM(\vu) \succeq 0$, so the result follows from Lemma~\ref{lem:oM-psd-to-M-invertible}.
\end{proof}

\subsubsection{Proof of Proposition~\ref{prop:edge-vs-cusp}}

Recall the notation \eqref{eq:wt-mu-def}, which will be frequently used below.
We begin with the first assertion of Proposition~\ref{prop:edge-vs-cusp}, namely that singularity of $M$ always corresponds to either an edge or cusp.

\begin{lemma}
\label{lem:M-singular-edge-cusp}
    Fix $\vchi\in\bbR^r_{>0}$ with $\|\vchi\|_1=1$.
    Then $M(\vu(\vv))$ is singular if and only if $0$ is either an edge or cusp for $\wt\mu_{\vchi}(\vv)$.
\end{lemma}

\begin{proof}
    First if $0$ is an edge or cusp, then \cite[Theorem 2.6]{ajanki2017singularities} makes it clear that $\vu$ is not locally Lipschitz in $\vv$, hence the contrapositive of Lemma~\ref{lem:vu-derivative} shows $M(\vu)$ is singular.

    In the other direction, we have seen that singularity of $M(\vu)$, for $\vu=\vu(\vv)$, implies $\vu\in\bbR^r$.
    Moreover since $M(\vu)$ is diagonally signed, its singularity implies by Lemma~\ref{lem:diagonally-signed-hs23} that there exists $\vw\succ \vzero\in\bbR^r$ with $M(\vu)\vw=0$.
    Suppose for sake of contradiction that $0\notin \supp\, \wt\mu_{\vchi}(\vv)$.
    Then by Lemma~\ref{lem:Dyson-weakly-continuous}, and the Stieltjes transform definition of $\vu$, it follows that $\gamma\mapsto \vu^{\gamma}_{\vchi}(\vv)$ is Lipschitz for $\gamma$ in a neighborhood of $0$ (since $\log(x)$ is Lipschitz away from $0$).
    A first order Taylor expansion of \eqref{eq:dyson-equation-complex} (similarly to Lemma~\ref{lem:u-derivative}) then implies
    \begin{align*}
    \lim_{\gamma\downarrow 0}
    M(\vu)\big(\vu_{\vchi}^{\gamma}-\vu\big)/\gamma
    =
    \lim_{\gamma\downarrow 0}
    \big(\vv+\gamma\vchi-\vv\big)/\gamma
    =
    \vchi
    .
    \end{align*}
    However the left-hand side above is orthogonal to $\vw$ for all $\gamma\neq 0$, while $\la \vw,\vchi\ra>0$ since both have strictly positive entries.
    This is a contradiction and completes the proof.
\end{proof}

Given absolutely continuous $\mu\in\cP(\bbR)$ and $q\in (0,1)$, let
\[
    \blambda_{(q)}(\mu)=\sup\{\lambda\in\bbR~:~\mu((-\infty,\lambda]\leq q) \}
\]
be its $q$-th quantile.

\begin{lemma}
\label{lem:quantile-continuous}
    Suppose $\diag(\vchi^{-1})(\vv-\vv')\in [a,b]^r$.
    Then for all $q\in (0,1)$, we have
    \[
        \blambda_{(q)}(\wt\mu_{\vchi}(\vv'))-\blambda_{(q)}(
        \wt\mu_{\vchi}(\vv))\in [a,b].
    \]
\end{lemma}

\begin{proof}
    Immediate by the Weyl inequalities applied to the eigenvalues of the $N\times N$ random matrices \eqref{eq:def-M_N-tilted} whose spectra tend to $\wt\mu_{\vchi}(\vv'),\wt\mu_{\vchi}(\vv)$.
    Indeed in this context, the shift from $\vv$ to $\vv'$ is equivalent to adding a diagonal matrix with all entries in $[a,b]$.
\end{proof}

\begin{lemma}
\label{lem:edge-characterization}
    Given any $\vu(\vv)\in \bbR^r$ with $M(\vu)$ singular, the following are equivalent:
    \begin{enumerate}[label=(\arabic*)]
        \item
        \label{it:all-small-shift-all-good}
        For all $\eps_0>0$ sufficiently small and all $\veps\in (0,\eps_0]^r$,
        \[
        \vu(\vv-\veps)\in \bbR^r.
        \]
        \item
        \label{it:exists-small-shift-all-good}
        For all $\eps_0>0$ sufficiently small, there exists $\veps\in (0,\eps_0]^r$ such that
        \[
        \vu(\vv-\veps)\in \bbR^r.
        \]
        \item
        \label{it:all-left-edge}
        For all $\vchi\in \bbR^r_{>0}$, the density $\wt\mu_{\vchi}(\vv)$ has $0$ as a left edge.
        \item
        \label{it:exists-left-edge}
        There exists $\vchi\in \bbR^r_{>0}$ such that the density $\wt\mu_{\vchi}(\vv)$ has $0$ as a left edge.
    \end{enumerate}
\end{lemma}

\begin{proof}
    We will show that point~\ref{it:exists-left-edge} implies point~\ref{it:all-small-shift-all-good}, and that point~\ref{it:exists-small-shift-all-good} implies point~\ref{it:all-left-edge}, which suffices.
    In the first direction, if $0$ is a left edge for some $\vchi$, then Lemma~\ref{lem:quantile-continuous} immediately implies point~\ref{it:all-small-shift-all-good}.

    In the other direction, suppose point~\ref{it:all-left-edge} does not hold.
    Singularity of $M(\vu)$ implies via Lemma~\ref{lem:M-singular-edge-cusp} that $0$ is an edge or cusp of $\wt\mu_{\vchi}(\vv)$.
    Hence if $0$ is not a left edge for some $\vchi$, it must be a right edge or a cusp. In either case, Lemma~\ref{lem:quantile-continuous} then implies that point~\ref{it:exists-small-shift-all-good} does not hold.
    This completes the proof.
\end{proof}

\begin{proof}[Proof of Proposition~\ref{prop:edge-vs-cusp}]
    Note that parts \ref{it:all-small-shift-all-good} and \ref{it:exists-small-shift-all-good} of Lemma~\ref{lem:edge-characterization} are both independent of $\vchi$. It follows that $0$ being a left edge, right edge, or cusp for $\wt\mu_{\vchi}(\vv)$ are also each independent of $\vchi$.
    Moreover the left and analogous right edge characterizations in parts \ref{it:all-small-shift-all-good}, \ref{it:exists-small-shift-all-good} of Lemma~\ref{lem:edge-characterization} directly correspond to case~\ref{itm:M-singular-real} of Proposition~\ref{prop:edge-vs-cusp}. This correspondence implies the result.
\end{proof}

\subsection{Exponential Growth Rate of Random Determinant}

This subsection is devoted to the proof of Theorem~\ref{thm:logdet}.
We adopt the same notation as in Subsection~\ref{subsec:stationarity}, setting $\oPsi(\vv) = \Psi(\vx)$ where $\vv = \Lambda^{1/2}\vx \in \bbR^r$.

\begin{lemma}
    \label{lem:oPsi-derivative}
    $\oPsi : \bbR^r \to \bbR$ is continuously differentiable, with $\nabla \oPsi(\vv) = -\Re(\vu(\vv))$.
\end{lemma}
The following non-rigorous calculation, which we carefully justify below, yields this formula.
Due to the identification \eqref{eq:u-identification}, we may freely switch between the notations $\vu(\,\cdot\,;\,\cdot\,)$ and $\vu(\cdot)$ in what follows.
\begin{equation}
\label{eq:non-rigorous}
\begin{aligned}
    \fr{\de}{\de v_s} \oPsi(\vv)
    &= \fr{1}{\pi} \int_\bbR \log|\gamma| \sum_{s'\in \sS} \lambda_{s'} \Im \lt(\fr{\de u_{s'}(\gamma; \vv)}{\de v_s}\rt) ~\de \gamma \\
    &= \fr{1}{\pi} \int_\bbR \log|\gamma| \sum_{s'\in \sS} \lambda_{s'} \Im \lt(\fr{\de u_{s'}(\vv + \gamma \vlambda)}{\de v_s}\rt) ~\de \gamma \\
    &\stackrel{(\ast)}{=} \fr{1}{\pi} \int_\bbR \log|\gamma| \sum_{s'\in \sS} \lambda_{s'} \Im \lt(\fr{\de u_s(\vv + \gamma \vlambda)}{\de v_{s'}}\rt) ~\de \gamma \\
    &= \fr{1}{\pi} \int_\bbR \log|\gamma| \Im \lt(\fr{\de u_s(\vv + \gamma \vlambda)}{\de \gamma}\rt) ~\de \gamma \\
    &\stackrel{(\diamond)}{=} - \fr{1}{\pi} \int_\bbR \fr{1}{\gamma} \Im \lt(u_s(\vv + \gamma \vlambda)\rt) ~\de \gamma \\
    &\stackrel{(\star)}{=} - \Re\lt(u_s(\vv)\rt)\,.
\end{aligned}
\end{equation}
Step $(\ast)$ uses that $\nabla u(\vv) = M(\vu(\vv))^{-1}$ (recall Lemma~\ref{lem:vu-derivative}) is a symmetric matrix; step $(\diamond)$ integrates by parts and step $(\star)$ is a contour integral.
However this is not a rigorous calculation, primarily because $M(\vu(\vv))^{-1}$ may be singular for $\vv \in \bbR^r$.

To make this calculation rigorous, we first work on the line $\bbR+i\eta$ for $\eta>0$, and then send $\eta\downarrow 0$ (see \cite[Proposition 4.9]{arous2021landscape} for a similar computation).
Recall from Proposition~\ref{prop:MDE-basic} that the probability densities $\mu_s$ solving the Dyson equation are uniformly compactly supported for $\vv\in \bbR^r$ with $\|\vv\|_\infty \leq R$.
It follows that for such $\vv$,
\begin{equation}
    \label{eq:general-MDE-bound}
    \begin{aligned}
    \Im \lt(u_s(\gamma+i\eta,\vv)\rt)
    &\leq
    \fr{C(R,\eta,\xi'',\vlambda)}{1+\gamma^2},
    \\
    \lt| \Im \lt(
    \frac{\de}{\de v_s}
    u_s(\gamma+i\eta,\vv)
    \rt) \rt|
    &\leq
    \fr{C(R,\eta,\xi'',\vlambda)}{1+\gamma^2}.
    \end{aligned}
\end{equation}
Let $\gamma\in\bbR$ and $\eta>0$.
By Lemma~\ref{lem:vu-derivative} and \eqref{eq:u-identification}, $\nabla \vu(\gamma+i\eta;\vv) = M(\vv+(\gamma+i\eta)\vlambda)^{-1}$ (which is invertible by Corollary~\ref{cor:M-invertible}).
Moreover this matrix is symmetric, so
\begin{equation}
    \label{eq:rigorous-eta}
    \begin{aligned}
    &\fr{1}{\pi} \int_\bbR \log|\gamma+i\eta| \sum_{s'\in \sS} \lambda_{s'} \Im \lt(\fr{\de u_{s'}(\gamma + i\eta; \vv)}{\de v_s}\rt) ~\de \gamma \\
    &=
    \fr{1}{\pi} \int_\bbR \log|\gamma+i\eta| \sum_{s'\in \sS} \lambda_{s'} \Im \lt(\fr{\de u_{s}(\gamma + i\eta; \vv)}{\de v_{s'}}\rt) ~\de \gamma
    \\
    &= \fr{1}{\pi}
    \int_\bbR
    \log|\gamma+i\eta|
    ~
    \Im \lt(\fr{\de u_s(\gamma + i \eta ; \vv)}{\de \gamma}\rt)
    ~\de \gamma
    \\
    &\stackrel{(\dagger)}{=} - \fr{1}{\pi} \int_\bbR  \Re\lt(\fr{1}{\gamma+i\eta}\rt) \Im \lt(u_s(\gamma + i\eta; \vv)\rt) ~\de \gamma,
\end{aligned}
\end{equation}
Step $(\dagger)$ is an integration by parts which is valid by \eqref{eq:general-MDE-bound}.
We use the residue theorem to evaluate the last integral. Note that
\[
\int_\bbR  \Re\lt(\fr{1}{\gamma+i\eta}\rt) \Im \lt(u_s(\gamma + i\eta; \vv)\rt) ~\de \gamma
=
\frac{1}{2}\Im \int_\bbR  \lt(\fr{1}{\gamma+i\eta}+\fr{1}{\gamma-i\eta}\rt) u_s(\gamma + i\eta; \vv) ~\de \gamma.
\]
The latter integral can be evaluated by completing the contour via a radius $R$ semicircle in $\bbH$; the contribution of this semicircle decays to $0$ with $R$ since the uniformly compact support of $\mu_s$ implies that $|u_s(z;\vv)|\leq O(1/|z|)$ for large $z$ and fixed $\vv$. Hence applying the residue theorem, we complete the calculation \eqref{eq:rigorous-eta} and obtain:
\begin{equation}
\label{eq:rigorous-eta-2}
\begin{aligned}
\fr{1}{\pi} \int_\bbR \log|\gamma+i\eta| \sum_{s'\in \sS} \lambda_{s'} \Im \lt(\fr{\de u_{s'}(\gamma + i\eta; \vv)}{\de v_s}\rt) ~\de \gamma
&= - \Re\lt(u_s(2i\eta;\vv)\rt)
\\
&= - \Re\lt(u_s(\vv + 2i\eta \vlambda)\rt).
\end{aligned}
\end{equation}

We complete the proof of Lemma~\ref{lem:oPsi-derivative} by taking $\eta\downarrow 0$ in \eqref{eq:rigorous-eta-2}.
The right-hand side poses no issue since $\vu$ is $1/3$-H{\"o}lder continuous by Theorem~\ref{thm:continuity}.
For the left-hand side, we differentiate under the integral sign (which is justified using e.g. compact support of $\mu_s$):
\begin{align*}
    &\fr{1}{\pi} \int_\bbR \log|\gamma+i\eta| \sum_{s'\in \sS} \lambda_{s'} \Im \lt(\fr{\de u_{s'}(\gamma + i\eta; \vv)}{\de v_s}\rt) ~\de \gamma \\
    &=
    \frac{\de}{\de v_s}
    \fr{1}{\pi} \int_\bbR \log|\gamma+i\eta| \sum_{s'\in \sS} \lambda_{s'} \Im \lt( u_{s'}(\gamma + i\eta; \vv)\rt) ~\de \gamma
    .
\end{align*}
We then take the $\eta\downarrow 0$ limit for the latter integrand.

\begin{proposition}
    Locally uniformly over $\vv\in\bbR^r$,
    \[
    \lim_{\eta\downarrow 0}
    \int_\bbR \log|\gamma+i\eta| \sum_{s'\in \sS} \lambda_{s'} \Im \lt( u_{s'}(\gamma + i\eta; \vv)\rt) ~\de \gamma
    =
    \int_\bbR \log|\gamma| \sum_{s'\in \sS} \lambda_{s'} \Im \lt( u_{s'}(\gamma; \vv)\rt) ~\de \gamma
    \]
\end{proposition}

\begin{proof}
    This follows directly by dominated convergence.
    The large $\gamma$ contributions are controlled by \eqref{eq:general-MDE-bound}, while the $\log 0$ singularity is integrable hence causes no issues.
\end{proof}

\begin{proof}[Proof of Lemma~\ref{lem:oPsi-derivative}]
    Define
    \[
    f_s(\eta;\vv)
    =
    \int_\bbR \log|\gamma+i\eta| \sum_{s'\in \sS} \lambda_{s'} \Im \lt( u_{s'}(\gamma + i\eta; \vv)\rt) ~\de \gamma,\quad\eta\geq 0.
    \]
    We have shown above that:
    \begin{enumerate}[label=(\arabic*)]
        \item $\lim_{\eta\downarrow 0} f_s(\eta;\vv)=f_s(0;\vv)$ holds locally uniformly in $\vv$.
        \item For $\eta>0$, we have $\frac{\de}{\de v_s} f_s(\eta;\vv)=-\Re(u_s(\vv + 2i\eta \vlambda))$.
        \item $\Re(u_s(\vv + 2i\eta \vlambda))$ is continuous on $\eta\geq 0$, locally uniformly in $\vv$.
    \end{enumerate}
    Recall from e.g. \cite[Theorem 7.17]{rudin1976principles} that if a family of functions and their derivatives each converge uniformly, then the derivative of the limiting function is the limit of the derivatives.
    This shows $\fr{\de}{\de v_s} \oPsi(\vv) = -\Re(u_s(\vv))$.
    Finally $\vu(\cdot)$ is continuous by Theorem~\ref{thm:continuity}, concluding the proof.
\end{proof}

\begin{proof}[Proof of Theorem~\ref{thm:logdet}]
    We claim that
    \[
        G(\vv)=\oPsi(\vv)
        - \fr12
        \Re\lt(\la \vu(\vv),\xi'' \vu(\vv)\ra\rt)
        +
        \sum_{s\in\sS}\lambda_s \log |u_s(\vv)|
    \]
    vanishes identically on $\vv\in\bbR^r$.
    We first check that $|G(\vv)|\to 0$ as $\min_s |v_s|\to \infty$. In this limit, \eqref{eq:mu-vx-approx} implies that
    \[
    \bbW_{\infty}\Big(
    \omu(\vv),\sum_{s\in\sS}\lambda_s \delta_{-v_s/\lambda_s}
    \Big)
    \]
    is bounded independently of $\vv$. Hence $\oPsi(\vv)-\sum_{s\in\sS} \lambda_s \log(|v_s|/\lambda_s)$ tends to $0$ as $\min_s v_s\to \infty$.
    Furthermore $\vu(\vv) = \vu(0;\vv)\to 0$ in this limit, so \eqref{eq:dyson-equation-complex} implies that in fact $u_s v_s \to -\lambda_s$. Thus $\oPsi(\vv)+\sum_{s\in\sS} \lambda_s \log |u_s|$ indeed tends to $0$ with $\min_s |v_s|$.

    Next let $\cT\subseteq\bbR^r$ denote the set of points at which $\det M(\vu(\vv))=0$. Since $\det M(\vu(\vv))$ is continuous, for any $\vv\notin \cT$, we may differentiate $G(\vv)$ using Lemma~\ref{lem:vu-derivative} to obtain $\nabla G(\vv) = \vzero$.
    Indeed, letting $\vu = \vu(\vv)$, one directly verifies
    \begin{align*}
    &\nabla_{\vv\in\bbR^r}
    \lt(
        -\fr12 \Re\la \vu(\vv),\xi'' \vu(\vv)\ra
        + \sum_{s\in\sS}\lambda_s \log |u_s(\vv)|
    \rt) \\
    &=
    \Re\lt(\nabla_{\vv\in\obbH^r}
    \lt(
        -\fr12 \la \vu(\vv),\xi'' \vu(\vv)\ra
        + \sum_{s\in\sS}\lambda_s \log u_s(\vv)
    \rt)
    \rt)
    \\
    &=
    \Re\lt(
    M(\vu)^{-1}
    \nabla_{\vu\in\obbH^r}
    \lt(
        -\fr12 \la \vu,\xi'' \vu\ra
        +\sum_{s\in\sS}\lambda_s \log u_s
    \rt)
    \rt)
    \\
    &=
    \Re\lt(
        M(\vu)^{-1}
        \lt((\lambda_s / u_s)_{s\in \sS} -\xi'' \vu\rt)
    \rt)
    \\
    &= \Re(\vu)
    \\
    &
    = -\nabla_{\vv \in \bbR^r}\oPsi(\vv).
    \end{align*}
    The gradient subscripts indicate whether we consider the expression as a gradient of a smooth function defined on $\bbR^r$ or of a holomorphic function on $\obbH^r$.

    It follows that $G(\vv)$ is locally constant on $\bbR^r\backslash\cT$.
    Moreover Proposition~\ref{prop:Psi-continuous} implies that $G$ is continuous.
    Finally, Proposition~\ref{prop:MDE-basic}\ref{it:Dyson-continuous-density} implies that each line $\vv(t)=\vv(0)+t\vlambda$ intersects $\cT$ at only finitely many points.
    Combining the above implies $G(\vv)=0$ for all $\vv\in\bbR^r$ as desired.
\end{proof}

\end{document}